\documentclass[11pt]{amsart}
\usepackage{amssymb}
\usepackage{verbatim}
\usepackage{hyperref}
\usepackage{color}

\begin{document}
\setcounter{tocdepth}{1}

\newtheorem{theorem}{Theorem}    
\newtheorem{proposition}[theorem]{Proposition}
\newtheorem{conjecture}[theorem]{Conjecture}
\def\theconjecture{\unskip}
\newtheorem{corollary}[theorem]{Corollary}
\newtheorem{lemma}[theorem]{Lemma}
\newtheorem{sublemma}[theorem]{Sublemma}
\newtheorem{fact}[theorem]{Fact}
\newtheorem{observation}[theorem]{Observation}
\theoremstyle{definition}
\newtheorem{definition}{Definition}
\newtheorem{notation}[definition]{Notation}
\newtheorem{remark}[definition]{Remark}
\newtheorem{question}[definition]{Question}
\newtheorem{questions}[definition]{Questions}

\newtheorem{example}[definition]{Example}
\newtheorem{problem}[definition]{Problem}
\newtheorem{exercise}[definition]{Exercise}

\numberwithin{theorem}{section}
\numberwithin{definition}{section}
\numberwithin{equation}{section}

\def\reals{{\mathbb R}}
\def\torus{{\mathbb T}}
\def\heis{{\mathbb H}}
\def\integers{{\mathbb Z}}
\def\rationals{{\mathbb Q}}
\def\naturals{{\mathbb N}}
\def\complex{{\mathbb C}\/}
\def\distance{\operatorname{distance}\,}
\def\diststar{{\rm dist}^\star}
\def\sym{\operatorname{Symm}\,}
\def\support{\operatorname{support}\,}
\def\dist{\operatorname{dist}}
\def\Span{\operatorname{span}\,}
\def\degree{\operatorname{degree}\,}
\def\kernel{\operatorname{kernel}\,}
\def\dim{\operatorname{dim}\,}
\def\codim{\operatorname{codim}}
\def\trace{\operatorname{trace\,}}
\def\Span{\operatorname{span}\,}
\def\dimension{\operatorname{dimension}\,}
\def\codimension{\operatorname{codimension}\,}
\def\Gl{\operatorname{Gl}}
\def\nullspace{\scriptk}
\def\kernel{\operatorname{Ker}}
\def\ZZ{ {\mathbb Z} }
\def\p{\partial}
\def\rp{{ ^{-1} }}
\def\Re{\operatorname{Re} }
\def\Im{\operatorname{Im} }
\def\ov{\overline}
\def\eps{\varepsilon}
\def\lt{L^2}
\def\diver{\operatorname{div}}
\def\curl{\operatorname{curl}}
\def\etta{\eta}
\newcommand{\norm}[1]{ \|  #1 \|}
\def\expect{\mathbb E}
\def\bull{$\bullet$\ }
\def\det{\operatorname{det}}
\def\Det{\operatorname{Det}}
\def\multiR{\mathbf R}
\def\bestA{\mathbf A}
\def\bestB{\mathbf B}
\def\bestC{\mathbf C}
\def\Apq{\mathbf A_{p,q}}
\def\Apqr{\mathbf A_{p,q,r}}
\def\rank{\operatorname{rank}}
\def\rankk{\mathbf r}
\def\diameter{\operatorname{diameter}}
\def\bp{\mathbf p}
\def\bff{\mathbf f}
\def\bg{\mathbf g}
\def\essd{\operatorname{essential\ diameter}}

\def\mab{M}
\def\t2{\tfrac12}

\newcommand{\abr}[1]{ \langle  #1 \rangle}
\def\unitQ{{\mathbf Q}}
\def\mbfp{{\mathbf P}}

\def\aff{\operatorname{Aff}}
\def\T{{\mathcal T}}

\newcommand{\Norm}[1]{ \Big\|  #1 \Big\| }
\newcommand{\set}[1]{ \left\{ #1 \right\} }
\newcommand{\sset}[1]{ \{ #1 \} }

\def\one{{\mathbf 1}}
\newcommand{\modulo}[2]{[#1]_{#2}}

\def\rint{ \int_{\reals^+} }
\def\Abest{{\mathbb A}}

\def\barrier{\medskip\hrule\hrule\medskip}

\def\scriptf{{\mathcal F}}
\def\scripts{{\mathcal S}}
\def\scriptq{{\mathcal Q}}
\def\scriptg{{\mathcal G}}
\def\scriptm{{\mathcal M}}
\def\scriptb{{\mathcal B}}
\def\scriptc{{\mathcal C}}
\def\scriptt{{\mathcal T}}
\def\scripti{{\mathcal I}}
\def\scripte{{\mathcal E}}
\def\scriptv{{\mathcal V}}
\def\scriptw{{\mathcal W}}
\def\scriptu{{\mathcal U}}
\def\scripta{{\mathcal A}}
\def\scriptr{{\mathcal R}}
\def\scripto{{\mathcal O}}
\def\scripth{{\mathcal H}}
\def\scriptd{{\mathcal D}}
\def\scriptl{{\mathcal L}}
\def\scriptn{{\mathcal N}}
\def\scriptp{{\mathcal P}}
\def\scriptk{{\mathcal K}}
\def\scriptP{{\mathcal P}}
\def\scriptj{{\mathcal J}}
\def\scriptz{{\mathcal Z}}
\def\frakv{{\mathfrak V}}
\def\frakG{{\mathfrak G}}
\def\frakA{{\mathfrak A}}
\def\frakB{{\mathfrak B}}
\def\frakC{{\mathfrak C}}
\def\frakf{{\mathfrak F}}
\def\fcross{{\mathfrak F^{\times}}}

\author{Michael Christ}
\address{
        Michael Christ\\
        Department of Mathematics\\
        University of California \\
        Berkeley, CA 94720-3840, USA}
\email{mchrist@math.berkeley.edu}

\date{January 19, 2014.  Revised June 4, 2014.}

\title{A Sharpened Hausdorff-Young Inequality}


\maketitle
\tableofcontents

\section{Introduction}

The Fourier transform $\widehat{f}$, for suitably bounded functions
$f:\reals^d\to\complex$, is 
\begin{equation} \widehat{f}(\xi) = \int_{\reals^d} e^{-2\pi i x\cdot\xi}f(x)\,dx.\end{equation}
Integration is with respect to Lebesgue measure. 
With this normalization the Fourier transform is
unitary on $L^2(\reals^d)$, and is a contraction from $L^1$ to $L^\infty$.  

Let $p\in[1,2]$, and let $q=p/(p-1)$ be the exponent conjugate to $p$. 
The Hausdorff-Young-Beckner inequality for $\reals^d$, first established by Babenko 
\cite{babenko} for a discrete family of exponents and by Beckner \cite{beckner} for general exponents, 
states that 
\begin{equation} \norm{\widehat{f}}_{L^{q}} \le \bestA_p^d\norm{f}_{L^p}\end{equation}
for $f\in L^p(\reals^d)$, with optimal constant of the form $\bestA_p^d$ where
\begin{equation} \bestA_p ={p^{1/2p}\,q^{-1/2q}}.  \end{equation}

Equality is attained by all functions 
\begin{equation}\label{eq:Gaussian} \phi(x) = ce^{-Q(x)+x\cdot v} \end{equation}
where $v\in\complex^d$, $c\in\complex$,
and $Q$ is a positive definite real quadratic form.
We call such a function a Gaussian, and denote by $\frakG$ the set of all Gaussians.

Lieb \cite{liebgaussian} went farther, proving that all extremizers are Gaussians. 
For exponents $p$ in the restricted range $(1,\tfrac43]$, Beckner had observed that
this is a consequence of uniqueness in corresponding cases of Young's convolution inequality.
For even integers $q\ge 4$, an alternative approach to determination of the
optimal constant, to identification of extremizers, and to a proof of their uniqueness,
is to use Plancherel's Theorem to reduce matters to an extension of Young's
inequality to convolutions with $q/2$ factors, then to apply the heat equation deformation method. 
But no such argument is available for general exponents.

This paper establishes a stabler form of  uniqueness of extremizers, and a sharper inequality.
Define the distance from $f\in L^p(\reals^d)$ to $\frakG$ to be 
\begin{equation} \dist_p(f,\frakG) = \inf_{G\in \frakG} \norm{f-G}_p.\end{equation}

\begin{theorem} \label{thm:truemain}
Let $d\ge 1$. Let $p\in(1,2)$, and let $q=p/(p-1)$. There exists $c=c(p,d)>0$ such that for every nonzero 
function $f\in L^p(\reals^d)$,
\begin{equation} \norm{\widehat{f}}_{q}\le 
\Big[\bestA_p^d - c\Big(\frac{ \dist_p(f,\frakG) }{ \norm{f}_p }\Big)^2\, \Big] 
\norm{f}_p.\end{equation}
\end{theorem}

The exponent $2$ is optimal.  To our knowledge, for $p\in(4/3,2)$
it was not previously known that $\norm{\widehat{f}}_q/\norm{f}_p-\bestA_p^d$
is majorized on the complement of $\frakG$ 
by any strictly negative function of $\dist_p(f,\frakG)/\norm{f}_p$ alone. 
For $p$ in the range $(1,4/3]$ such a nonquantitative result is a consequence of the
corresponding result for Young's convolution inequality, proved in \cite{christyoungest}.

\subsection{Refinements}
For functions very close to $\frakG$,
a more precise inequality can be formulated in terms of a modified distance function.

\begin{definition} \label{defn:scriptp}
$\scriptp=\scriptp(\reals^d)$ is the real vector space consisting of all
complex-valued quadratic polyonomials $P:\reals^d\to\complex$ of the form 
\begin{equation} P(x) = -\sum_{i,j=1}^d a_{i,j}x_ix_j + \sum_k b_k x_k +c\end{equation}
where $b_k,c\in\complex$, $a_{i,j}\in \reals$,
$a_{j,i}=a_{i,j}$, and the quadratic form $x\mapsto \sum_{i,j}a_{i,j}x_ix_j$
is positive definite.
\end{definition}

The set of all nonzero extremizers $\frakG\setminus\{0\}$ of the $L^p$ Hausdorff-Young-Beckner inequality 
for $\reals^d$
is the real submanifold of $L^p$ consisting of all functions $\exp(P)$ where $P\in\scriptp$. 
The real tangent space to $\frakG$ at $g\in\frakG\setminus\{0\}$ 
is the vector space of all product functions $Pg$, where $P\in\scriptp$ is arbitrary. 
Define the normal space $\scriptn_g$ to $\frakG$ at $g\in\frakG$ to be 
\begin{equation} \scriptn_g = \big\{h\in L^p:
\Re\big(\textstyle\int h \,P \,\overline{g}\,|g|^{p-2}\big)\,
=\,0\ \text{ for all $P\in\scriptp$.} \big\}\end{equation}
There exists $\delta_0>0$ such that
any nonzero function that satisfies $\dist_p(f,\frakG)\le \delta_0\norm{f}_p$
can be expressed in a unique way as 
\begin{equation}
f=\pi(f)+f^\perp\ \text{ where } \  \pi(f)\in\frakG \ \text{ and } \ 
f^\perp\in\scriptn_{\pi(f)}.\end{equation}
Then $\norm{f^\perp}_p\ge \dist_p(f,\frakG)$, and these two quantities are uniformly
comparable, provided that $\norm{f}_p^{-1}\dist_p(f,\frakG)$ is sufficiently small.
Define \[{\rm dist}^\star(f,\frakG)=\norm{f^\perp}_p\] 
whenever $f\ne 0$ and $\dist_p(f,\frakG)\le \delta_0\norm{f}_p$.
Define \begin{equation} \bestB_{p,d} = \tfrac12(p-1)(2-p)\bestA_p^d.\end{equation} 

\begin{theorem} \label{thm:refinement1}
Let $p\in (1,2)$. There exists $\rho>0$ such that uniformly for all functions
$0\ne f\in L^p$ for which $\dist_p(f,\frakG)/\norm{f}_p$ is sufficiently small, 
\begin{equation}\label{eq:refinement1bound} 
\frac{\norm{\widehat{f}}_{q}}{\norm{f}_p} \le \bestA_p^d
-\bestB_{p,d} \norm{f}_p^{-2}{\rm dist}^\star(f,\frakG)^2 
+O\big(\norm{f}_p^{-1}\dist_p(f,\frakG)\big)^{2+\rho}.
\end{equation}
\end{theorem}

The constant $\bestB_{p,d}$ in \eqref{eq:refinement1bound} is not optimal,
as will implicitly be shown below, but we are not able to calculate an optimal constant in explicit form. 
However, the bound \eqref{eq:refinement1bound}
is derived from a still more precise inequality, in which the same constant appears and is optimal. 
Assuming that $f\ne 0$, for any $\eta\in(0,\tfrac12]$ define 
\begin{equation} f^\perp_\eta(x)=
\begin{cases}
f^\perp(x)\qquad &\text{if $|f^\perp (x)|\le \eta|\pi(f)(x)|$,}
\\ 0 & \text{if $|f^\perp(x)| > \eta|\pi(f)(x)|$.} 
\end{cases}\end{equation}

\begin{theorem} \label{thm:refinement2}
For each $d\ge 1$ and $p\in(1,2)$ there exist $\eta_0,\gamma>0$ and $C,c\in\reals^+$
such that for all $0<\eta\le\eta_0$,
if $0\ne f\in L^p$ and $\dist_p(f,\frakG)/\norm{f}_p\le\eta^\gamma$ then
\begin{multline} \frac{\norm{\widehat{f}}_{q}}{\norm{f}_p} \le \bestA_p^d
-\big(\bestB_{p,d}-C\eta\big)\norm{f}_p^{-p} \int |f^\perp_\eta|^2 |\pi(f)|^{p-2}
\\
-c\eta^{2-p} \left(\frac{\dist_p(f,\frakG)}{\norm{f}_p}\right)^{p-2}\,\left(\frac{\norm{f^\perp-f^\perp_\eta}_p}{\norm{f}_p}\right)^2.
\end{multline}
\end{theorem}

In this formulation, the constant $\bestB_{p,d}$ is optimal.
The second and third terms on the right-hand side are both quite favorable.
Because the exponent $p-2$ is negative, the factor $|\pi(f)(x)|^{p-2}$
grows rapidly as $|x|\to\infty$,
and therefore stronger control is provided by the term $\norm{f}_p^{-p} \int |f^\perp_\eta(x)|^2|\pi(f)(x)|^{p-2}\,dx$
than by $\norm{f^\perp_\eta}_p^2\norm{f}_p^{-2}$,
especially for $x$ far from the point at which $|\pi(f)|$ achieves its maximum.
Since $p<2$, $\eta^{2-p} \big(\norm{f}_p^{-1}\dist_p(f,\frakG)\big)^{p-2}$ is much larger than 
$1$ provided that $\dist_p(f,\frakG)/\norm{f}_p\ll\eta$.
Thus the final term is more negative than any bounded
multiple of $(\norm{f^\perp-f_\eta^\perp}_p/\norm{f}_p)^2$.

\subsection{Compactness}
The main step in the analysis is the proof of a nonquantitative result.
\begin{proposition} \label{prop:submain}
Let $p\in(1,2)$ and $d\ge 1$. Let $q=\frac{p}{p-1}$.
For every $\eps>0$ there exists $\delta>0$
such that if $0\ne f\in L^p(\reals^d)$
satisfies
$\norm{\widehat{f}}_{q}\ge (1-\delta)\bestA_p^d\norm{f}_p$
then
$\dist_p(f,\frakG) \le \eps\norm{f}_p$.
\end{proposition}
The case $p\in(1,\tfrac43)$ was established in \cite{christyoungest} as a corollary
of a corresponding result for Young's convolution inequality. 

An equivalent formulation of Proposition~\ref{prop:submain} is that if a sequence
of functions satisfies $\norm{f_n}_p=1$
and $\norm{\widehat{f_n}}_q\to\bestA_p^d$ as $n\to\infty$,
then the sequence $(f_n)$ is precompact in $L^p(\reals^d)$ after each $f_n$ is renormalized
by an appropriate element of the group of symmetries of the inequality. 
Much of our effort is directed towards a proof of this weaker result,
which yields no control on the dependence of $\delta$ on $\eps$.

Analysis of compactness for the Fourier transform goes differently from
corresponding analyses for less nonlocal operators.
Consider a function $f=f_1+f_2$ where $f_j$ are supported in balls $B(x_i,\eps)$
where $|x_1-x_2|\ge\eps^{-1}$ and $\norm{f_j}_p^p = \tfrac12$.
Then $\norm{\widehat{f_1}\widehat{f_2}}_{L^{q/2}}$
is not majorized by any quantity that tends to zero as $\eps\to 0$.
Indeed, $\big|\widehat{f_1}\widehat{f_2}\big|$ is unaffected by
independent translations of $f_1,f_2$.
This makes the exclusion of hypothetical extremizing sequences that concentrate
at more than one point delicate. Such concentration is merely one of the simplest of a broad
spectrum of possibilities.

In a series of works, we have treated
compactness theorems using additive combinatorial information.
These works have examined the Riesz-Sobolev inequality \cite{christRS3}, 
the Brunn-Minkowski inequality \cite{christbmtwo},\cite{christbmhigh}, and Young's convolution inequality
for $\reals^d$ \cite{christyoungest}.
They have relied on a characterization of
sets $A,B\subset\reals^1$ whose sumset $A+B$ is of almost minimal size, in the sense that $|A+B|$ 
is nearly equal to $|A|+|B|$. Some of these works \cite{christbmhigh}, \cite{christyoungest}
have also used symmetrization inequalities.
We have not succeeded in applying the description of sumsets of nearly minimal size,
or symmetrization inequalities, to the Hausdorff-Young inequality.
Instead, we rely here on a description of sets that satisfy the weaker assumption
$|A+B|\le K|A|+K|B|$, where $K$ is an arbitrary parameter; the regime of arbitrarily large $K$ is essential
to the analysis. The control on $A,B$ thus intially obtained is consequently far from what is required
for our purpose.

The reasoning in this paper does not provide an independent proof that all extremizers are Gaussians.
Instead, we show directly that 
that any near-extremizer is close to some exact extremizer, then invoke the known characterization
of extremizers to conclude Proposition~\ref{prop:submain} in the form stated.
On the other hand, techniques for demonstrating compactness are of intrinsic interest. 
We hope that this analysis may find further applications to nonlocal inequalities.

Once Proposition~\ref{prop:submain} is proved,
Theorem~\ref{thm:truemain} can be approached through analysis of the Taylor expansion
of $\norm{\widehat{f}}_q/\norm{f}_p$ about an element of $\frakG$. 
This strategy was used in recent work of Chen, Frank, and Werth \cite{chenfrankwerth} 
to treat a corresponding result for the Hardy-Littlewood-Sobolev inequalities.
The situation is awkward because the functional
$\norm{\widehat{f}}_q/\norm{f}_p$ fails to be twice differentiable; 
its denominator $\norm{f}_p$ suffers this defect in a strong form as a consequence of the relation $p<2$.
Indeed, the discrepancy $\bestA_p^d\norm{f}_p\,-\, \norm{\widehat{f}}_q$
is comparable to $(\norm{f}_p^{-1}\dist_p(f,\frakG))^p\norm{f}_p$
for certain $f$ close in $L^p$ norm to $\frakG$, just as Theorem~\ref{thm:refinement2} suggests. 
Such a variation is larger than quadratic in the distance for small perturbations.

\subsection{Young's convolution inequality}\label{subsect:convolution}

Theorem~\ref{thm:truemain} has a direct consequence of the same type for Young's convolution inequality,
for a limited range of exponents.
Young's convolution inequality states that
\begin{equation}
\big|\langle f_1*f_2,f_3\rangle\big| \le \bestC_p^d \prod_{j=1}^3\norm{f_j}_{L^{p_j}(\reals^d)}
\end{equation}
whenever $p_j\in[1,\infty]$ and $\sum_{j=1}^3 p_j^{-1}=2$,
where the optimal constant, determined by Beckner \cite{beckner}
and by Brascamp and Lieb \cite{brascamplieb} is
\begin{equation} \bestC_p = \prod_{j=1}^3 \bestA_{p_j}.\end{equation}
Here $\langle \varphi,\,\psi\rangle = \int_{\reals^d} \varphi\,\overline{\psi}$,
and $*$ denotes convolution.
Denote by
$\vec{\frakG}\subset L^{p_1}\times L^{p_2}\times L^{p_3}$
the collection of ordered triples
\[ \vec{g} = \left(
c_1 e^{i\xi_1\cdot x} \varphi^{p'_1}(x-a_1),\,c_2 e^{i\xi_2\cdot x} \varphi^{p'_2}(x-a_2),
\,c_3 e^{i\xi_3\cdot x}\varphi^{p'_3}(x-a_3) \right)\]
where $\varphi$ is a nonzero real Gaussian, 
$\xi_j,a_j\in\reals$, $\xi_3 = \xi_1+\xi_2$, $a_3 = a_1+a_2$, and $0\ne c_j\in\complex$.

Denote by $\vec{p}=(p_1,p_2,p_3)$ an ordered triple of exponents, and by
$\vec{f}=(f_1,f_2,f_3)$ any ordered triple of functions $f_j\in L^{p_j}$. 
If each $f_j$ has nonzero norm, then $\vec{f}$ extremizes Young's inequality 
if and only if $\vec{f}\in \vec{\frakG}$.
If $\norm{f_j}_{p_j}=1$ for each index $j$,
define the distance $\vec{\rm dist}(\vec{f},\vec{\frakG})$
from $\vec{f}$ to $\vec{\frakG}$ to be the infimum over all $\vec{g}=(g_1,g_2,g_3)\in\vec{\frakG}$
of $\max_{1\le j\le 3} \norm{f_j-g_j}_{L^{p_j}(\reals^d)}$.
For arbitrary $\vec{f}\in \prod_{j=1}^3 L^{p_j}$ such that each component $f_j$ has nonzero $L^{p_j}$ norm,
set $\tilde f_j=f_j/\norm{f_j}_{p_j}$, and define
\begin{equation} \vec{\rm dist}(\vec{f},\vec{\frakG})
= \prod_{i=1}^3\norm{f_i}_{p_i}\,\cdot\, \inf_{\vec{g}\in\vec{\frakG}} 
\max_{1\le j\le 3} \norm{\tilde f_j-g_j}_{p_j}.\end{equation}
Thus \[ \vec{\rm dist}((c_1f_1,c_2f_2,c_3f_3),\,\vec{\frakG}) = |c_1c_2c_3|
\,\cdot\,\vec{\rm dist}((f_1,f_2,f_3),\,\vec{\frakG}).\]

\begin{corollary} \label{cor:convolution}
Let $d\ge 1$.  Let $\vec{p}\in (1,2]^3$ satisfy $\sum_{j=1}^3 p_j^{-1}=2$.
There exists $c>0$ such that for all $\vec{f}\in \big(L^{p_1}\times L^{p_2}\times L^{p_3}\big)(\reals^d)$,
\begin{equation} \label{eq:bad4bound}
\frac{\big|\langle f_1*f_2,f_3\rangle\big|} 
{\prod_{j=1}^3\norm{f_j}_{L^{p_j}(\reals^d)} }
\le \bestC_p^d 
- c\,\frac{\vec{\rm dist}(\vec{f},\vec{\frakG})^4 } {\prod_{j=1}^3 \norm{f_j}_{p_j}^4}
\end{equation}
Moreover, for each index $j$,
\begin{equation}
{\rm dist}_{p_j}(f_j,\frakG) \le C\delta^{1/2} \norm{f_j}_{p_j}.
\end{equation}
\end{corollary}

The exponent $4$ in \eqref{eq:bad4bound} is not expected to be optimal; we have simply recorded 
a conclusion that follows readily from an invocation of Theorem~\ref{thm:truemain}.  

In earlier work \cite{christyoungest} we have established a compactness result of the same type as Proposition~\ref{prop:submain}
for Young's convolution inequality, for the full range of exponents.
It should be possible to extend Corollary~\ref{cor:convolution} to this full range,
and to improve the suboptimal exponent $4$ in \eqref{eq:bad4bound},
by analyzing the second variation as is done in this paper for the Hausdorff-Young inequality. 
Work in this direction is underway.

\section{Outline of the proof}

A simplified outline of the argument is as follows, for $d=1$.
\begin{enumerate}

\item
The inequalities of Hausdorff-Young for the Fourier transform, and Young for convolution, are
interrelated: If  $\norm{\widehat{f}}_q\ge \eta\norm{f}_p$ then for suitable exponents $\gamma,r,s$,
$\norm{|f|^\gamma*|f|^\gamma}_r\ge c\eta^2 \norm{f^\gamma}_s^2$.

\item
By continuum analogues of theorems of Balog-Szemer\'edi and Fre{\u\i}man, 
a function $f$ that satisfies
$\norm{|f|^\gamma*|f|^\gamma}_r\ge c\eta^2 \norm{f^\gamma}_s^2$
must place a significant portion of its $L^p$ mass on
a continuum multiprogression of controlled rank and Lebesgue measure.

\item
If $\norm{\widehat{f}}_q\ge (1-\delta)\bestA_p^d\norm{f}_p$ and $\delta$ is small
then this reasoning leads to a disjointly supported decomposition $f=g_1+f_1$  
and a (continuum) multiprogression $P_1$ such that $g_1$ is supported on $P_1$, 
$\norm{g_1 \one_{P_1}}_p\ge c \norm{f}_p$ and $\norm{g_1}_\infty|P_1|^{1/p} \le C\norm{f}_p$,
while $\norm{f_1}_p\le (1-c_\delta)\norm{f}_p$.

Given $\eps>0$, if $\delta>0$ is  sufficiently small then
this reasoning can be iterated to produce a decomposition $f = f_N + \sum_{j=1}^N g_j $
where $\norm{f_N}_p<\eps$, $N = O_\delta(1)$, and each $g_j$ is supported
on a multiprogression $P_j$ whose rank is $O_\delta(1)$ and satisfies suitable upper bounds.
This is a general construction once the first two steps are in hand.

\item
If $f$ is a near-extremizer then
the multiprogressions $\set{P_j: 1\le j\le N}$ in any such decomposition
are necessarily compatible in such a manner that their union is 
contained in a single multiprogression $P$, whose rank is still suitably bounded above
and whose measure is $O_\delta(\max_j|P_j|)$, so that
$g=\sum_j g_j$ satisfies $\norm{g}_\infty \le O_\eps(1)|P|^{1/p}$.

\item
Replace $f$ by $\lambda^{1/p}f(\lambda x)$, so that $|P|\asymp 1$.
Then there exists $\lambda\gtrsim 1$ such that the dilate $\lambda P$
is contained in a small neighborhood of $\integers$ in $\reals$.
After replacing $P$ by $\lambda P$, either $|P|\asymp 1$ and
$P$ is contained in an interval of length $\le O_\eps(1)$,
or $P$ has large gaps in a strong sense. The former case is favorable.
The crux of the entire proof is to show that the latter case cannot arise.

\item
If $f$ is supported on a continuum multiprogression contained in a small neighborhood of $\integers$ then
$f$ can be lifted to a function $F(n,x)$ on $\integers\times\reals$ in a natural way. 
The resulting function $F$ nearly extremizes the Hausdorff-Young inequality for $\integers\times\reals$. 

\item
For nearly all $x$ in a suitable sense, the function $n\mapsto F(n,x)$
nearly extremizes the Hausdorff-Young inequality for $\integers$.

\item
It is known that
any near-extremizer of the Hausdorff-Young inequality for a discrete group is very nearly supported on a single point.
Applying this characterization to $n\mapsto F(n,x)$ and then
translating the conclusion back to $f$, one finds that $f$ is very nearly supported
in a single interval of suitable Lebesgue measure,
up to an additive error which is small in $L^p$ norm.

\item
If $f$ nearly extremizes the Hausdorff-Young inequality, then by duality, so does $|\widehat{f}|^{q-2}\,\widehat{f}$.
Therefore the above reasoning applies to $\widehat{f}$, which consequently 
is well localized to some interval $J$ in the same sense as $f$, albeit with $p$ replaced by $q$.

\item
If $f$ is a near-extremizer then
the intervals $I,J$ 
must satisfy a reversed uncertainty bound $|I|\cdot|J|\le C_\eps<\infty$.
By exploiting the action of the group of symmetries of the inequality
generated by translations, dilations, and modulations,
one can therefore ensure that $|I| + |J| \le C_\eps$, and that $I,J$ are centered at $0$.
Given $\eps_0>0$, this construction applies for every $\eps\in[\eps_0,1]$, producing
decompositions and intervals that depend on $\eps$.

\item
It follows that for
any sequence of functions satisfying $\norm{f_n}_p=1$ and $\norm{\widehat{f_n}}_q \to \bestA_p$,
after each function $f_n$ is renormalized to $F_n$ via the action of an appropriate member of the group of symmetries 
of the inequality, $(\widehat{F_n})$ is precompact in $L^q$ norm.

\item
For an extremizing sequence $(g_n)$, precompactness of $(\widehat{g_n})$ in $L^q$ 
implies precompactness of $(g_n)$ in $L^p$, even though this is not so for arbitrary sequences.
This completes the proof of the compactness result, Proposition~\ref{prop:submain}.

\item
We establish a general lemma concerning second variations of ratios
$\norm{Tf}_q/\norm{f}_p$, where  $T:L^p\to L^q$ is any bounded linear operator 
and $1<p<2<q<\infty$.  Its thrust is that while
terms much larger than quadratic do appear, they have favorable signs
and thus reduce the size of the ratio beyond the desired quadratic bound.

\item Analysis of the second variation of the functional $\norm{\widehat{f}}_q/\norm{f}_p$
about a Gaussian leads to the a question about certain eigenvalues of
a specific compact self-adjoint linear operator on $L^2(\reals^d)$.
Exact calculation of its spectrum and eigenfunctions, 
in combination with the general lemma about second variations,
leads to Theorems~\ref{thm:refinement1} and \ref{thm:refinement2}. 
\end{enumerate}

\section{Notations}
By a Gaussian we mean a function of the form $G(x) = c\exp(-Q(x)+x\cdot v)$
where $v\in\complex^d$, $x\cdot v = \sum_{j=1}^d x_jv_j$, and
 $Q:\reals^d\to\infty$ is a real-valued positive definite homogeneous quadratic polynomial;
$Q(x)>0$ for all $x\ne 0$.

If $f,g$ are functions and $E$ is a set, 
we say that $f$ is supported on $E$ if $f(x)=0$ for Lebesgue almost every $x\notin E$,
and we write $f\prec E$.
The notation $f\prec g$ means that $f,g$ have a common domain, up to null sets, 
and that for almost every $x$ in that domain,
either $f(x)=g(x)$ or $f(x)=0$; $f$ equals the product of $g$ with the indicator function of a measurable set.
We say that two functions $f,g$ are disjointly supported if $f(x)g(x)=0$ for almost every point
$x$ in their common domain.

Consider an inequality $\norm{Tf}_q\le \norm{T}\norm{f}_p$, where $T:L^p\to L^q$ is a bounded linear operator
and $\norm{T}$ is its operator norm. 
By a near-extremizer we mean a function that satisfies $\norm{Tf}_q\ge (1-\delta)\norm{T}\cdot\norm{f}_p$
where $\delta>0$ is small. 
By a quasi-extremizer we mean a function that satisfies $\norm{Tf}_q\ge \delta \norm{f}_p$
where $\delta>0$ may be small. 

If $A,B$ are subsets of $\reals^d$ and $m,n$ are positive integers,
$mA-nB$ is defined to be
\[ mA-nB=\big\{\sum_{i=1}^m  a_i-\sum_{j=1}^n  b_j: a_i\in A \ \text{ and } b_j\in B\big\}.\]
A translate of $S$ is a set of the form $S+a = \set{s+a: s\in S}$.

\begin{definition}
A discrete multiprogression ${\mathbf P}$ in $\reals^d$ of rank $r$ is a function
\[{\mathbf P}: \prod_{i=1}^r\set{0,1,\dots,N_i-1} \to \reals^d\] of the form
\[ (n_1,\dots,n_r)\mapsto \big\{a + \sum_{i=1}^r n_i v_i: 0\le n_i<N_i\big\}, \]
for some $a\in\reals$ and some positive integers $N_1,\dots,N_r$. The size of ${\mathbf P}$ is 
$\sigma({\mathbf P})=\prod_{i=1}^r N_i$.
${\mathbf P}$ is said to be proper if this mapping is injective.

$\unitQ^d$ denotes the unit cube \[\unitQ^d=\big\{x=(x_1,\dots,x_d)\in\reals^d: \text{ $0\le x_j\le 1$ 
for every $1\le j\le d$}\big\}.\]

A continuum multiprogression $P$ in $\reals^d$ of rank $r$ is a 
function 
\[P: \prod_{i=1}^r \set{0,1,\dots,N_i-1}\times\unitQ^d \to \reals^d\] of the form
\[(n_1,\dots,n_d;y)\mapsto a+\sum_i n_i v_i + sy\] where $a,v_i\in\reals^d$ and $s\in\reals^+$. 
The size of $P$ is defined to be \[\sigma(P)=s^d \prod_i N_i.\]
$P$ is said to be proper if this mapping is injective.
If $P$ is proper then the Lebesgue measure of its range equals its size.
\end{definition}

A more invariant definition would replace $\unitQ^d$ by an arbitrary convex set 
of positive and finite Lebesgue measure.
The above definition is equivalent for our purposes, and is more convenient.

We will loosely identify a multiprogression with its range, 
and will thus refer to multiprogressions as if they were sets rather than functions. 

$\norm{y}_{\reals/\integers}$  will denote the distance from $y\in\reals$ to $\integers$.
For $d>1$, if $x=(x_1,\dots,x_d)$,
$\norm{x}_{\reals^d/\integers^d} = \max_{1\le j\le d} \norm{x_j}_{\reals/\integers}$.
There is a triangle inequality $\norm{x+y}_{\reals^d/\integers^d} 
\le \norm{x}_{\reals^d/\integers^d} +  \norm{y}_{\reals^d/\integers^d}$.

$O_\delta(1)$ denotes a quantity whose absolute value is less than or equal to a finite constant
that depends only on $\delta$ and on the dimension $d$, or sometimes on both $d$ and
the exponent $p$ in the Hausdorff-Young inequality. Similarly $o_\delta(1)$ denotes a quantity
that tends to zero as $\delta\to 0$, at a rate that may depend on $d$ or on both $d$ and $p$. 
All bounds that depend on $p$ are uniform for $p$ in any compact subset of $(1,2)$.

$\Gl(d)$ denotes be the group of all invertible linear transformations $T:\reals^d\to\reals^d$.  $\aff(d)$ 
denotes the group of all affine automorphisms of $\reals^d$, that is, all maps $x\mapsto \scriptt(x)=T(x)+v$
where $v\in\reals^d$ and $T\in \Gl(d)$.
$\det(T)$ denotes the determinant of $T$, and $|J(\scriptt)|$ denotes the Jacobian determinant of $\scriptt$,
which is equal to $|\det(T)|$. 

$|S|$ denotes the Lebesgue measure of a subset $S$ of a Euclidean space of arbitrary dimension. 
Most frequently $S$ will be a subset of $\reals^d$,
but other spaces including $\reals^{d^2}$ and $\reals^{d(d-1)}$ will also arise.
$\#(S)$ denotes the cardinality of a finite set $S$, and equals $\infty$ if $S$ is an infinite set.

The letters $p,q$ will be systematically used to denote the two exponents in the Hausdorff-Young inequality. 
The exponent conjugate to an exponent $s$ will sometimes be denoted by $s'$.

A disjointly supported decomposition of a Lebesgue measurable function $f$ is a representation $f = \sum_{j} g_j$
where the number of summands is finite, each summand $g_j$ is Lebesgue measurable,
and for almost every $x$, at most one summand $g_j(x)$ is nonzero.

\section{Generalities}
In this section we discuss some general facts about linear inequalities, functions, and operators
which are not specific to the Fourier transform.

\begin{lemma}[Distribution function control] \label{lemma:nicelyboundednearextremizers}
Let $p,q\in(1,\infty)$.
Let $0\ne T$ be a bounded linear operator from $L^p$ to $L^q$, with operator norm $\norm{T}$.
Suppose that $T$ is bounded from $L^{\tilde p}$ to $L^{\tilde q}$ for all
$(\tilde p,\tilde q)$ sufficiently close to $(p,q)$ that satisfy
$\tilde p^{-1} + \tilde q^{-1} = p^{-1}+q^{-1}$.
Then there exists a positive continuous function $\Theta:(0,1]\to (0,\infty)$
satisfying $\lim_{t\to \infty} \Theta(t)=0$ with the following property.
For any $\delta\in(0,\tfrac12]$ and any $0\ne f\in L^p$
satisfying $\norm{Tf}_q \ge (1-\delta)\norm{T}\cdot\norm{f}_p$,
there exists $\tau\in\reals^+$ such that for all $t\ge 1$,
\begin{align}
\int_{|f|\ge t\tau} |f|^p &\le \big(\Theta(t)+o_\delta(1)\big)\norm{f}_p^p
\\ \int_{|f|\le t^{-1}\tau} |f|^p & \le \big(\Theta(t)+o_\delta(1)\big)\norm{f}_p^p.
\end{align}
\end{lemma}
$\Theta$ depends only on $p,q$ and the constants $C,\gamma$ in the definition of a nicely bounded operator.

Let $p\in(1,\infty)$ and consider any $0\ne f\in L^p$.
For $j\in\integers$ define $E_j=\set{x: 2^j\le |f(x)|< 2^{j+1}}$.
Define $f_j = f\one_{E_j}$.
The decomposition $f = \sum_{j\in\integers} f_j$ is uniquely determined by $f$
and will be called the canonical discrete representation of $f$.

If $T:L^p\to L^q$ is a bounded linear operator satisfying the hypotheses
of Lemma~\ref{lemma:nicelyboundednearextremizers} then
there exist $\gamma>0$ and $C<\infty$ such that for any Lebesgue measurable sets $E,E'$
with finite, positive measures,
\begin{equation} \label{eq:nicelybounded}
\big|\langle T\one_E,\one_{E'}\rangle\big|
\le C\min\Big(\frac{|E|}{|E'|},\frac{|E'|}{|E|}\Big)^\gamma |E|^{1/p}|E'|^{1/q'}.
\end{equation}
The two conclusions of Lemma~\ref{lemma:nicelyboundednearextremizers} can be restated as follows:
There exists $k^*\in\integers$ such that the canonical discrete representation of $f$ satisfies
\begin{equation} \norm{\sum_{|j-k^*|\ge N} f_j}_p \le (\Theta(N)+o_\delta(1))\norm{f}_p.  \end{equation}

Inequality \eqref{eq:nicelybounded} leads easily to a proof of Lemma~\ref{lemma:nicelyboundednearextremizers}.
See \cite{christyoungest} for a proof of an analogue of this lemma
for Young's convolution inequality. A simplified version of that reasoning applies here. Therefore we do not
include the details of the proof of Lemma~\ref{lemma:nicelyboundednearextremizers}.

The next lemma applies to arbitrary measure spaces.
\begin{lemma}[No slacking] \label{lemma:noslacking}
For any $p,q\in[1,\infty)$ there exist $c,C_0<\infty$ with the following property.
Let $T:L^p\to L^q$ be a bounded linear operator, with operator norm $\norm{T}$.
Let $\delta\in[0,1]$.
Suppose that $0\ne f\in L^p$ satisfies \begin{equation} \norm{Tf}_q\ge (1-\delta) \norm{T}\cdot\norm{f}_p.\end{equation}
Suppose that $f=g+h$ where $g,h$ have disjoint supports and \begin{equation} \norm{h}_p\ge C_0\delta^{1/p}\norm{f}_p.\end{equation}
Then \begin{equation} \norm{Th}_q \ge c\delta^{(p-1)/p}\norm{f}_p.\end{equation} 
\end{lemma}

\begin{proof}
There exists $C<\infty$ such that for arbitrary functions $G,H\in L^q$,
\begin{equation}
\norm{G+H}_q^q \le \norm{G}_q^q + C \norm{G}_q^{q-1}\norm{H}_q + C\norm{H}_q^q.
\end{equation}
Consequently
\begin{align*}
\norm{T(g+h)}_q^q &\le \norm{Tg}_q^q + C \norm{Tg}_q^{q-1}\norm{Th}_q + C\norm{Th}_q^q
\\ & \le \norm{T}^q \norm{g}_p^q + C\norm{T}^{q-1}\norm{g}_p^{q-1}\norm{Th}_q + C\norm{Th}_q^q.
\end{align*}
On the other hand,
$\norm{g}_p^p + \norm{h}_p^p = \norm{f}_p^p$.

We may assume without loss of generality that $\norm{f}_p=1$. 
Set $y = \norm{h}_p\in(0,1]$ and $x = \frac{\norm{Th}_q}{\norm{T}\cdot\norm{h}_p}\in[0,1]$. 
Then $\norm{g}_p = (1-y^p)^{1/p}\le 1-c_py^p$  for a certain constant $c_p>0$.  Thus
\begin{align*} (1-\delta)^q &\le  (1-y^p)^{q/p} + C(1-y^p)^{(q-1)/p} xy + Cx^qy^q
\\ & \le 1-c_py^p + Cxy \end{align*}
since $(1-y^p)\le 1$ and $x^qy^q\le xy$.  Therefore
\[ x \ge cy^{p-1}-C\delta y^{-1} \ge c\delta^{(p-1)/p} \] 
provided that the constant $C_0$ in the hypothesis $\norm{h}_p\ge C_0 \delta^{1/p}\norm{f}_p$ is sufficiently large.
\end{proof}

\begin{lemma} [Cooperation] \label{lemma:crossterm}
Let $p\in[1,2)$ and $q\in[2,\infty)$.
There exist $c,C\in\reals^+$ with the following property.
Let $0\ne T:L^p\to L^q$ be a bounded linear operator with norm $\norm{T}$.
Let $0\ne f\in L^p$ satisfy $\norm{Tf}_q\ge (1-\delta)\norm{T}\norm{f}_p$.
Suppose that $f=f^\sharp+f^\flat$ where $f^\sharp,f^\flat$ 
satisfy
\begin{align*}
\norm{f^\sharp}_p^p + \norm{f^\flat}_p^p &\le \norm{f}_p^p
\\\min(\norm{f^\sharp},\norm{f^\flat}) &\ge\eta\norm{f}_p.
\end{align*}
Then
\begin{equation} \norm{Tf^\sharp\cdot Tf^\flat}_{L^{q/2}} \ge (c\eta^p-C\delta)\norm{f}_p^2.  \end{equation}
\end{lemma}

The first of the two inequalities for the norms of $f^\sharp,f^\flat$ holds with equality if these two functions
are disjointly supported.  More generally, it holds if $|f^\sharp+f^\flat| = |f^\sharp|+|f^\flat|$
almost everywhere.

\begin{proof}
\begin{align*}
\norm{Tf}_q^q &\le \int (|Tf^\sharp|^2+|Tf^\flat|^2)|Tf|^{q-2} + 2\int |Tf^\sharp\cdot Tf^\flat||Tf|^{q-2}
\\& \le \norm{T}^q\big(\norm{f^\sharp}_p^2 + \norm{f^\flat}_p^2 \big)\norm{f}_p^{q-2}
+ 2\norm{Tf^\sharp\cdot Tf^\flat}_{q/2} \norm{T}^{q-2}\norm{f}_p^{q-2}.
\end{align*}
Thus
\[ \norm{Tf^\sharp\cdot Tf^\flat}_{q/2} 
\ge \norm{T}^2\big( (1-\delta)^q \norm{f}_p^2- \norm{f^\sharp}_p^2-\norm{f^\flat}_p^2\big).  \]

Moreover,
\[ \big(\norm{f^\sharp}_p^2 + \norm{f^\flat}_p^2\big)^{p/2}
\le \norm{f^\sharp}_p^p + \norm{f^\flat}_p^p\]
with strict inequality whenever neither summand vanishes.
By virtue of this strict inequality, we may assume without loss of generality that $\eta$ is small.

Since $p$ is strictly less than $2$ and $\min(\norm{f^\sharp}_p,\norm{f^\flat}_p)\ge\eta\norm{f}_p$,
Taylor expansion gives
\[ \norm{f^\sharp}_p^p + \norm{f^\flat}_p^p
\ge (1+c\eta^p) \big(\norm{f^\sharp}_p^2 + \norm{f^\flat}_p^2 \big)^{p/2}. \]
Equivalently, 
\[\norm{f^\sharp}_p^2 + \norm{f^\flat}_p^2 \le (1-c\eta^p) \big(\norm{f^\sharp}_p^p + \norm{f^\flat}_p^p\big)^{2/p}.\]
Therefore since $\norm{f^\sharp}_p^p + \norm{f^\flat}_p^p\le \norm{f}_p^p$,
\[ \norm{Tf^\sharp\cdot Tf^\flat}_{q/2} \ge \norm{T}^2 \big((1-\delta)^q -1+c\eta^p \big)\norm{f}_p^2
\ge (c\eta^p - C\delta)\norm{f}_p^2.\]
\end{proof}

\section{Properties of multiprogressions}

Our analysis relies on various properties of (continuum) multiprogressions,
and properties of more general sets or functions that are closely associated with continuum multiprogressions. 
In this section we state various auxiliary results concerning these properties.
Several of these are merely translations to the continuum setting of results known for discrete Abelian groups.
Lemma~\ref{lemma:interactionimpliesstructure} concerns the interaction of the Fourier transform with functions
supported on continuum multiprogressions. 
Lemma~\ref{lemma:gapd} is concerned with a specific notion of approximation of a
multiprogressions in $\reals^d$ of fixed but arbitrary rank, by an affine image of a lattice
of rank $d$. It has a more intricate proof.
Proofs of all these results are postponed to \S\ref{section:proofs}.

Recall the notion of the size of a continuum multiprogression, introduced above.
\begin{lemma}[Equivalence with proper multiprogressions] \label{lemma:properversusgeneral}
Let $d\ge 1$.
For any $r<\infty\in\naturals$
there exists $C_r<\infty$ such that for any continuum multiprogression
$Q$ in $\reals^d$ of rank $r$ there exists a continuum multiprogression $P\supset Q$
of rank $r$ satisfying $|P|\le C_r\sigma(Q)$
and $\sigma(P)\le C_r |P|$.
\end{lemma}

The last conclusion is a weak form of the condition that $P$ should be proper.
We will say that $P$ is {\em semiproper} to indicate that $\sigma(P)\le C|P|$,
where $C$ is a constant which depends only on the rank of $P$ and the dimension $d$.

Consider the quotient space $\torus^d = \reals^d/\integers^d$.
For $x\in\reals^d$ define $\norm{x}_{\reals^d/\integers^d}$
to be the Euclidean distance from $x$ to the nearest element of $\integers^d$.

\begin{lemma}[Approximation by $\integers^d$] \label{lemma:gapd}
For each $d\ge 1$ and $\rankk\ge 0$ there exists $c>0$ with the following property. Let 
$P$ be a proper continuum multiprogression in $\reals^d$ of rank $\rankk$, whose Lebesgue measure satisfies $|P|=1$.
Let $\delta\in(0,\tfrac12]$. 
There exists $\scriptt\in \aff(d)$ whose Jacobian determinant satisfies 
\begin{equation} |J(\scriptt)| \ge c \delta^{d\rankk+d^2}\end{equation}
such that
\begin{equation} \norm{\scriptt(x)}_{\reals^d/\integers^d}<\delta \ \text{ for all $x\in P$.}\end{equation}
\end{lemma}

Consider a continuum multiprogression $P\subset\reals^1$ satisfying $|P|=1$, whose rank is bounded above,
and which has large diameter $D\gg 1$.
Applying the lemma with $\delta = D^{-1/2}$ yields $\lambda,b$ such that $Q=\lambda P +b$ lies close to $\integers$, 
yet still has large diameter and therefore cannot lie close to a single element of $\integers$. 
This situation will lead to a contradiction at a critical stage of the proof, leading
to the conclusion that certain continuum multiprogressions have diameters majorized by constant multiples
of their Lebesgue measures. Thus these multiprogressions will be equivalent to intervals.

There is no upper bound on $|J(\scriptt)|$ in the conclusion.
Indeed, consider these examples: Let $\delta>0$. Let $N$ be any integer satisfying $N\ge \delta^{-2}$.
Let $P$ be a union of $N$ intervals of lengths $N^{-1}$, whose centers form a rank one arithmetic progression
with $N$ elements, containg $0$, with increment equal to $N^{-1/2}$. 
Then $|P|=1$, and $P$ has diameter $\approx N^{1/2}$.
Then a good choice for $\scriptt$ is $\scriptt(x) = N^{1/2}x$; 
the dilated set $N^{1/2}P$ lies within distance $N^{-1/2}\le\delta$ of $\integers$.
Any affine mapping with derivative larger than $N^{-1/2}$ but smaller than $N^{1/2}$ fails to map
$P$ to a set lying close to $\integers$.
The derivative of $\scriptt$ is not bounded above uniformly in $N$. 
On the other hand, for any particular $N$, the absolute value of the derivative of $\scriptt$ 
plainly cannot exceed $N$, since $\scriptt(P)$ would otherwise contain intervals of lengths $>1$.

\begin{lemma} \label{lemma:largesubsets}
There exists $C<\infty$ with the following property.  For any $\alpha\in(0,1]$ 
there exists a natural number $N \le C\log(1+\alpha^{-1})$
with the following property.
Let $G$ be any finite Abelian group, and let $E\subset G$ satisfy $\#(E)\ge\alpha\#(G)$.
Then $G$ is contained in the union of $N$ translates of $NE-NE$. 
\end{lemma}

\begin{lemma}[Small sumset implies compatible structures] \label{lemma:compatible}
Let $d\ge 1$, let $\rho\in(1,2)$, and define $s\in(1,\infty)$ by $s^{-1}= 2\rho^{-1}-1$.  
Let $R<\infty$ and $\tau>0$.
Let $P,Q\subset\reals^d$ be continuum multiprogressions of rank $\le R$ with positive, finite Lebesgue measures.
If $\norm{\one_P*\one_Q}_s \ge \tau |P|^{1/\rho}|Q|^{1/\rho}$ then 
\[ \max(|P|,|Q|) \le C\tau^{-C} \min(|P|,|Q|)\]
and there exists a continuum multiprogression $\scriptp\subset\reals^d$ of rank $\le C_{\tau,\rho}R+C_{\tau,\rho}$ satisfying
\begin{gather} P+Q \subset\scriptp \\ |\scriptp|\le C \tau^{-C} \max(|P|,|Q|)  \end{gather}
where $C\in\reals^+$ depends only on $R,\rho,d$.
\end{lemma}

\begin{lemma}[Compatibility of nonnegligibly interacting multiprogressions] 
\label{lemma:interactionimpliesstructure}
Let $d\ge 1$. Let $\Lambda$ be a compact subset of $(1,2)$.  Let $\lambda>0$  and $R<\infty$. 
There exists $C<\infty$, depending only on $\lambda,R,d,\Lambda$, with the following property.
Let $p\in\Lambda$.
Let $P,Q\subset\reals^d$ be nonempty proper continuum multiprogressions of ranks $\le R$.
Let $\varphi\prec P$ and $\psi\prec Q$ be functions that satisfy
$\norm{\varphi}_\infty|P|^{1/p}\le 1$ and $\norm{\psi}_\infty|Q|^{1/p}\le 1$.  If
\begin{equation} \norm{\widehat{\varphi}\,\widehat{\psi}}_{p'/2} \ge \lambda \end{equation}
then \begin{align} \max(|P|,|Q|)& \le C\min(|P|,|Q|) \\ |P+Q| &\le C\min(|P|,|Q|).\end{align}
\end{lemma}

\begin{proposition}[Continuum Fre{\u\i}man Theorem] \label{prop:FreimanR}
Let $K<\infty$.
Let $A,B\subset\reals^d$ be Borel measurable sets with finite, positive Lebesgue measures.
Suppose that $|A+B|\le K|A|+K|B|$  and $|A|\le K|B|\le K^2|A|$.
There exists a proper continuum multiprogression $P\subset\reals^d$ of rank $\le C_K$ such that 
\[ A\subset P\ \text{ and }\  |P| \le C_K  |A|.\]
\end{proposition}

\begin{definition}
Let $A,B\subset\reals^d$ be Lebesgue measurable sets with finite Lebesgue measures. The additive energy $E(A,B)$
of the pair $(A,B)$ is defined to be
\begin{equation} E(A,B) = \norm{\one_A*\one_B}_{L^2(\reals^d)}^2.  \end{equation}
\end{definition}
By Young's convolution inequality,
\begin{equation} E(A,B) \le |A|^{3/2}|B|^{3/2}.  \end{equation}

\begin{proposition}[Continuum Balog-Szemer\'edi Theorem] \label{prop:BSGR}
Let $d\ge 1$ and $1\le K<\infty$.
Let $A,B\subset\reals^d$ be Lebesgue measurable sets with finite Lebesgue measures. 
Suppose that $\max(|A|,|B|)\le K\min(|A|,|B|)$
and that $E(A,B) \ge K^{-1}|A|^3$.
Then there exist Lebesgue measurable subsets $A'\subset A$, $B'\subset B$ satisfying 
\begin{align} \min(|A'|,|B'|) &\ge CK^{-C}|A|
\\ |A'+B'| &\le CK^C |A|.  \end{align}
\end{proposition}

\section{Quasi-extremizers} 

In this section we exploit two connections to obtain structural information concerning
quasi-extremizers of the Hausdorff-Young inequality. The first is a connection between
the Hausdorff-Young inequality for the Fourier transform, and Young's inequality for convolutions.
If $p\in [\tfrac43,2)$, and if $f$ is a near-extremizer for the Hausdorff-Young inequality
with exponent $p$, then $f$ is a near-extremizer for Young's convolution inequality;
more accurately, the ordered pair $(f,f)$ is a near-extremizer 
in the sense that if $\norm{\widehat{f}}_{p'}\ge (1-\delta)\bestA_p^d\norm{f}_p$
then  $\norm{f*f}_s\ge (1-\delta)^2 A\norm{f}_p^2$ where $s^{-1} = 2p^{-1}-1$
and $A$ is the optimal constant in this inequality.
For $p\in(1,\tfrac43)$ this implication fails, but there
is a rather weak substitute: a suitable power of $f$  remains a {\em quasi}--extremizer for the convolution inequality.
Moreover, this holds if $f$ is merely a quasi-extremizer for the Hausdorff-Young inequality.

The second is a connection between Young's convolution inequality, and 
fundamental principles of additive combinatorics. These imply
that quasi-extremizers of Young's convolution inequality have components of nonnegligible size
with certain arithmetic structure. 

Young's inequality with optimal constant states that for any nonnegative functions $f_i\in L^{p_i}(\reals^d)$,
\[ \norm{f_1*f_2}_s \le \bestA_{p_1}^d \bestA_{p_2}^d \bestA_{s'}^d \norm{f_1}_{p_1}\norm{f_2}_{p_2}.  \]

\begin{definition}
$(f_1,f_2)$ is said to be a $\delta$--quasi-extremizer for Young's inequality (with exponents $(p_1,p_2)$) if
$f_j$ has a positive $L^{p_j}$ norm and
\[ \norm{f_1*f_2}_s \ge (1-\delta)\bestA_{p_1}^d \bestA_{p_2}^d \bestA_{s'}^d \norm{f_1}_{p_1}\norm{f_2}_{p_2}.  \]
Likewise, $0\ne f\in L^p$ is a $\delta$--quasi-extremizer for the Hausdorff-Young inequality if
\[ \norm{\widehat{f}}_q \ge (1-\delta)\bestA_p^d\norm{f}_p,\] where $q=p'$.
\end{definition}
By a quasi--extremizer we mean a $\delta$--quasi-extremizer for some small unspecified value of $\delta$.

\begin{lemma}[Quasi-extremizers for Young's inequality] \label{lemma:Youngquasidecomp}
Let $p_1,p_2\in(1,\infty)$ and suppose that the exponent $s$ defined by $s^{-1} = p_1^{-1}+p_2^{-1}-1$
also belongs to $(1,\infty)$. Then
for each $\delta>0$ there exist $c_\delta,C_\delta\in(0,\infty)$ such that
for any $\delta$--quasiextremizer $(f_1,f_2)$ for Young's convolution inequality there exist a 
proper continuum multiprogession $P$ and a function $\varphi$ satisfying
\begin{align*}
& \varphi\prec f_1,
\\&\varphi \prec P,
\\&\norm{\varphi}_{\infty} |P|^{1/p_1} \le C_\delta \norm{f}_{p_1},
\\&\norm{f-\varphi}_{p_1}\le (1-c_\delta) \norm{f}_{p_1}.
\end{align*}
Moreover, there exists $\alpha\in\reals^+$ such that 
$\alpha\le |\varphi(x)|\le 2\alpha$  whenever $\varphi(x)\ne 0$.
\end{lemma}

The final conclusion can be equivalently formulated as $\alpha\le |\varphi(x)|\le C_0\alpha$ 
where $C_0$ is a constant independent of $(f_1,f_2)$. Indeed, then
$\varphi=\sum_{n=0}^{1+\log_2 C_0} \varphi_n$ where $2^n\alpha\le |\varphi_n(x)|< 2^{n+1}\alpha$
whenever $\varphi_n(x)\ne 0$. At least one of the summands $\varphi_n$ must satisfy the conclusions
of the lemma with the stated factor of $2$, with adjusted values of $c_\delta,C_\delta$.

To prove Lemma~\ref{lemma:Youngquasidecomp}, 
choose $f_3$ such that $\langle f_1*f_2,f_3\rangle = \norm{f_1*f_2}_s\norm{f_3}_{p_3}$
where $\sum_{i=1}^3 p_i^{-1}=2$.
Assume without loss of generality that $\norm{f_i}_{p_i}=1$ for each $i\in\set{1,2,3}$.
For each $i\in\set{1,2,3}$ and each $k\in\integers$ let $E_{i,k}=\set{x: 2^k\le |f_i(x)|<2^{k+1}}$. Express 
\begin{equation} \label{eq:levelsetdecomp}
f_i = \sum_{k=-\infty}^\infty 2^k F_{i,k}\end{equation}
where $f_i\one_{E_{i,k}} = 2^k F_{i,k}$.
Thus $1\le |F_{i,k}(x)|<2$ for $x\in E_{i,k}$ and $F_{i,k}(x)=0$ for $x\notin E_{i,k}$.

\begin{lemma} \label{lemma:lorentztheoryconsequences}
There exist $k_i^*$ such that 
\begin{gather*}
2^{k_i^* p_i}|E_{i,k_i^*}| \ge c_\delta 
\\ \langle \one_{E_{1,k_1^*}}*\one_{E_{2,k_2^*}},\one_{E_{3,k_3^*}}\rangle \ge c_\delta \prod_{i=1}^3 |E_{i,k_i^*}|^{1/p_i}
\\ |k_i^*-k_j^*|\le C_\delta\ \text{ for all $i,j\in\set{1,2,3}$}.
\end{gather*}
\end{lemma}

Lemma~\ref{lemma:lorentztheoryconsequences} is parallel to Proposition~11.1 of \cite{christyoungest}
and to a corresponding result in \cite{christonextremals}, and
is proved by the same reasoning. Therefore the proof is omitted.
Granting this lemma,
the proof of Lemma~\ref{lemma:Youngquasidecomp} can be completed, as follows.

\begin{proof}
We may assume without loss of generality that the functions $f_1,f_2$ are nonnegative, 
since their absolute values satisfy the hypotheses.
The third conclusion of Lemma~\ref{lemma:lorentztheoryconsequences} 
states that the three sets $E_{i,k_i^*}$ have comparable measures, with bounds depending only on $\delta$.
Since
\[ \langle \one_{E_{1,k_1^*}}*\one_{E_{2,k_2^*}},\one_{E_{3,k_3^*}}\rangle 
\le \norm{\one_{E_{1,k_1^*}}*\one_{E_{2,k_2^*}}}_2|E_{3,k_3^*}|^{1/2}\]
we have
\[ \norm{\one_{E_{1,k_1^*}}*\one_{E_{2,k_2^*}}}_2
\ge c_\delta |E_{3,k_3^*}|^{-1/2}
\prod_{i=1}^3 |E_{i,k_i^*}|^{1/p_i}
\ge c'_\delta|E_{1,k_1^*}|^{3/2}\]
using the comparability of the measures of these three sets.
Therefore by the Balog-Szemer\'edi theorem, there exist subsets $E_i^\dagger\subset E_{i,k_i^*}$ for $i=1,2$
such that $|E_i^\dagger|\ge c_\delta |E_{i,k_i^*}|$
and $|E_1^\dagger+E_2^\dagger| \le C_\delta |E_1^\dagger|+C_\delta |E_2^\dagger|$.
Apply Proposition~\ref{prop:FreimanR} to obtain a proper continuum multiprogression $P$
such that $E_1^\dagger\subset P$, $|P| \le C_\delta |E_1^\dagger|$, and $P$ has rank $O_\delta(1)$.

Define the function $\varphi = f_1 \one_{E_1^\dagger}$,
which satisfies $\varphi\prec P$, and $0\le\varphi_1 \le 2^{1+k_1^*}\one_{E_{i,k_1^*}}\le f$.  
Moreover \begin{align*}
\varphi 
&= |P|^{-1/p_1} f_1 |P|^{1/p_1}  \one_{E_1^\dagger}
\\&\le 2 |P|^{-1/p_1} 2^{k_1^*} |P|^{1/p_1}  \one_{E_1^\dagger}
\\&\le 2 |P|^{-1/p_1} 2^{k_1^*} |E_1^\dagger|^{1/p_1}  
\\&\le 2|P|^{-1/p_1} 2^{k_1^*} |E_{1,k_1^*}|^{1/p_1}  
\\&\le 2|P|^{-1/p_1} \norm{f_1}_{p_1}.
\end{align*}

Now \[\norm{\varphi}_{p_1}^{p_1} \ge 2^{p_1k_1^*} |E_1^\dagger|
\ge c_\delta 2^{p_1k_1^*} |E_{1,k_1^*}| \ge c_\delta>0\norm{f_1}_{p_1}^{p_1}.\]
Therefore since $\varphi$ and $f-\varphi$ have disjoint supports,
\begin{align*} \norm{f-\varphi}_{p_1}^{p_1} = \norm{f_1}_{p_1}^{p_1} - \norm{\varphi}_{p_1}^{p_1}  
\le \norm{f_1}_{p_1}^{p_1} - c_\delta \norm{f_1}_{p_1}^{p_1}.  
\end{align*}
\end{proof}

\begin{lemma} \label{lemma:quasiextremizersviaconvolution}
Let $d\ge 1$ and $p\in(1,2)$. Let $r\in(1,p)$. There exist $c,\gamma\in\reals^+$ with the following property.
Let $\eta>0$.
Suppose that $0\ne f\in L^p(\reals^d)$ satisfies $\norm{\widehat{f}}_{p'} \ge \eta \norm{f}_p$.
Then $g=|f|^{p/r}$ satisfies
\begin{equation}
\norm{g*g}_{t}\ge c\eta^\gamma \norm{g}_{r}^2
\end{equation} 
where $t^{-1} = 2 r^{-1}-1$.
\end{lemma}

\begin{proof}
We may suppose without loss of generality that $\norm{f}_p=1$.
Consider the family of functions $F_z(x) = f(x)|f(x)|^{z-1}$ for $\Re(z)>1$.
For any $z$, $F_z\in L^{\rho}$ where $\rho=p/\Re(z)$ with $\norm{F_z}_\rho = \norm{f}_p^{\Re(z)}=1$.
In particular, when $\Re(z) = p/2$, $F_z\in L^2$ with norm equal to $1$.

Consider any $s\in(1,p)$, and consider all $z\in\complex$ in the closed strip defined by $\tfrac12 p \le\Re(z)\le s$.
$F_z\in L^2$ whenever $\Re(z)=\tfrac12 p$, while $F_z\in L^{p/s}$ whenever $\Re(z)=s$.
We have $ps^{-1}<p<2$.
Expressing $p^{-1} = (1-\theta) 2^{-1}+\theta sp^{-1}$,
\[ \norm{\widehat{f}}_{p'} = \norm{\widehat{F_1}}_{(p/1)'} 
\le \sup_{\Re(z)=s} \norm{\widehat{F_z}}_{(p/s)'}^\theta
\sup_{\Re(w)=p/2} \norm{\widehat{F_w}}_{2}^{1-\theta}
= \sup_{\Re(z)=s} \norm{\widehat{F_z}}_{(p/s)'}^\theta \]
by the Three Lines Lemma proof of the Riesz-Th\"orin interpolation theorem and Plancherel's theorem.
Therefore there exists $\zeta$ satisfying $\Re(\zeta)=s$ such that 
\[\norm{\widehat{F_\zeta}}_{(p/s)'} \ge \eta^\gamma.\]

Choose $s\in[\tfrac 34 p,p)$. Then $(p/s)'\ge 4$, so $\tfrac12(p/s)'\ge 2$.
With $\zeta$ chosen as above, write $\tfrac12(p/s)' = \sigma'$ where $\sigma\in(1,2]$.
By the Hausdorff-Young inequality,
\[ \norm{\widehat{F_\zeta}}_{(p/s)'}^2 = \norm{ \widehat{F_\zeta*F_\zeta}}_{\sigma'} \le \norm{F_\zeta*F_\zeta}_\sigma.  \]
Thus $\norm{F_\zeta*F_\zeta}_\sigma \ge c\eta^\gamma$.
Since $|F_\zeta|=|f|^{\Re(\zeta)}$, $|F_\zeta*F_\zeta|\le |f|^s*|f|^s$
and thus
$\norm{|f|^s*|f|^s}_\sigma \ge c\eta^\gamma$.
Here $\sigma^{-1} = 1-2(1-sp^{-1})= 2 s p^{-1}-1$.
\end{proof}

\begin{proposition}[Quasi-extremizers for the Hausdorff-Young inequality] \label{prop:HYquasidecomp}
Let $d\ge 1$, let $\Lambda\subset(1,2)$ be a compact set, and let $\eta>0$.
There exist $C_\eta,c_\eta\in\reals^+$ with the following property for all $p\in\Lambda$.
Suppose that $0\ne f\in L^p(\reals^d)$ satisfies $\norm{\widehat{f}}_{p'} \ge \eta \norm{f}_p$.
Then there exist a proper continuum multiprogression $P$ and a decomposition $f=g+h$ such that
\begin{gather*}
\text{$g,h$ have disjoint supports,} 
\\
\text{$g$ is supported on $P$,}
\\
\norm{g}_\infty |P|^{1/p}\le C_\eta \norm{f}_p,
\\
\text{the rank of $P$ is $O_\eta(1)$,}
\\
\norm{g}_p\ge c_\eta \norm{f}_p.
\end{gather*}
Moreover, there exist $\alpha\in\reals^+$, depending only on $d,\Lambda$, 
and $k\in\integers$ depending also on $f$, such that 
\begin{equation} \label{eq:limitedrange} 
\alpha^k\le |g(x)|\le \alpha^{k+1}\ \text{ whenever $g(x)\ne 0$.} \end{equation}
\end{proposition}
Equivalently, \begin{equation}\norm{h}_p\le (1-c'_\eta)\norm{f}_p,\end{equation}
where $c'_\eta>0$ depends only on $d,\Lambda,\eta$.

\begin{proof}
Combine Lemma~\ref{lemma:quasiextremizersviaconvolution} with Lemma~\ref{lemma:Youngquasidecomp}.
\end{proof}

\section{Multiprogression structure of near-extremizers}

\subsection{Simplified formulation}

\begin{lemma} \label{lemma:NE1}
Let $d\ge 1$, and let $\Lambda\subset (1,2)$ be a compact set. 
For any $\eps>0$ there exist $\delta>0$, $N_\eps<\infty$, and $C_\eps<\infty$ 
with the following property for all $p\in\Lambda$.
Let $0\ne f\in L^p(\reals^d)$ satisfy $\norm{\widehat{f}}_{p'} \ge (1-\delta)\bestA_p^d\norm{f}_p$.
Then there exist a disjointly supported decomposition $f=g+h$ 
and continuum multiprogressions $\set{P_i:  1\le i\le N_\eps }$ such that 
\begin{gather*}
 \norm{h}_p<\eps\norm{f}_p
\\ \norm{g}_\infty \sum_i|P_i|^{1/p}\le C_\eps\norm{f}_p
\\ g\prec \cup_{i=1}^{N_\eps} P_i
\\ \rank(P_i)\le C_\eps.
\end{gather*}
\end{lemma}

\begin{proof}[Proof of Lemma~\ref{lemma:NE1}]
Lemma~\ref{lemma:noslacking} asserts that
there exist $\delta>0$ and $\sigma>0$, depending on $\eps,\Lambda,d$ but not on $N$, such that
for any function $h$ satisfying $\norm{\widehat{h}}_{p'} \ge (1-\delta)\bestA_p^d\norm{h}_p$
and any disjointly supported Lebesgue measurable decomposition $h = \varphi+\psi$ with $\norm{\varphi}_p\ge\eps$, 
the summand $\varphi$ must satisfy $\norm{\widehat{\varphi}}_{p'}\ge \sigma\norm{h}_p$.
We choose the parameter $\delta$ in the statement of this lemma to be less than or equal to this value.

Let $f\in L^p(\reals^d)$ and suppose that $\norm{\widehat{f}}_{p'} \ge (1-\delta)\bestA_p^d\norm{f}_p$. 
Assume without loss of generality that $\norm{f}_p=1$. 
Apply Proposition~\ref{prop:HYquasidecomp} iteratively, as follows. 

Let $\eps\in(0,\tfrac12]$ be given.  Let $\delta>0$ be a small quantity to be chosen below as a function of $\eps$. 
For the Step 1, apply Proposition~\ref{prop:HYquasidecomp} to $F=F_0=f$, 
with $\tau=\eps$ and with $\sigma = (1-\delta)\bestA_p^d$.
Since $\eps<1=\norm{F}_p$, the construction produces a disjointly supported decomposition $F=G_1+H_1$ 
and a proper continuum multiprogression $P_1$ 
such that  $\norm{H_1}_p\le (1-\rho)$, $G_1\prec P_1$,
the rank of $P_1$ is $O(1)$, and $\norm{G_1}_\infty |P_1|^{1/p}\le C\norm{F}_1 = C$.
Here $\rho>0$ depends only on $d,p$ provided that $\delta$ is sufficiently small.
Set $F_1=H_1$. Then $f = G_1+F_1$ and  $\norm{F_1}_p\le (1-\rho)<1$.

We execute an iterative procedure indexed by $N\in\naturals$.
The outcome of Step $N$ is a decomposition
$f=G_1+\dots+G_{N}+F_N$ of $f$ such that the summands have pairwise disjoint supports, 
along with proper continuum multiprogressions $P_1,\dots,P_{N}$, 
such that $G_j\prec P_j$, $\rank(P_j)=O_\eps(1)$,
$\norm{G_j}_\infty |P_j|^{1/p}=O_\eps(1)$, and $\norm{G_j}_p \ge c_\eps>0$.

Step $N$ proceeds as follows.
At the completion of Step $N-1$ we have a disjointly supported decomposition
$f=F_{N-1}+ \sum_{j=1}^{N-1}G_j$  of $f$ with associated proper continuum muttiprogressions $P_j$,
with the indicated properties.
If $\norm{F_{N-1}}_p<\eps$ then the iterative procedure halts, and we define 
\begin{equation*} g = \sum_{i=1}^{N-1}G_i \ \text{ and } h = F_{N-1}.  \end{equation*}
Then $f=g+h$ and $g=0$ on the complement of $\cup_i P_i$.
It will be shown below that $g,h,\set{P_i: 1\le i\le N-1}$ have the required properties.

If $\norm{F_{N-1}}_p\ge\eps$ then by the choice of $\delta$,
Proposition~\ref{prop:HYquasidecomp} may be applied to $F_{N-1}$ to produce a decomposition $F_{N-1}=G_N+H_N$ 
with the properties indicated in the statement of Proposition~\ref{prop:HYquasidecomp}. Define $F_{N}=H_N$.
In particular, to $G_N$ is associated a multi-progression $P_N$ satisfying $\norm{G_N}_\infty |P_N|^{1/p} \le C_\eps<\infty$.
Then  \[\norm{G_N}_p \ge c\norm{F_{N-1}}_p\ge c_\eps\] where $c_\eps>0$ is independent of $N$.
Since \[Nc_\eps^p \le \sum_{i=1}^N \norm{G_i}_p^p\le \norm{f}_p^p=1\] because the summands $G_i$ are disjointly supported,
this iterative process is guaranteed to terminate after at most $c_\eps^{-p}$ steps.

Suppose that the procedure has terminated.
We claim that there exists $C_\eps\in\reals^+$ such that for any $f$ satisfying the hypotheses,
there exists $\lambda\in\reals^+$ such that $\norm{G_i}_\infty\in [\lambda,C_\eps\lambda]$ for all indices $i$.
Combined with the bounds $\norm{G_i}_\infty|P_i|^{1/p}\le C_\eps$, 
this implies that \[\max_{i,j}\norm{G_j}_\infty |P_i|^{1/p}\le C_\eps<\infty\]
with a larger value of $C_\eps$ which still depends only on $d,\Lambda,\eps$.

To prove the claim, recall that the functions $G_j$ are constructed to have disjoint supports and so that for
each $x$, either $G_j(x)=f(x)$ or $G_j(x)=0$.
Moreover, according to the final conclusion of Proposition~\ref{prop:HYquasidecomp},
for each $j$ there exists $\alpha_j\in\reals^+$ such that $|G_j(x)|\in[\alpha_j,C\alpha_j]$ whenever $G_i(x)\ne 0$.
By Lemma~\ref{lemma:nicelyboundednearextremizers}, there exists $\lambda=\lambda(f)\in\reals^+$  such that for all $\eta\in\reals^+$,
\begin{equation}\label{eq:middlingvalues} \int_{|f(x)|\ge \eta^{-1}\lambda}|f|^p +  \int_{|f(x)|\le \eta \lambda} |f|^p
\le\Theta(\eta,\delta)\norm{f}_p^p\end{equation}
where $\Theta(\eta,\delta)\to 0$ as $\max(\eta,\delta)\to 0$; $\Theta$ depends on $p,d$ but is independent of $f$.
$\norm{G_j}_p$ is bounded below by the product of $\norm{f}_p$ with a function of $\eps$ alone. 
Together with \eqref{eq:middlingvalues} and the fact that $|G_j(x)|\in [\alpha_j,C\alpha_j]$ whenever $G_j(x)\ne 0$, 
this forces \[C_\eps^{-1}\lambda\le\alpha_j\le C_\eps\lambda\] for all $j$,
provided that $C_\eps$ is sufficiently large and $\delta$ is chosen to be sufficiently small.
This completes the proofs of the claim, and hence of Lemma~\ref{lemma:NE1}.
\end{proof}

We have implicitly proved the following general result. 
\begin{lemma} \label{lemma:NE1general}
Let $p,q\in[1,\infty)$.
Let $T:L^p\to L^q$ be a bounded linear operator.
Let there be given, for each $\tau>0$, a class of functions $\scriptg_\tau\subset L^p$  
and a number $\sigma(\tau)>0$ with the following property:
If $\tau>0$ and if $f\in L^p$ satisfies $\norm{Tf}_q\ge \tau \norm{f}_p$,
then there exists a disjointly supported Lebesgue measurable decomposition $f=g+h$ with $g\in \scriptg_\tau$
and $\norm{h}_p \le (1-\sigma(\tau))\norm{f}_p$.
Then for any $\eps>0$ there exist $\delta,\eta>0$ and $N<\infty$
such that for any function $f\in L^p$ satisfying $\norm{Tf}_q > (1-\delta)\norm{f}_p$
there exists a disjointly supported decomposition $f=h+\sum_{i=1}^N  g_i$ with 
\begin{gather*} \norm{h}_p<\eps\norm{f}_p \\ g_i\in \scriptg_\eta.  \end{gather*}
The quantities $\delta,\eta,N$ may be taken to depend only on $\eps$
and on the function $\tau\mapsto\sigma(\tau)$.
\end{lemma}

\subsection{A more structured decomposition} \label{subsection:morestructure}

Lemma~\ref{lemma:NE1} is not adequate for our purpose, because it produces multiple multiprogressions $P_j$.
In order to eventually invoke
Lemma~\ref{lemma:gapd}, we would like to replace $\set{P_j}$ by a single multiprogression, of appropriately bounded
measure and rank.  We now show that an elaboration of the construction underlying
the proof of Lemma~\ref{lemma:NE1} produces mutually compatible multiprogressions, 
whose union is contained in a single multiprogression, with appropriate bounds.

\begin{lemma}[Multiprogression structure of near-extremizers] \label{lemma:NE2}
Let $d\ge 1$, and let $\Lambda\subset(1,2)$ be compact. 
For any $\eps>0$ there exist $\delta>0$ and $C_\eps<\infty$ with the following property for any $p\in\Lambda$.
Let $0\ne f\in L^p(\reals^d)$ satisfy $\norm{\widehat{f}}_{p'} \ge (1-\delta)\bestA_p^d\norm{f}_p$.
There exist a disjointly supported decomposition $f=g+h$ and a continuum multiprogression $P$ satisfying
\begin{gather*}
 \norm{h}_p<\eps\norm{f}_p
\\ \norm{g}_\infty |P|^{1/p}\le C_\eps\norm{f}_p
\\ g\prec P
\\ \rank(P)\le C_\eps.
\end{gather*}
\end{lemma}

The proof follows that of Lemma~\ref{lemma:NE1}, but requires a refinement.

\begin{proof}
The construction is an iterative one in the form of two nested loops. 
We index the outer loop by $N\ge 1$ and the inner loop by $k\ge 1$.
Step $N=1$ of the construction is unchanged; there is no inner loop for this initial step.

Set $G_0=0$.
Let $\rho>0$ be a small quantity, depending only on $\eps,p,d$, to be chosen  below.

If the outer loop does not terminate after the completion of Step $N-2$, then
Step $N-1$ will produce
a disjointly supported decomposition $f=G_{N-1}+F_{N-1}$ and a single continuum multiprogression $P_{N-1}$, such that
$P_{N-1}$ has rank $\le C_{\eps}$, $G_{N-1}\prec P_{N-1}$, and
$\norm{G_{N-1}}_\infty |P_{N-1}|^{1/p}\le C_\eps\norm{f}_p$.
Moreover, $\norm{G_{N-1}}_p$ is bounded below by a positive multiple of $\norm{f}_p$,
uniformly in $\eps,\delta$ provided that these two quantities are sufficiently small.
Finally, $\norm{G_{N-1}}_p^p \ge \norm{G_{N-2}}_p^p+c_\eps\norm{f}_p^p$
where $c_\eps>0$ depends only on $\eps,p,d$. 

Substep $(N,k)$ will produce a disjointly supported decomposition \[f = G_{N-1} + \psi_{N,k} + \sum_{j=1}^k \varphi_{N,j}\]
in which $\norm{\varphi_{N,j}}^p\ge c_\eps\norm{f}_p$, $\varphi_{N,j}\prec Q_{N,j}$, 
$Q_{N,j}$ is a continuum multiprogression whose rank is $O_\eps(1)$, 
and $\norm{\varphi_{N,j}}_\infty|Q_{N,j}|^{1/p} \le C_\eps$.

Consider any $N\ge 2$.
A decomposition $f = G_{N-1}+F_{N-1}$ is given as the outcome of Step $N-1$.
If $\norm{F_{N-1}}_p<\eps\norm{f}_p$ then the construction halts; its outcome will be analyzed  below.

If $\norm{F_{N-1}}_p\ge \eps\norm{f}_p$ then begin substep $(N,1)$,
by setting $\psi_{N,0}=F_{N-1}$. 
Consider the ordered pair $(G_{N-1},\psi_{N,0})$.
By Lemma~\ref{lemma:noslacking}, $\norm{\widehat{\psi_{N,0}}}_q=\norm{\widehat{F_N}}_q \ge c\eps^{p-1}\norm{f}p$,
provided that $\delta$ is chosen to be sufficiently small. Since $\psi_{N,0}$ is a summand
in a disjointly supported decomposition of $f$, $\norm{f}_p\ge\norm{\psi_{N,0}}_p$.
Thus $\norm{\widehat{\psi_{N,0}}}_q \ge c\eps^{p-1}\norm{\psi_{N,0}}p$.

Proposition~\ref{prop:HYquasidecomp} can consequently 
be applied to $\psi_{N,0}$ to obtain a disjointly supported decomposition 
\[\psi_{N,0}=\varphi_{N,1}+\psi_{N,1}\] and a proper continuum multiprogression $Q_{N,1}$ satisfying
\begin{align*}
&\varphi_{N,1}\prec Q_{N,1}, 
\\& \rank(Q_{N,1})\le C_\eps,
\\& \norm{\varphi_{N,1}}_\infty |Q_{N,1}|^{1/p} \le C_\eps\norm{f}_p,
\\& \norm{\varphi_{N,1}}_p\ge c_\eps\norm{f}_p.
\end{align*}

If $\norm{\widehat{G_{N-1}}\,\widehat{\varphi_{N,1}}}_{p'/2} \ge \rho\norm{f}_p^2$,
or if $\norm{\psi_{N,1}}_p<\eps\norm{f}_p$,
then substep $(N,1)$ halts.
If $\norm{\widehat{G_{N-1}}\,\widehat{\varphi_{N,1}}}_{p'/2} < \rho\norm{f}_p^2$ 
and $\norm{\psi_{N,1}}_p \ge \eps\norm{f}_p$
then rather than augmenting $G_{N-1}$ by $\varphi_{N,1}$ to form $G_N$, we set $\varphi_{N,1}$ aside, 
replace the ordered pair $(G_{N-1},F_N)$ by $(G_{N-1},\psi_{N,1})$, and execute substep $(N,2)$, 
with $\psi_{N,1}$ playing the part of $\psi_{N,0}$ in the above discussion, with halting criterion
$\norm{\widehat{G_{N-1}}\,\widehat{\varphi_{N,2}}}_{p'/2} \ge \rho\norm{f}_p^2$
or $\norm{\psi_{N,1}}_p<\eps\norm{f}_p$.

The general step
$(N,k+1)$ proceeds as follows. It has as its input an ordered pair 
$(G_{N-1},\psi_{N,k})$  with $\norm{\psi_{N,k}}_p\ge\eps\norm{f}_p$
and a disjointly supported decomposition \[f = G_{N-1} + \psi_{N,k} + \sum_{j=1}^k \varphi_{N,j},\]
along with a proper continuum multiprogression $P_{N-1}$.
Apply Proposition~\ref{prop:HYquasidecomp} to $\psi_{N,k}$ 
to obtain a disjointly supported decomposition \[\psi_{N,k} = \varphi_{N,k+1} + \psi_{N,k+1}\]
where $\varphi_{N,k+1}$ is associated to a proper continuum multiprogression $Q_{N,k+1}$ as above.
In particular, $\norm{\varphi_{N,k+1}}_p\ge c_\eps\norm{\psi_{N,k}}_p\ge c'_\eps \norm{f}_p$.
If \[\norm{\widehat{G_{N-1}}\,\widehat{\varphi_{N,k+1}}}_{p'/2} \ge \rho\norm{f}_p^2 
\ \text{ or }  \norm{\psi_{N,k+1}}_p < \eps\norm{f}_p\] 
then substep $(N,k+1)$ halts. Its outcome will be discussed below.
If not, execute substep $(N,k+2)$.
The procedure halts when the first index $k=M$ is reached for which either
$\norm{\widehat{G_{N-1}}\,\widehat{\varphi_{N,M}}}_{p'/2} \ge  \rho\norm{f}_p^2$ or $\norm{\psi_{N,M}}_p < \eps\norm{f}_p$.

For any $N$, the inner loop indexed by $k$ must halt after at most $C_\eps<\infty$ iterations.
Indeed, because the functions $\varphi_{N,i}$ are disjointly supported for distinct indices $i$, 
after substep $(N,M)$
\[ \norm{f}_p^p \ge \sum_{i=1}^M \norm{\varphi_{N,i}}_p^p \ge Mc_\eps\norm{f}_p^p.  \]
Therefore $M\le c_\eps^{-1}$.

Suppose that the subconstruction is iterated $M\le C_\eps$ times, 
and halts because $\norm{\psi_{N,M}}_p < \eps\norm{f}_p$, but  
$\norm{\widehat{G_{N-1}}\,\widehat{\varphi_{N,N}}}_{p'/2} <  \rho\norm{f}_p^2$. 
Then $\norm{\widehat{G_{N-1}}\,\widehat{\varphi_{N,j}}}_{p'/2} <  \rho\norm{f}_p^2$ 
for each index $j\in\set{1,2,\dots,M}$. 
Consequently $\varphi = \sum_{j=1}^N \varphi_{N,j}$ satisfies
\begin{align*}
\norm{\widehat{G_{N-1}}\,\widehat{\varphi}}_{p'/2} 
&\le \sum_{j=1}^M \norm{\widehat{G_{N-1}}\,\widehat{\varphi_{N,j}}}_{p'/2} 
+ \norm{\widehat{G_{N-1}}\,\widehat{\psi_{N,M}}}_{p'/2} 
\\ &\le M\rho\norm{f}_p^2 + \norm{G_{N-1}}_p\norm{\psi_{N,M}}_p
\\ &\le C_\eps \rho\norm{f}_p^2 + \norm{f}_p\cdot \eps \norm{f}_p.
\end{align*}
The parameter $\rho$ is at our disposal.  Choose $\rho$ to satisfy $C_\eps\rho<\eps$,  thus concluding that
\begin{equation}\label{eq:crosstoosmall} \norm{\widehat{G_{N-1}}\widehat{\varphi}}_{p'/2} \le 2\eps\norm{f}_p^2.\end{equation}
On the other hand,
\[ \norm{\varphi}_{p} = \norm{f-\psi_{N,M}}_p \ge (1-\eps)\norm{f}_p\ge\tfrac12\norm{f}_p.\]
Provided that $\eps$ is sufficiently small, as we may assume without loss of generality, 
this together with \eqref{eq:crosstoosmall} contradicts Lemma~\ref{lemma:crossterm}.

Therefore after $M\le C_\eps$ executions of the inner loop,
a function $\varphi_{N,M}$ is produced which satisfies
\begin{equation}\label{eq:crosstermnotsmall} 
\norm{\widehat{G_{N-1}}\,\widehat{\varphi_{N,M}}}_{p'/2} \ge  \rho\norm{f}_p^2\ 
\text{ and }\ \norm{\varphi_{N,M}}_p\ge c_\eps\norm{f}_p\end{equation}
with an associated proper continuum multiprogression $Q_N$ 
such that $\norm{\varphi_{N,M}}_\infty |Q_N|^{1/p}\le C_\eps$
and $Q_N$ has rank $O_\eps(1)$. 
Set $G_N = G_{N-1}+\varphi_{N,M}$,
and set \[F_{N+1} = f-G_N = \psi_{N,M} + \sum_{j=1}^{M-1}\varphi_{N,j}.\] 
By their construction, $G_{N-1}$ and $\varphi_{N,M}$ have disjoint supports. Therefore
\[ \norm{G_N}_p^p = \norm{G_{N-1}}_p^p + \norm{\varphi_{N,M}}_p^p \ge \norm{G_{N-1}}_p^p + c_\eps\norm{f}_p^p\]
where $c_\eps>0$.

By Lemma~\ref{lemma:interactionimpliesstructure}, \eqref{eq:crosstermnotsmall} 
implies that $P_{N-1}$ and $Q_N$ are compatible, in the sense that there exists a proper continuum multiprogression $P_N$
of rank $O_\eps(1)$ that contains both $P_{N-1}$ and $Q_N$ and satisfies $|P_N|\le C_\eps\max(|P_{N-1}|,|Q_N|)$.
Then 
\[\norm{G_N}_\infty |P_N|^{1/p} \le \norm{G_{N-1}}_\infty|P_{N-1}|^{1/p} + \norm{\varphi_{N,M}}_\infty|Q_N|^{1/p}\le C_\eps.\]

If $\norm{F_{N+1}}_p<\eps\norm{f}_p$ then the process halts, and the proof of Lemma~\ref{lemma:NE2} is complete.
If $\norm{F_{N+1}}_p\ge\eps\norm{f}_p$ then begin Step $N+1$. 
Because $\norm{G_N}_p^p \ge \norm{G_{N-1}}_p^p + c_\eps\norm{f}_p^p$, this outer loop must halt after at most $O_\eps(1)$ iterations. 
\end{proof}

\section{Discrete and hybrid groups}

The analysis of Beckner \cite{beckner} relates the Hausdorff-Young inequality for $\reals$ to an inequality
on an infinite product of groups with two elements. Lieb \cite{liebgaussian} exploits the product structure
of the inequality in a different way, using $\reals\times\reals$ to analyze the inequality for $\reals$.
We will likewise use products, lifting functions from $\reals^d$ to $\integers^d\times\reals^d$,
and exploiting analysis with respect to the $\integers^d$ coordinate.

In this discussion, $\widehat{\cdot}$ continues to denote the Fourier transform for $\reals^d$. 
Four other Fourier transforms,
denoted by $\frakf,\fcross,\scriptf,\tilde\scriptf$ and to be defined presently, will be in play. 
The first of these is
$\frakf$, the Fourier transform for the group $\integers^d$, for any $d\ge 1$. It is defined as follows:
\begin{equation}\label{eq:hybridFT} \frakf(f)(\theta) = \sum_{n\in\integers^d} e^{-2\pi i \theta \cdot n}f(n)\end{equation}
for $\theta\in\torus^d$ where $\torus = \reals/\integers$.
Equip $\torus^d$ with the measure defined by Lebesgue measure under the identification of a coset
in $\reals^d/\integers^d$ with its unique representative in $[0,1)^d$.
We will often identify $\torus^d$ instead with $[-\tfrac12,\tfrac12]^d$
in the natural way, disregarding sets of Lebesgue measure zero 
for which the natural maps fail to be well-defined or injective.

The Hausdorff-Young inequality for $\integers^d$ states that
\begin{equation} \norm{\frakf(f)}_{L^{p'}(\torus^d)}\le \norm{f}_{\ell^p(\integers^d)} \end{equation}
for all $1\le p\le 2$, where $p'$ is the exponent conjugate to $p$. The optimal constant in the inequality
is equal to $1$.

The Fourier transform of $f\in L^1(\integers^\kappa\times\reals^d)$  is
\begin{equation}
\fcross(f)(\theta,\xi) = \sum_{n\in\integers^\kappa} \int_{\reals^d} e^{-2\pi i n\cdot\theta} e^{-2\pi i x\cdot\xi}f(n,x)\,dx
\ \text{for $(\theta,\xi)\in\torus^\kappa\times\reals^d$.} 
\end{equation}
The Hausdorff-Young inequality with optimal constant for $\integers^\kappa\times\reals^d$ states that for $p\in[1,2]$ 
and $f\in L^1\cap L^p$,
\begin{equation}\label{eq:sharpHYhybrid} \norm{\fcross(f)}_{L^{p'}(\torus^\kappa\times\reals^d)}
\le \bestA_p^d\norm{f}_{L^p(\integers^\kappa\times\reals^d)} \end{equation}
where $\bestA_p$ is the optimal constant in the $L^p\to L^{p'}$ Hausdorff-Young inequality for $\reals^1$.

To verify that the optimal constant in \eqref{eq:sharpHYhybrid} cannot be smaller than 
the corresponding optimal constant for $\reals^d$, it suffices to consider functions 
\begin{equation*} f(n,x)= \begin{cases} F(x) \qquad &\text{if $n=0$} \\ 0 &\text{if $n\ne 0$.}  \end{cases}\end{equation*}

\begin{proof}[Proof of \eqref{eq:sharpHYhybrid}]
Let $q=p'$.
The product Fourier transform $\fcross$ can be expressed as a composition
$\fcross = \tilde\scriptf\circ\scriptf$ of commuting operators, with 
\begin{align}
\scriptf f(\theta,x) &= \sum_{n\in\integers^\kappa} f(n,x) e^{-2\pi i n\cdot\theta}
\\ \tilde\scriptf g(\theta,\xi) &= \int_{\reals^d} g(\theta,x) e^{-2\pi i x\cdot\xi}\,dx .\end{align}
The former operator maps $L^p(\integers^\kappa\times\reals^d)$ to $L^p_x L^q_\theta(\torus^\kappa_\theta \times \reals^d_x)$
with operator norm $1$, while 
$\tilde\scriptf$ maps $L^q_\theta L^p_x(\integers^\kappa_\theta \times \reals^d_x)$
to $L^q(\torus^\kappa\times\reals^d)$ with operator norm $\bestA_p^d$.
For any function $g(\theta,\xi)$,
\[ \norm{g}_{L^q_\theta L^p_x(\torus^\kappa_\theta \times\reals^d_x)}
\le \norm{g}_{L^p_x L^q_\theta(\reals^d_x\times \torus^\kappa_\theta)}\]
by Minkowski's integral inequality, since $q\ge p$.
Therefore
\begin{multline*}
\norm{\fcross{f}}_{L^q} 
= \norm{\fcross{f}}_{L^q_\theta L^q_\xi}
= \norm{\tilde\scriptf{\scriptf f}}_{L^q_\theta L^q_\xi }
\le  \bestA_p^d \norm{\scriptf f}_{L^q_\theta L^p_x }
\\ \le  \bestA_p^d \norm{\scriptf f}_{L^p_x L^q_\theta} \le \bestA_p^d \norm{f}_{L^p_x L^p_n}  =  \bestA_p^d\norm{f}_p.
\end{multline*}
\end{proof}
This calculation also demonstrates a result which will be important below.
\begin{lemma} \label{lemma:mixednormlowerbound}
If $\norm{\fcross(f)}_q\ge (1-\delta)\bestA_p^d\norm{f}_p$
then \[\norm{\scriptf{f}}_{L^p_xL^q_\theta} \ge (1-\delta) \norm{f}_{p}.\]
\end{lemma}

Near extremizers of the Hausdorff-Young inequality, for arbitrary discrete Abelian groups,
were characterized implicitly by Fournier \cite{fournier},
and explicitly by Eisner and Tao \cite{eisnertao}.
The following more precise statement was shown 
in \cite{christmarcos}, where it was also observed that an equivalent formulation in terms of 
Young's convolution inequality extends to all discrete groups, not necessarily Abelian. 

\begin{theorem} [Near-extremizers for $\integers^d$] \label{thm:discretecase}
Let $\Lambda$ be a compact subset of $(1,2)$.
For any $\eps>0$ there exists $\delta>0$ such that for any dimension $d\ge 1$, any exponent $p\in\Lambda$,
and any $f\in L^p(\integers^d)$ satisfying $\norm{\frakf(f)}_{p'}\ge (1-\delta)\norm{f}_p$,
there exists $z\in \integers^d$ such that $\norm{f}_{L^p(\integers^d\setminus\set{z})}<\eps\norm{f}_p$.
More precisely, there exist a continuous nondecreasing function $\Lambda:(0,1]\to(0,1]$
satisfying \[\Lambda(t) \le 1-c(1-t)^\gamma \text{ as } t\to 1^-\]
for some constant $c>0$ and exponent $\gamma\in(0,\infty)$, 
such that for any $d\ge 1$ and for any function $f\in \ell^p(\integers^d)$,
\begin{equation}
\norm{\frakf(f)}_{p'} \le \norm{f}_p\cdot \Lambda\Big(\frac{\norm{f}_{\infty}}{\norm{f}_{p}}\Big).
\end{equation}\end{theorem}

The next result falls short of fully characterizing near extremizers for $\integers^\kappa\times\reals^d$,
and is not used in the proofs of our main theorems,
but does provide significant information concerning them which will be used in our analysis of
near extremizers for $\reals^d$.
\begin{proposition}[Near-extremizers for $\integers^\kappa\times\reals^d$] \label{prop:1ptdecomp}
Let $\Lambda$ be a compact subset of $(1,2)$.
Let $p\in\Lambda$, and let $\kappa,d$ be positive integers. Let $\delta>0$.
Suppose that $0\ne f\in L^p(\integers^\kappa\times\reals^d)$,
and that \[\norm{\fcross(f)}_{p'} \ge (1-\delta)\bestA_p^d\norm{f}_p.\]
Then there exists a disjointly supported Lebesgue measurable decomposition $f=g+h$ with 
\[\norm{h}_p\le C\delta^{1/(p+1)}\norm{f}_p,\] and
for each $x\in\reals^d$ 
\[\text{There exists at most one $n\in\integers^\kappa$ for which $g(n,x)\ne 0$.}\]
\end{proposition}

\begin{proof}
By Lemma~\ref{lemma:mixednormlowerbound},
the hypothesis on the Fourier transform $\fcross(f)$ implies that
\begin{equation} \label{eq:constantgone} \norm{\scriptf(f)}_{p'} \ge (1-\delta)\norm{f}_p.\end{equation}
Define $f_x(n) = f(n,x)$.
By the definitions of the two Fourier transforms, 
\[\scriptf(f)(\theta,x) \equiv \frakf(f_x)(\theta).\]

Let \[\scripte=\set{x\in\reals^d: \norm{f_x}_p>0}\] and  define the measure
$\mu$ on $\reals^d$ by $\mu(S)=\int_S \norm{f_x}_{L^p(\integers^\kappa)}^p\,dx$.
Define $\tilde f_x(n) = f_x(n)\norm{f_x}_p^{-1}$ for $x\in\scripte$ and
$\tilde f_x(n)\equiv 0$ for all $x\notin\scripte$.
Thus $\norm{\tilde f_x}_{L^p(\integers^\kappa)}=1$ for every $x\in\scripte$. 

By Lemma~\ref{lemma:mixednormlowerbound}, 
\[ \int_{\reals^d} \norm{\frakf(f_x)}_{L^q(\torus^\kappa)}^p\,dx 
\ge (1-c\delta) \norm{f}_{L^p(\integers^\kappa\times\reals^d)}^p.\] 
Equivalently,
\[ \int_{\reals^d} \norm{\frakf(\tilde f_x)}_{L^q(\torus^\kappa)}^p\,d\mu \ge  (1-c\delta)\norm{f}_p^p.\] 

By Theorem~\ref{thm:discretecase}, 
since $\norm{\tilde f_x}_{L^p(\integers^\kappa)}=1$ whenever $\tilde f_x$ does not vanish identically,
\[ \norm{\frakf(\tilde f_x)}_{L^q(\torus^\kappa)}
\le \Lambda(\norm{\tilde f_x}_{L^\infty(\integers^\kappa)}) \le 1-c(1-\norm{\tilde f_x}_{L^\infty(\integers^\kappa)})^\gamma \]
so since $\norm{\tilde f_x}_{L^\infty(\integers^\kappa)}\in[0,1]$ and $\norm{\frakf(\tilde f_x)}_{L^q(\torus^\kappa)}\in[0,1]$,
\[ \norm{\frakf(\tilde f_x)}_{L^q(\torus^\kappa)}^p \le 1-c(1-\norm{\tilde f_x}_{L^\infty(\integers^\kappa)})^\gamma \]
with different constants $c,\gamma\in\reals^+$.  Consequently
\begin{equation}\label{eq:awkwardcomparison}
\int_{\reals^d} (1-\norm{\tilde f_x}_{L^\infty(\integers^\kappa)})^{\gamma}\,d\mu(x) 
\le  \scriptc \int_{\reals^d} (1-\norm{\frakf(\tilde f_x)}_{L^q(\torus^\kappa)}^p)\,d\mu(x)  \end{equation}
where $\scriptc$ denotes a certain finite constant, whose  value does not change in the chain
of inequalities below.

For $x\in\scripte$, $\norm{\tilde f_x}_\infty\le\norm{\tilde f_x}_p=1$,
with equality only if the support of $f_x$ consists of a single element of $\integers^\kappa$.
Let $\eta>0$ be a small parameter to be chosen below, and define 
\[\scriptg=\set{x\in\scripte: \norm{\tilde f_x}_\infty\ge 1-\eta}.\]
For each $x\in\scriptg$ there exists a decomposition
\begin{equation} f_x=g_x+h_x\end{equation} 
where the support of $g_x$ consists of a single element of $\integers^\kappa$, and 
\begin{equation} \label{eq:hxsmallnorm}
\norm{h_x}_{L^p(\integers^\kappa)} \le C\eta^\gamma\norm{f_x}_{L^p(\integers^\kappa)}\end{equation}
where $C,\gamma\in\reals^+$ are positive constants.

By Chebyshev's inequality,
\begin{align*}
\mu(\reals^d\setminus \scriptg) 
&= \mu\big(\{x: 1-\norm{\tilde f_x}_{L^\infty(\integers^\kappa)} > \eta \}\big)
\\&\le \eta^{-\gamma} \int_{\reals^d} (1-\norm{\tilde f_x}_{L^\infty(\integers^\kappa)})^{\gamma}\,d\mu(x) 
\\&\le \scriptc \eta^{-\gamma} \int_{\reals^d} (1-\norm{\frakf(\tilde f_x}_q)^p\,d\mu(x) 
\qquad \text{ by \eqref{eq:awkwardcomparison}}
\\ & = \scriptc\eta^{-\gamma}
-\scriptc\eta^{-\gamma} \int_{\reals^d} \norm{\frakf(\tilde f_x)}_q^p\,d\mu(x)
\\& = \scriptc\eta^{-\gamma}
-\scriptc\eta^{-\gamma} \norm{\scriptf(f)}_{L^p_x L^q_\theta}^p
\\&\le \scriptc\eta^{-\gamma} -\scriptc\eta^{-\gamma} (1-\delta)^p\norm{f}_p^p
\qquad\text{ by \eqref{eq:constantgone}}
\\&  \le C\eta^{-\gamma}\delta \norm{f}_p^p.
\end{align*}

Define $g(n,x)=g_x(n)$ for $(n,x)\in\integers^\kappa\times\scriptg$ and $g(n,x)=0$ for $x\notin\scriptg$.
Define $h(n,x) = h_x(n)$ for $x\in\scriptg$ and $h(n,x)=f(n,x)$ for $x\notin\scriptg$.
Then $g$ has the required properties, while
\begin{align*}
\norm{h}_p^p 
&= \int_{\reals^d\setminus\scriptg} \norm{f_x}_p^p\,dx + \int_\scriptg\norm{h_x}_p^p\,dx
\\ &= \mu(\reals^d\setminus\scriptg) + \int_\scriptg\norm{h_x}_p^p\,dx
\\ &\le C\eta^{-\gamma}\delta \norm{f}_p^p + \int_{\scriptg} C\eta^{p\gamma}\norm{f_x}_p^p\,dx
\qquad \text{ by \eqref{eq:hxsmallnorm}}
\\ & \le  C\eta^{-\gamma}\delta \norm{f}_p^p + C\eta^{p\gamma}\norm{f}_p^p.
\end{align*}
Choose $\eta = \delta^{1/(p+1)\gamma}$ to obtain \[\norm{h}_p  \le C\delta^{1/(p+1)}\norm{f}_p.\]
\end{proof}

\section{Lifting from $\reals^d$ to $\integers^d\times\reals^d$}

Let $d\ge 1$, let $\Lambda\subset(1,2)$ be compact,  and let $p\in\Lambda$.
Let  $\delta\in(0,\tfrac12)$.
Let $0\ne f\in L^p(\reals^d)$ satisfy $\norm{\widehat{f}}_{p'} \ge (1-\delta)\bestA_p^d\norm{f}_p$,
and suppose that $f$ is supported on $\set{x\in\reals^d: \distance(x,\integers^d)<\delta}$.
Define $F:\integers^d\times\reals^d\to\complex$ by
\begin{equation} \label{eq:Fdefn} F(n,x)=
\begin{cases} f(n+x)\ \ &\text{for $(n,x)\in \integers^d\times[-\delta,\delta]^d$}
\\ 0 &\text{for $x\in\reals^d\setminus[-\delta,\delta]^d$.} \end{cases} \end{equation}
Then $\norm{F}_{L^p(\integers^d\times\reals^d)} = \norm{f}_{L^p(\reals^d)}$.

Then
\begin{equation} \fcross(F)(\theta,k+\theta) = \widehat{f}(k+\theta) \ \text{ for all $k\in\integers^d$.}  \end{equation}
Indeed,
\begin{align*}
\fcross(F)(\theta,\xi) 
&= \sum_{n\in\integers^d}\int_{[-\tfrac12,\tfrac12]^d} e^{-2\pi i(n\cdot\theta + x\cdot\xi)} f(n+x)\,dx
\\& = \sum_{n\in\integers^d}\int_{n+[-\tfrac12,\tfrac12]^d} e^{-2\pi i(n\cdot\theta + (y-n)\cdot\xi)} f(y)\,dy
\end{align*}
which simplifies to
$\int_{\reals^d} e^{-2\pi i(y\cdot\xi)} f(y)\,dy$ when $\xi-\theta\in\integers^d$.

This formula is not well suited to our application, because it provides no
control of $\scriptf(F)(\theta,\xi)$ for most $(\theta,\xi)\in\torus^d\times\reals^d$.
Writing $\reals^d\owns\xi=k+\alpha$ where $(k,\alpha) \in\integers^d\times[-\tfrac12,\tfrac12]^d$,
\begin{align*}
\fcross(F)(\theta,k+\alpha) 
&= \sum_{n\in\integers^d} \int_{x\in\reals^d}
e^{-2\pi i(n\cdot\theta + x\cdot(k+\alpha))} F(n,x)\,dx
\\&= \sum_{n\in\integers^d} \int_{x\in[-\delta,\delta]^d}
e^{-2\pi i(n\cdot\theta + x\cdot(k+\alpha))} f(n+x)\,dx
\\&= \sum_{n\in\integers^d} \int_{x\in[-\delta,\delta]^d}
e^{-2\pi i((n+x)\cdot(k+\theta)} e^{-2\pi i x\cdot(\alpha -\theta)} f(n+x)\,dx
\\ & =\widehat{f}(k+\theta) + \sum_{n\in\integers^d} \int_{x\in[-\delta,\delta]^d}
e^{-2\pi i((n+x)\cdot(k+\theta)} \big(e^{-2\pi i x\cdot(\alpha -\theta)}-1\big) F(n,x)\,dx.
\end{align*}

\begin{lemma}
There exist $C,\gamma\in\reals^+$ such that for all $p\in\Lambda$,
all $\delta\in(0,\tfrac12]$, and all functions $f\in L^p(\reals^d)$ supported
in $\set{x: \distance(x,\integers^d)<\delta}$,
\begin{equation} \Big|\norm{\fcross(F)}_{L^{p'}(\integers^d\times\reals^d)} - \norm{\widehat{f}}_{L^{p'}(\reals^d)}\Big| 
\le C\delta^\gamma\norm{f}_p.  \end{equation}
\end{lemma}

\begin{proof}
The mapping $f\mapsto F$ defined above
is a bounded linear operator from $L^r(\reals^d)$ to $L^r(\integers^d\times\reals^d)$
for all $r\in[1,\infty]$, so by the Hausdorff-Young inequality,
$f\mapsto \fcross(F)$ is bounded from $L^r(\reals^d)$ to $L^{r'}(\torus^d\times\reals^d)$ for all $r\in[1,2]$.

Consider the linear operator $f\mapsto G(f)$, where $G(f)(\theta,y)$ is defined 
for all $(\theta,y)\in\torus^d\times\reals^d$ by
\[G(f)(\theta,k+\alpha) = \widehat{f}(k+\theta)
\ \text{ for $(k,\alpha)\in\integers^d\times[-\tfrac12,\tfrac12]^d$.}\]
The operator $G$ is likewise bounded from $L^r(\reals^d)$ to $L^{r'}(\torus^d\times\reals^d)$ for all $r\in[1,2]$.

Thus the difference $T(f)=\fcross(F)-G(f)$
is a bounded linear operator from $L^r(\integers^d+[-\delta,\delta]^d)$ to $L^{r'}(\torus^d\times\reals^d)$, for $r\in[1,2]$. 
For $r=2$, $T$ is bounded uniformly in $\delta$.  For $r=1$, $T$ is bounded with norm $O(\delta)$
since
\[
e^{-2\pi i x\cdot(\alpha -\theta)}-1\  =\  O(\delta)\]
uniformly for all $\theta,\alpha\in [-\tfrac12,\tfrac12]^d$. 
The lemma follows by complex interpolation between these two inequalities.
\end{proof}

\begin{corollary}[Lifts of near-extremizers remain near-extremizers] \label{cor:lifted}
For any $d\ge 1$ and any compact set $\Lambda\subset(1,2)$ there exist positive constants $C,\gamma$
with the following property.  
Let $\eta<\tfrac12$. 
Let $p\in\Lambda$ and $f\in L^p(\reals^d)$.
Suppose that $f$ is supported in $\set{x\in\reals^d: \distance(x,\integers^d)\le\eta}$.
Let $F\in L^p(\integers^d\times\reals^d)$ be associated to $f$ by \eqref{eq:Fdefn}.
If $\norm{\widehat{f}}_{L^{p'}(\reals^d)} \ge (1-\delta)\bestA_p^d\norm{f}_{L^p(\reals^d)}$
then \begin{equation} \norm{\fcross(F)}_{L^{p'}(\torus^d\times\reals^d)} 
\ge (1-\delta-C\eta^\gamma)\bestA_p^d\norm{F}_{L^p(\integers^d\times\reals^d)}.\end{equation}
\end{corollary}

\section{Spatial localization}

In this section we combine what has come before to prove that
modulo an additive remainder term with small norm,
any near-extremizer is majorized by a suitably bounded multiple of the indicator function of an ellipsoid.

\begin{proposition}[Spatial localization] \label{prop:spatiallocalization}
Let $d\ge 1$, and let $\Lambda$ be a compact subset of $(1,2)$.
For every $\eps>0$ there exists $\delta>0$ with the following property for every exponent $p\in\Lambda$.
Let $0\ne f\in L^p(\reals^d)$ satisfy $\norm{\widehat{f}}_{p'} \ge (1-\delta)\bestA_p\norm{f}_p$.
There exist an ellipsoid $\scripte\subset\reals^d$ and a decomposition $f = \varphi+\psi$ such that 
\begin{gather}
\norm{\psi}_p<\eps
\\ \varphi\prec \scripte
\\ \norm{\varphi}_\infty |\scripte|^{1/p}\le C_\eps\norm{f}_p.
\end{gather}
\end{proposition}

\begin{proof}
Let $q=p'$.
Let $\delta$ be small and assume that $f\in L^p$ satisfies $\norm{\widehat{f}}_{q} \ge (1-\delta)\bestA_p^d \norm{f}_p$.
There exist a disjointly supported decomposition $f = G+H$ and associated proper continuum multiprogression $P$
satisfying the conclusions of Lemma~\ref{lemma:NE2}. 
Thus $\norm{H}_p \le o_\delta(1)\norm{f}_p$,
$G\prec P$, $\norm{G}_\infty |P|^{1/p} \le C_\delta \norm{f}_p$,
and the rank of $P$ is majorized by a finite quantity that depends only on $\delta,\Lambda,d$.  
Thus \begin{gather*} \big| \norm{G}_p-\norm{f}_p\big| \le\norm{H}_p \le o_\delta(1)\norm{f}_p,
\\ \norm{\widehat{G}}_q \ge (1-o_\delta(1))\bestA_p^d\norm{G}_p.  \end{gather*}
To complete the proof, it suffices to show that there exists an ellipsoid $\scripte\supset P$ that satisfies
$|\scripte|\le C|P|$, where the constant $C$ depends only on $\delta,\Lambda,d$.

By replacing $f$ by $t^{d/p}f(t x)$ for $t=|P|$ we may assume that $|P|=1$.
Express $P$ as $P = [0,\eta]^d + \{a + \sum_{j=1}^\rankk n_j v_j: 0\le n_j<N_j\}$ where each $v_j\in\reals^d$
and $\eta^d \prod_j N_j = |P|=1$.

Let $\tau>0$ be a small quantity to be chosen later.
Let $\rankk$ be the rank of $P$.
According to Lemma~\ref{lemma:gapd}, there exists 
$\scriptt\in\aff(d)$ such that 
\begin{equation*} \norm{\scriptt(x)}_{\reals^d/\integers^d}<\tau\ \text{ for every $x\in P$} \end{equation*}
and
\begin{equation*} \label{eq:lambdalowerbound} |\det(\scriptt)|\ge c\tau^{d\rankk+d^2}.  \end{equation*}
Here $\det(\scriptt)$ denotes the determinant of $\scriptt$.
The availability of a positive lower bound that depends only on $\tau,d$
in the second inequality is of central importance, but the precise form of this bound is not.

Replace the function $f(x)$ by $|\det(\scriptt)|^{-1/p}f(\scriptt^{-1}x)$, 
and make the corresponding modifications of $G,H$.
Replace $P$ by $\scriptt(P)$.
The ratio $\norm{\widehat{f}}_{q}/\norm{f}_p$ is unchanged, and the modified functions $G,H$ enjoy the same properties
relative to the modified $f,P$ as above, including the bound $\norm{G}_\infty |P|^{1/p}\le C_\delta\norm{f}_p$.
We now have
\begin{gather} \label{eq:nearlattice} \norm{x}_{\reals^d/\integers^d}\le \tau\ \text{ for all $x\in P$}
\\ \label{eq:Pcontrolled} |P| \ge c_\tau>0 \end{gather}
where $c_\tau$ depends only on $\tau,d$.
However, Lemma~\ref{lemma:gapd} provides no upper bound on the Jacobian determinant of
$\scriptt$; therefore we now have no upper bound on the measure of this modified set $P$.

For $(n,x)\in\integers^d\times\reals^d$ define 
\begin{equation}F(n,x) = \begin{cases} G(n+x)\  &\text{if $x\in[-\tfrac12,\tfrac12]^d$}
\\  0 &\text{ otherwise.} \end{cases}\end{equation} 
This function satisfies upper and lower bounds 
\[ (1-o_\delta(1))\norm{f}_{L^p(\reals^d)} \le \norm{F}_{L^p(\integers^d\times\reals^d)}
=\norm{G}_{L^p(\reals^d)} \le \norm{f}_{L^p(\reals^d)}.\]

By Corollary~\ref{cor:lifted}, \eqref{eq:nearlattice} allows us to lift the problem
to $\integers^d\times\reals^d$, ensuring that
\begin{equation} \norm{\fcross(F)}_{L^q(\torus^d\times\reals^d)} \ge (1-o_{\delta,\tau}(1))\bestA_p^d
\norm{F}_{L^p(\integers^d\times\reals^d)}.\end{equation}
Here and below, $o_{\delta,\tau}(1)$ denotes a quantity that is majorized by a function
of $\delta,\tau,\Lambda,d$ alone,
and that tends to zero as $\max(\delta,\tau)\to 0$ while $\Lambda,d$ remain fixed.

Define $\phi = \tilde\scriptf(F)$; 
$\phi(n,\xi)$ is the partial Fourier transform of $F$ with respect to the second variable.
Define $\phi_\xi(n) = \phi(n,\xi)$. 
Then $\fcross(F)(\theta,\xi) = \frakf{\phi_\xi}(\theta)$ for $(\theta,\xi)\in\torus^d\times\reals^d$.

Define $A(\xi) = \norm{\frakf{\phi_\xi}}_q/\norm{\phi_\xi}_p$ if $\norm{\phi_\xi}_p\ne 0$,
and $A(\xi)=0$ otherwise. Then $A(\xi)\le 1$ for all $\xi$ by the Hausdorff-Young inequality for $\integers^d$, and
\begin{align*}
\norm{\fcross(F)}_{L^q(\torus^d_\theta\times\reals^d_\xi)} 
= \norm{\frakf{\phi_\xi}}_{L^q_\xi L^q_\theta}
=  \norm{A(\xi) {\phi_\xi}}_{L^q_\xi L^p_n}.
\end{align*}
Let $\sigma>0$ be small and define
\begin{equation} \scripte=\set{\xi\in\reals^d: 
\norm{\phi_\xi}_p\ne 0 \ \text{ and } \ \norm{\frakf{\phi_\xi}}_q < (1-\sigma)\norm{\phi_\xi}_p }.  \end{equation}
That is, $\scripte$ is the set of all $\xi$ for which $A(\xi)<1-\sigma$.
Now
\begin{align*}
\norm{A(\xi)\phi_\xi}_{L^q_\xi L^p_n}^q
& = \int_{\scripte}A(\xi)^q \norm{\phi_\xi}_p^q\,d\xi + \int_{\reals^d\setminus\scripte}A(\xi)^q \norm{\phi_\xi}_p^q\,d\xi
\\& \le \int_{\scripte}(1-\sigma)^q \norm{\phi_\xi}_p^q\,d\xi + \int_{\reals^d\setminus\scripte}\norm{\phi_\xi}_p^q\,d\xi
\\& \le \norm{\phi}_{L^q_\xi L^p_n}^q - c\sigma \int_\scripte \norm{\phi_\xi}_p^q\,d\xi.
\end{align*}
Therefore
\begin{align*}
(1-o_{\delta,\tau}(1))\bestA_p^{qd} \norm{F}_{p}^q
&\le \norm{\fcross(F)}_q^q
\\&= \norm{A(\xi)\phi_\xi}^q_{L^q_\xi L^p_n}
\\&\le \norm{\phi}_{L^q_\xi L^p_n}^q - c\sigma \int_\scripte \norm{\phi_\xi}_p^q\,d\xi
\\&\le \bestA_p^{dq} \norm{F}_p^q - c\sigma \int_\scripte \norm{\phi_\xi}_p^q\,d\xi.
\end{align*}

Therefore by choosing $\sigma$ to be a function of $\delta$ that tends to zero
sufficiently slowly as $\max(\delta,\tau)\to 0$, one obtains
\begin{align} & \int_\scripte \norm{\phi_\xi}_p^q\,d\xi \le o_{\delta,\tau}(1)\norm{f}_p^q 
\\ & \norm{\frakf(\phi_\xi)}_q \ge (1-o_{\delta,\tau}(1))\norm{\phi_\xi}_p \ \text{ for all $\xi\notin\scripte$.} \end{align}

By Theorem~\ref{thm:discretecase},
for each $\xi\notin\scripte$ there exist $m(\xi)\in\integers^d$ and a decomposition 
\begin{equation} \phi_\xi = g_\xi + h_\xi\end{equation}
such that $g_\xi,h_\xi\in L^p(\integers^d)$, 
$g_\xi$ and $h_\xi$ have disjoint supports for each $\xi$,
\begin{gather*} g_\xi(n)=0\ \text{ for all $n\ne m(\xi)$,}
\\
\norm{h_\xi}_p \le o_{\delta,\tau}(1)\norm{\phi_\xi}_p.\end{gather*}
Set $g(n,\xi)=g_n(\xi)$ and $h(n,\xi) = h_n(\xi)$ for all $x\notin\scripte$;
set $g(n,\xi)=h(n,\xi)=0$ for all $x\in\scripte$.
Thus $\phi(n,\xi) = g(n,\xi)+h(n,\xi)$ for all $\xi\notin\scripte$. 

For each $\xi\in\reals^d$, 
\begin{align*} \scriptf(F)(\theta,\xi) 
&= \sum_{n\in\integers^d} e^{-2\pi i n\cdot\theta} \phi_\xi(n)
\\ &= \sum_{n\in\integers^d} e^{-2\pi i n\cdot\theta}g(n,\xi) 
+ \sum_{n\in\integers^d} e^{-2\pi i n\cdot\theta}h(n,\xi) .\end{align*}

Since the function $n\mapsto g(n,\xi)$ is supported where $n=m(\xi)$,
\[|\sum_n e^{-2\pi i n\cdot\theta} g(n,\xi)|^q = |g(m(\xi),\xi)|^q = \sum_n |g(n,\xi)|^q \]
for $\xi\in\reals^d\setminus\scripte$. Consequently
\begin{multline} 
\norm{\scriptf(g)(\cdot,\xi)}_{L^q(\torus^d)}^q
=
\int_{\reals^d/\integers^d} |\sum_n e^{-2\pi i n\cdot\theta}g(n,\xi)|^q\,d\theta
=
\int_{\reals^d/\integers^d} \sum_n  |g(n,\xi)|^q\,d\theta
\\=  \sum_n |g(n,\xi)|^q
\le  \sum_n |\phi(n,\xi)|^q
= \norm{\phi(\cdot,\xi)}_{L^q(\integers^d)}^q;
\end{multline}
the last inequality follows because $\phi(n,\xi)=g(n,\xi)+h(n,\xi)$
and $g(n,\xi)$, $h(n,\xi)$ are disjointly supported functions of $n$ for each $\xi$
under discussion, that is, for each $\xi\in\reals^d\setminus\scripte$.

Writing $F_n(\xi) = F(n,\xi)$,
\begin{align*}
\int_{\reals^d} \sum_{n\in\integers^d} |\phi(n,\xi)|^q\,d\xi
& = \sum_{n\in\integers^d} \int_{\reals^d}|\phi(n,\xi)|^q\,d\xi
\\& 
 = \sum_{n\in\integers^d} \int_{\reals^d}|\tilde\scriptf(F)(n,\xi)|^q\,d\xi
\\& \le 
\sum_{n\in\integers^d} \bestA_p^{dq} \norm{F_n}_p^q
\\& \le  
\sup_{n\in\integers}\norm{F_n}_p^{q-p} \bestA_p^{dq} \sum_{n\in\integers^d}  \norm{F_n}_p^p
\\&  =   
\sup_{n\in\integers}\norm{F_n}_p^{q-p} \bestA_p^{dq} \norm{F}_{L^p}^p
\\& \le 
\sup_{n\in\integers}\norm{F_n}_p^{q-p} \bestA_p^{dq} \norm{f}_{L^p}^p.
\end{align*}
Thus we have shown that
\begin{equation} \label{eq:scriptfgbound}
\norm{\tilde\scriptf(g)}_{L^q(\torus^d\times\reals^d)}^q
\le \sup_{n\in\integers}\norm{F_n}_p^{q-p} \bestA_p^{dq} \norm{f}_{L^p}^p.
\end{equation}

The contribution of $h$ is small. Indeed, by the Hausdorff-Young inequality for $\integers^d$,
\[ \int_{\torus^d} |\scriptf(h)(\theta,\xi)|^q\,d\theta \le \norm{h(\cdot,\xi)}_{L^p(\integers^d)}^q\]
where $\scriptf$ continues to denote the partial Fourier transform with respect to the
first variable with the second variable, $\xi$, held fixed.
Therefore since $h(\cdot,\xi)\equiv 0$ for $\xi\in\scripte$,
\begin{multline*}
\norm{\scriptf(h)}_{L^q(\torus^d\times\reals^d)}^q
 =  \int_{\reals^d\setminus\scripte} \norm{\scriptf(h) (\cdot,\xi)}_{L^q(\torus^d)}^q\,d\xi 
\\ \le \int_{\reals^d\setminus\scripte} \norm{h(\cdot,\xi)}_{L^p(\integers^d)}^q\,d\xi 
\le \int_{\reals^d} o_{\delta,\tau}(1)\norm{\phi_\xi}_{L^p(\integers^d)}^q\,d\xi
\\= o_{\delta,\tau}(1)\norm{\phi}^q_{L^q_\xi L^p_n}
\le o_{\delta,\tau}(1)\norm{f}^q_p.
\end{multline*}

Set $F_n(x)=F(n,x)$.
Consider 
$\norm{\fcross(F)}_{L^q(\torus^d\times\reals^d)}$,
recalling that $\fcross(h)(\theta,\xi)=\scriptf(h)(\theta,\xi)$. 
The contributions of $\torus^d\times\scripte$
and of $h$ have been shown above to be $o_{\delta,\tau}(1)\cdot \norm{f}_p$. 
Combine these upper bounds with the bound for the contribution of $g$ obtained in \eqref{eq:scriptfgbound}
to conclude that
\begin{align*}
\norm{\fcross(F)}_{L^q(\torus^d\times\reals^d)}^q
&\le o_{\delta,\tau}(1) \norm{f}_p^q 
+ (1+o_{\delta,\tau}(1)) \int_{\reals^d} \sum_{n\in\integers^d} |\phi(n,\xi)|^q\,d\xi
\\&  \le   o_{\delta,\tau}(1) \norm{f}_p^q  
+ (1+o_{\delta,\tau}(1)) 
\sup_{n\in\integers}\norm{F_n}_p^{q-p} \bestA_p^{dq} \norm{f}_{L^p}^p.
\end{align*}

Therefore since $\norm{\fcross(F)}_{q}\ge (1-o_{\delta,\tau}(1)) \bestA_p^d \norm{F}_p
\ge (1-o_{\delta,\tau}(1))\norm{f}_p$,
\begin{equation*} \sup_{n\in\integers}\norm{F_n}_p \ge \big(1-o_{\delta, \tau}(1)\big)\norm{f}_p, \end{equation*}
that is, there exists $n_0\in\integers^d$ such that
\begin{equation} \norm{G}_{L^p(n_0+\tfrac12\unitQ_d)} \ge \big(1-o_{\delta,\tau}(1)\big)\norm{f}_p.\end{equation}

Define $g(x)=G(x)$ if $x\in n_0+\tfrac12 \unitQ_d$, and $g(x)=0$ otherwise.
Define a progression $P^*\subset\reals^d$ to be $P^* = n_0+\tfrac12\unitQ_d$. Define $h = H + (G-g)$.
Then $f=g+h$, $g\prec P^*$, and $\norm{h}_p\le o_{\delta,\tau}(1)\norm{f}_p$.

By \eqref{eq:Pcontrolled},
\[\norm{g}_\infty |P^*|^{1/p} \le \norm{G}_\infty
= \norm{G}_\infty |P|^{1/p} \cdot|P|^{-1/p}
\le C_\delta |P|^{-1/p} \norm{f}_p 
\le C_\delta C'\tau^{-C'}\norm{f}_p\]
where $C'<\infty$ depends only on $d$ and on the rank of $P$, which in turn is bounded by a function of $\delta$ alone.
Choose $\tau$ to be any function of $\delta$ that tends to zero as $\delta$ tends to zero.
Then $\norm{g}_\infty |P^*|^{1/p} \le C_\delta\norm{f}_p$, as required.
Choosing $\delta$ to be sufficiently small as a function of the quantity $\eps$ in the statement of
Proposition~\ref{prop:spatiallocalization}, all conclusions have been established.
\end{proof}

\section{Frequency localization and tightness}

The Fourier transform of any near-extremizer is localized in essentially the same sense as the function itself.
\begin{lemma}[Frequency localization] \label{lemma:frequencylocalization}
Let $d\ge 1$ and let $\Lambda\subset(1,2)$ be a compact set.
For every $\tau>0$ there exists $\delta>0$ with the following property for every $p\in\Lambda$.
Let $0\ne f\in L^p(\reals^d)$ satisfy $\norm{\widehat{f}}_{p'} \ge (1-\delta)\bestA_p^d\norm{f}_p$.
There exist an ellipsoid $\scripte'\subset\reals^d$ and a decomposition $\widehat{f} = \varphi+\psi$ such that 
\begin{gather}
\norm{\psi}_{p'}<\tau
\\ \varphi\prec\scripte' 
\\ \norm{\varphi}_\infty |\scripte'|^{1/p'}\le C_\tau\norm{f}_p.
\end{gather}
\end{lemma}

\begin{proof}
Let $q=p'$.
To simplify notation, assume that $\norm{f}_p=1$.
Consider the function $g(\xi) = \widehat{f}(\xi)|\widehat{f}(\xi)|^{q-2}$,
which belongs to $L^p$ and satisfies \[\norm{g}_p = \norm{\widehat{f}}_q^{q-1}
\in [(1-\delta)^{q-1}\bestA_p^{(q-1)d},\bestA_p^{(q-1)d}].\]
Now
\begin{align*} \norm{\widehat{f}}_q^q = \int g(\xi) \overline{\widehat{f}(\xi)}\,d\xi
 = \int \widehat{g}(-x) \overline{f(x)}\,dx \le \norm{\widehat{g}}_q\norm{f}_p, \end{align*}
so
\[
\norm{\widehat{g}}_{p'} 
\ge 
\norm{\widehat{f}}_{p'}^q\norm{f}_p^{-1}
\ge (1-\delta)^q \bestA_p^{qd}\norm{f}_p^{q-1}
\ge (1-\delta)^q \bestA_p^{qd}\bestA_p^{-(q-1)d}\norm{g}_p
 =  (1-\delta)^q \bestA_p^d\norm{g}_p.
\]
Thus $g$ also nearly extremizes the Hausdorff-Young inequality with exponent $p$.
Applying Proposition~\ref{prop:spatiallocalization} to $g$ gives all of the required conclusions.
\end{proof}

For any ellipsoid $\scripte\subset\reals^d$ there exists $\scriptt_\scripte\in\aff(d)$
that maps the unit ball of $\reals^d$ bijectively onto $\scripte$. 
Define the distance function $\rho_\scripte(x,y) = |\scriptt_\scripte^{-1}(x-y)|$.
This distance is uniquely defined, even though $\scriptt_\scripte$ is not.
The distance $\rho_\scripte(x,S)$ between a point and a set is defined in
terms of distances between points, in the usual way.

Let $\Theta$ be an auxiliary function satisfying $\lim_{t\to\infty}\Theta(t)=0$.
\begin{definition}
Let $\scripte\subset\reals^d$ be an ellipsoid, and $\delta>0$.
$0 \ne f\in L^p$ is $\delta$--normalized in $L^p$ with respect to $\Theta,\scripte$ if  for all $t\in[0,\infty)$,
\begin{align*}
\int_{|f| \ge t|\scripte|^{-1/p}\norm{f}_p} |f|^p &\le (\Theta(t)+\delta)\norm{f}_p^p
\\
\int_{\rho_\scripte(x,\scripte) \ge t} |f|^p &\le (\Theta(t)+\delta)\norm{f}_p^p.
\end{align*}
\end{definition}

In the next lemma, if $E\subset\reals^d$ is an ellipsoid then its polar is 
\[ E^*=\set{y: |\langle x,y\rangle|\le 1\ \text{ for every $x\in E$}}\]
where $\langle \cdot,\cdot\rangle$ denotes the Euclidean inner product.
For $s\in\reals^+$, $sE^*=\set{Ay: y\in E^*}$.

\begin{lemma}[Tightness] \label{lemma:tightness}
Let $d\ge 1$ and let $\Lambda\subset(1,2)$ be a compact set. 
Let $\Theta$ be as above.
There exist $\delta_0>0$ and $A<\infty$ with the following property for every $p\in\Lambda$.
Set $q=p'$.  
Let $\scripte,\tilde\scripte\subset\reals^d$ be ellipsoids and let $u,v\in\reals^d$.
Suppose that $0\ne f\in L^p(\reals^d)$ is $\delta_0$--normalized in $L^p$ with respect to $\Theta,\scripte+u$. 
Suppose further that
$\widehat{f}$ is $\delta_0$--normalized in $L^q$ with respect to $\Theta,\tilde\scripte+v$. 
Then
\begin{equation} \scripte \subset A\tilde\scripte^* \ \text{ and } \ \tilde\scripte\subset A\scripte^*. \end{equation}
\end{lemma}
For $d=1$, the conclusion is simply that $|\scripte|\cdot|\tilde\scripte|\le A'<\infty$.
Conversely, it is a simple consequence of the uncertainty principle that 
this product $|\scripte|\cdot |\tilde\scripte|$ is bounded below
by some positive constant that depends only on $\delta_0,\Theta$.

\begin{proof}
In this proof, by a constant we mean a finite quantity that depends only on $d,\Lambda,\Theta$.
We assume without loss of generality that $\norm{f}_p=1$.

We claim first that
\begin{equation}\label{eq:ellipsoidproductupper} |\scripte|\cdot |\tilde\scripte|\le C \end{equation}
where $C<\infty$ is a constant.
By making an affine change of variables and multiplying $f$ by a scalar, 
we may assume without loss of generality 
$|\scripte|=1$, without altering the assumption that $\norm{f}_p=1$. 
Let $\eta>0$ be small. If $\delta_0$ is sufficiently small, depending on $\eta,d$, then by definition 
of normalization, there exists a decomposition
$f=g+h$ with $\norm{h}_p\le\eta$, $\norm{g}_p\le\norm{f}_p=1$, and $\norm{g}_2\le C<\infty$.
Here $C\in\reals^+$ is a positive constant, provided that $\eta$ is chosen to be sufficiently small.

The hypothesis that $\widehat{f}$ is normalized in $L^q$ guarantees that there exists 
a set $S\subset\reals^d$ satisfying $|S|\ge c'|\tilde\scripte|$ such that $|\widehat{f}(x)| \ge c'|\tilde\scripte|^{-1/q}$ 
for all $x\in S$.
If $\eta$ is sufficiently small then 
it follows from Chebyshev's inequality that for any function $H$ satisfying $\norm{H}_q\le\eta$,
there exists a measurable set $\tilde S_H\subset S$ such that $|\tilde S_H| \ge \tfrac12|S| \ge \tfrac12 c'|\tilde\scripte|$
and
\[|\widehat{f}(\xi)-H(\xi)| \ge \tfrac12 c'|\tilde\scripte|^{-1/q} \ \text{ for all $\xi\in\tilde S_H$}\]

Apply this to $H=\widehat{h}$, where $h$ is as above. Then 
\[\norm{\widehat{f}-H}_2 =\norm{\widehat{g}}_2 = \norm{g}_2\le C.\]
On the other hand,
\begin{align*}
\norm{\widehat{f}-H}_2
\ge \tfrac 12 c'|\tilde\scripte|^{-1/q} |\tilde S_H|^{1/2} 
\ge \tfrac12 c'|\tilde\scripte|^{\gamma}
\end{align*}
where $\gamma= \tfrac12-\tfrac1q$ is strictly positive since $q>2$.
Thus $|\tilde\scripte|^\gamma$ is majorized by a finite constant,
completing the proof of \eqref{eq:ellipsoidproductupper}.

Via an affine change of variables we can reduce to the case in which $\tilde\scripte$ is the
ball of radius $1$ centered at $0$. By \eqref{eq:ellipsoidproductupper}, $|\scripte|$ is then bounded above.
By applying a rotation and translation of $\reals^d$ we can further reduce to the case
in which \[\scripte=\{x\in\reals^d: \sum_{j=1}^d s_j^{-2}x_j^2\le 1\},\] where each $s_j\in\reals^+$.
Since $|\scripte| = c_d \prod_j s_j$, this product is bounded above. 
Without loss of generality suppose that $s_1=\min_j s_j$. 
We claim that $s_1$ is bounded below by a strictly positive constant.
This implies that $\scripte\subset A\tilde\scripte^*$ and $\tilde\scripte\subset A\scripte^*$ for some constant $A$,
completing the proof of the lemma.

Let $\eps>0$. Define $\scripte_r=\{x: \sum_j s_j^{-2} x_j^2\le r^2\}$. 
If $\delta_0$ is sufficiently small 
then it is possible to decompose the $\delta_0$--normalized  function $f$
as $f=g+h$ where $\norm{h}_p<\eps$, $g$ is supported in $\scripte_r$,
and $\norm{g}_\infty |\scripte_r|^{1/p}\le C_\eps$, 
where $r$ is bounded above
by a quantity that depends only on $\eps,\Lambda,d,\Theta$ provided that $\delta$ is sufficiently small.
We have
\[ \norm{x_1 g}_p \le rs_1 \norm{g}_p\le Crs_1.\]
Therefore
\begin{equation}\label{eq:rs1bound} \norm{\frac{\partial \widehat{g}}{\partial \xi_1}}_q \le Crs_1.  \end{equation}
Because $\widehat{f}$ is normalized in $L^q$ with respect to the unit ball $\tilde\scripte$,
there exists a constant $A_0\in\reals^+$ such that
$\norm{\widehat{f}}_{L^q(A_0\tilde\scripte)} \ge \tfrac12 \bestA_p^d \norm{f}_p\ge c>0$. 
Therefore if $\eps$ is a sufficiently small constant then
\begin{equation}\label{eq:rs2bound} \norm{\widehat{g}}_{L^q(A_0\tilde\scripte)} 
\ge \norm{\widehat{f}}_{L^q(A_0\tilde\scripte)}-\norm{\widehat{h}}_q 
\ge \norm{\widehat{f}}_{L^q(A_0\tilde\scripte)}-C\norm{h}_p 
\ge c- C\eps=c'>0.  \end{equation}

Bounds $\int_{|\xi|\le A_0}|h(\xi)|^q\,d\xi \ge c'>0$
and
$\norm{\frac{\partial h}{\partial \xi_1}}_q \le \rho$
imply $\norm{h}_q\ge R$ where $R\to\infty$ as $\rho\to 0$ so long as $c',A_0$ remain fixed.
In \eqref{eq:rs1bound} and \eqref{eq:rs2bound}, 
we have shown that $h=\widehat{g}$ satisfies these conditions
with $\rho=Crs_1$. Since $r$ is bounded above, $s_1$ must be bounded below.
\end{proof}

\section{Proof of Proposition~\ref{prop:submain}}

\begin{lemma} \label{lemma:tightnessapplied}
Let $d\ge 1$ and let $\Lambda\subset(1,2)$ be a compact set.
For every $\eps>0$ there exists $\delta>0$ with the following property for every  $p\in\Lambda$.
Let $0\ne f\in L^p(\reals^d)$ satisfy $\norm{\widehat{f}}_{p'} \ge (1-\delta)\bestA_p^d\norm{f}_p$.
There exist ellipsoids $\scripte,\scriptf\subset\reals^d$ centered at $0$,
vectors $u,v\in\reals^d$, and decompositions $f = \varphi+\psi$ 
and $\widehat{f} = \tilde\varphi+\tilde\psi$ such that 
\begin{gather}
\\ \varphi\prec \scripte+u
\ \text{ and } 
\tilde \varphi\prec \scriptf +v
\\ \norm{\varphi}_\infty |\scripte|^{1/p}\le C_\eps\norm{f}_p
\ \text{ and } 
 \norm{\tilde\varphi}_\infty |\scriptf|^{1/p}\le C_\eps\norm{f}_p
\\ \norm{\psi}_p<\eps \ \text{ and } \norm{\tilde \psi}_{p'}<\eps
\\ |\scripte|\cdot|\scriptf| \le C_\eps
\\ \scripte\subset C_\eps\scriptf^* \ \text{ and } \ \scriptf\subset C_\eps\scripte^*.
\end{gather}
\end{lemma}

\begin{proof}
This is a direct combination of Proposition~\ref{prop:spatiallocalization},
Lemma~\ref{lemma:frequencylocalization}, and Lemma~\ref{lemma:tightness}. 
\end{proof}

\begin{lemma} [Precompactness after renormalization]
Let $d\ge 1$ and $p\in(1,2)$. 
Let $(f_n)$ be a sequence of functions
in $L^p(\reals^d)$ satisfying $\norm{f_n}_p=1$ and $\lim_{n\to\infty} \norm{\widehat{f_n}}_{p'} =  \bestA_p^d$.
Then there exist a sequence of positive real numbers $\lambda_n$ 
and sequences of vectors $a_n,\xi_n\in\reals^d$ such that the sequence of functions
\begin{equation} \label{renormalize} g_n(x) = e^{i\xi_n \cdot x} \lambda_n^{1/p}f_n(\lambda_n (x-a_n))\end{equation}
is precompact in $L^p(\reals^d)$.
\end{lemma}

\begin{proof}
For each sufficiently large $n$ there exist ellipsoids $\scripte_n,\scriptf_n$, vectors $u_n,v_n$, 
and associated decompositions of $f_n$
such that the conclusions of Lemma~\ref{lemma:tightnessapplied} hold with $\eps=\tfrac14$.
By replacing $f_n$ by $e^{-iv_n\cdot x}f_n(x+u_n)$
we may reduce to the case $u_n=v_n=0$.
By composing $f_n$ with an element of $\Gl(d)$ and multiplying by the appropriate power of the absolute
value of its determinant, we may reduce to the case in which $\scripte_n$ is the unit ball of $\reals^d$.

Continue to denote these modified functions by $f_n$. 
Denote by $\varphi_{n},\psi_{n},\tilde\varphi_{n},\tilde\psi_{n}$ the associated functions in
the decompositions of $f,\widehat{f}$ respectively, that are provided by Lemma~\ref{lemma:tightnessapplied}. 
For each $\eps>0$, there exists $N<\infty$ such that for each $n\ge N$,
Lemma~\ref{lemma:tightnessapplied} associates to $f_n$ ellipsoids $\scripte_{n,\eps}$ and $\scriptf_{n,\eps}$
along with appropriate decompositions $f_n = \varphi_{n,\eps}+\psi_{n,\eps}$, 
$\widehat{f_n} = \tilde\varphi_{n,\eps}+\tilde\psi_{n,\eps}$. 

Symmetries of the inequality have been exploited to normalize so that $\scripte_n,\scriptf_n$ 
are balls centered at the origin with radii comparable to $1$. We next claim that this ensures 
corresponding normalizations for $\scripte_{n,\eps},\scriptf_{n,\eps}$; $\eps$--dependent symmetries
are not needed. 

According to Lemma~\ref{lemma:tightness}, 
\begin{gather*}
\norm{\varphi_{n}-\varphi_{n,\eps}}_p \le \tfrac14+2\eps
\\ \norm{\varphi_{n}}_p\ge\tfrac34
\\ \norm{\varphi_{n}}_\infty|\scripte_n|^{1/p}\le C
\\ \norm{\varphi_{n,\eps}}_\infty|\scripte_{n,\eps}|^{1/p}\le C_\eps
\end{gather*}
provided that $n$ is sufficiently large and $\eps$ is sufficiently small.
A consequence is that $\norm{\varphi_{n,\eps}}_\infty \ge c\norm{\varphi_n}_\infty$,
where $c>0$ is independent of $n,\eps$.

For each $x$, each of $\varphi_n(x),\varphi_{n,\eps}(x)$ is equal either to $f(x)$, or to $0$.
From these inequalities and this fact, along with the support relations $\varphi_n\prec\scripte_n$
and $\varphi_{n,\eps}\prec \scripte_{n,\eps}$,
it follows that $|\scripte_{n,\eps}|\le C_\eps|\scripte_n|$.
Likewise $|\scriptf_{n,\eps}|\le C_\eps|\scriptf_n|$.
Moreover, since $\varphi_{n}$ is supported on $\scripte_n$ while $\varphi_{n,\eps}$
is supported on $\scripte_{n,\eps}$, the intersection $\scripte_n\cap\scripte_{n,\eps}$
must be nonempty.
Moreover,
\[ |\scripte_{n,\eps}|^{1/p} \le C_\eps \norm{\varphi_{n,\eps}}_\infty^{-1}
\le C_\eps \norm{\varphi_{n}}_\infty^{-1} \le C_\eps |\scripte_n|^{1/p}\norm{\varphi_n}_p^{-1}
\le C_\eps |\scripte_n|^{1/p}.  \]
Likewise, $\scriptf_{n,\eps}$ must intersect $\scriptf_n$,
and $|\scriptf_{n,\eps}|\le C_\eps|\scriptf_n|$.
Since $\scripte_n,\scriptf_n$ are contained in Euclidean balls centered at $0$ that
are independent of $n$, and since $\scripte_{n,\eps},\scriptf_{n,\eps}$ are convex sets,
it follows from these volume bounds together with the nonempty intersection property
that $\scripte_{n,\eps},\scriptf_{n,\eps}$ are contained in balls centered at $0$,
whose radii depend only on $\eps$.

Define the norm
\begin{equation} \norm{f}_{\scriptf L^q} = \norm{\widehat{f}}_{L^q}.\end{equation}
The sequence $(f_n)$ (as modified in the first paragraph of the proof, above)
is precompact with respect to the $\scriptf L^q$ norm. Indeed,
for any $\eps>0$,
$\norm{\nabla \widehat{\varphi_{n,\eps}}}_q\le C_\eps$
since $\norm{|x|\varphi_{n,\eps}(x)}_p$ is majorized by a contant multiple of the diameter
of $\scripte_{n,\eps}$.
Since $\norm{\widehat{f_n}}_q$ is uniformly bounded and $\norm{\widehat{\psi_{n,\eps}}}_q\to 0$
as $n\to\infty$,
it follows that the sequence $(\widehat{f_n})$ is precompact in $L^q$ norm,
on any fixed bounded subset of $\reals^d$. Since $\scriptf_{n,\eps}$ is contained
in a ball independent of $n$, and since 
$\int_{\xi\notin\scriptf_{n,\eps}}|\widehat{f_n}(\xi)|^q\,d\xi\to 0$ as $n\to\infty$
for each fixed $\eps>0$,
the sequence $\widehat{f_n}$ is precompact in $L^q(\reals^d)$.

This reasoning cannot be directly reversed to establish precompactness in $L^p$,
because no bound for the $L^p$ norm of the inverse Fourier transform of $\tilde\psi_{n,\eps}$ is immediately available.
To circumvent this difficulty, consider any subsequence that converges in $\scriptf L^q$. Rename this
subsequence to be $(g_n: n\in\naturals)$. 
Then $\widehat{g_n}\to h$ in $L^q$ for some $h\in L^q(\reals^d)$.
Since $\norm{g_n}_p=1$ and $\norm{\widehat{g_n}}_q\to \bestA_p^d$, $\norm{h}_{L^q}=\bestA_p^d$.
Moreover, $h=\widehat{G}$ for some tempered distribution $G$.

Since the unit ball of $L^p$ is weak star compact and $L^p$ is separable, 
some subsequence of $(g_n)$ must converge in 
the weak star topology of $L^p(\reals^d)$, to some limit $g\in L^p$.
Continue to denote this sub-subsequence by $(g_n)$.
Then \[\norm{g}_p\le \liminf_{n\to\infty} \norm{g_n}_p =1.\] 

In the topology of the space $\scripts'$ of tempered distributions on $\reals^d$, 
$g_n\to g$. Therefore $\widehat{g_n}\to\widehat{g}$ in $\scripts'$. 
Since $\widehat{g_n}\to h$ in $L^q$ norm, $\widehat{g}=h$.
Therefore $\norm{\widehat{g}}_q = \bestA_p^d$.
By the Hausdorff-Young inequality, $\norm{\widehat{g}}_q\le \bestA_p^d\norm{g}_p$.
Therefore $\norm{g}_p\ge 1$, and hence $\norm{g}_p=1$.

We have shown that $g_n\to g$ in the weak star topology for $L^p$, and that $\norm{g_n}_p\to\norm{g}_p$. 
These two properties alone
imply, for $p\in(1,\infty)$, that $g_n\to g$ in the $L^p$ norm. 
\end{proof}

\begin{remark}
An alternative way to complete the proof, once precompactness in the $\scriptf L^q$ norm
has been established, is to consider
the sequence of convolutions $(g_n*g_n)$. This sequence converges in $L^r$ norm, where $r^{-1} = 2 p^{-1}-1$.
Indeed, $r\in (2,\infty)$ so by the Hausdorff-Young inequality, the inverse Fourier transform
is bounded from $L^{r'}$ to $L^r$; 
\begin{align*} \norm{g_n*g_n-g_m*g_m}_r 
&\le \bestA_{r'}^d \norm{\widehat{g_n}^2-\widehat{g_m}^2}_{r'}
\\&\le \bestA_{r'}^d \norm{\widehat{g_n}-\widehat{g_m}}_{q} \big(\norm{\widehat{g_n}}_q + \norm{\widehat{g_m}}_q\big)
\\& \le 2\bestA_{r'}^d \bestA_p^d \norm{\widehat{g_n}-\widehat{g_m}}_{q}. 
\end{align*}
The limit of this sequence must equal the inverse Fourier transform of $\widehat{g}\,^2$.
Since $\widehat{g}\,^2$ is a Gaussian, its inverse Fourier transform  $g*g$ satisfies
\begin{equation*}
\norm{g*g}_r = \bestA_{r'}^d\norm{\widehat{g}\,^2}_{r'}
= \bestA_{r'}^d\norm{\widehat{g}}_{q}^2
= \bestA_{r'}^d\bestA_p^{2d}\norm{g}_p^2.
\end{equation*}
The product
$\bestA_{r'}^d\bestA_p^{2d}\norm{g}_p^2$
is the optimal constant in Young's convolution inequality
$\norm{u*v}_r\le C\norm{u}_p\norm{v}_p$.
Therefore the sequence of ordered pairs $(g_n,g_n)$ is an extremizing sequence for Young's inequality,
that is, the limit as $n\to\infty$ of
$\norm{g_n*g_n}_r/\norm{g_n}_p^2$ 
equals the optimal constant in that inequality. 
The main theorem of \cite{christyoungest} asserts that any such extremizing
sequence for Young's convolution inequality, normalized so that $\norm{g_n}_p=1$,
is precompact in $L^p$, modulo the action of the group generated
by affine automorphisms of $\reals^d$ and modulations.
The proof of that theorem is not short, so this alternative is not preferable
to the self-contained argument indicated above.
\qed
\end{remark}

\section{Proofs of auxiliary results on multiprogressions} \label{section:proofs}

\begin{proof}[Proof of Lemma~\ref{lemma:properversusgeneral}]
Let $Q=[0,\eta]^d+\eta Q'$ where $Q'\subset\integers^d$ is a multiprogression
of rank $r$. Then $\sigma(Q) = \eta^d\sigma(Q')$.
Let $\delta= N^{-1}\eta$ for a large positive integer $N$.
If $N$ is chosen sufficiently large then 
$Q$ is contained in $[-C_\delta,C_d\delta] + \tilde Q'+R$
where $\tilde Q'$ and $R$ are multiprogressions contained in
$\delta\integers^d$, the rank and size of $Q'$ equal the rank and size of $\tilde Q'$ respectively,
and $R$ is a multiprogression of rank $d$ and size comparable to $N^d$.
Then $\sigma(Q'+R) \asymp N^d \sigma(Q')$.

According to Theorem~3.40 of \cite{TaoVu},
there exists a proper multiprogression $P'\subset\delta\integers^d$ of rank $r+d$
that contains $\tilde Q'+R$ and satisfies $\#(P') \le C_r\sigma(Q'+R)\le C'_r N^d \sigma(Q')$. 
Then $P = [-C\delta,C\delta]^d + P'$ is a continuum multiprogression in $\reals^d$ 
whose size is $C\delta^d\sigma(P') \le  C_r\delta^d N^d\sigma(Q') = C_r\sigma(Q)$.
Now $[0,\delta)^d+P'$ is a proper continuum multiprogression, whose
Lebesgue measure is comparable to that of $[-C\delta,C\delta]^d+P'$.
So $[-C\delta,C\delta]^d+P'$ has the required properties.
\end{proof}

To establish Lemma~\ref{lemma:gapd}, we will first prove a discrete analogue. 
\begin{lemma} \label{lemma:gappreparation}
For each $d\ge 1$ and $\rankk\ge 1$ there exists $c>0$ with the following property. 
Let $\delta\in(0,\tfrac12]$ and $\Lambda\in\reals^+$.
Let $P$ be a proper discrete multiprogression in $\integers^d$ of rank $\rankk$ and cardinality $N$.
Then there exists $\scriptt\in\aff(d)$ satisfying   
\begin{align}
&|J(\scriptt)| \ge c_\rankk \delta^{d \rankk} \#(P)^{-d}\Lambda^{d^2}
\\ &|T| \le \Lambda 
\\ &\norm{\scriptt(x)}_{\reals^d/\integers^d}\le \delta \text{ for every $x\in P$}
\end{align}
where $\scriptt(x)=T(x)+v$, $T\in\Gl(d)$, and $v\in\reals^d$.
\end{lemma}
Here and below, $|\cdot|$ is used to indicate both the norm of a vector in $\reals^d$,
and the operator norm of $T\in\Gl(d)$, $|T|=\sup_{|x|\le 1}|T(x)|$. 

\begin{example} \label{example:pq}
The following example shows that Lemma~\ref{lemma:gappreparation} does not extend to unions of multiprogressions.
That is, if $P\cup Q\subset\reals^1$ has large diameter relative to $\#(P)+\#(Q)$ then there need not exist any dilate of $P\cup Q$
that lies close to $\integers$ and still has diameter $\ge 1$.  
It is this phenomenon that makes Lemma~\ref{lemma:NE2} necessary in our construction.

Let $p,q\gg 1$ be distinct and relatively prime natural numbers of comparable magnitudes. 
Let $N=pq$.  Let $P=\set{kp: 0\le k<q}$ and $Q=\set{kq: 0\le k<p}$. Both are subsets of $\set{0,1,\dots,N-1}$.
Their union is sparse; $\#(P\cup Q) = p+q-1$. On the other hand, $P+Q\subset\set{0,1,\dots,2N-2}$ has $pq=N$ elements.
Assume that $\delta<\tfrac12$. If $x$ and $2x$ both satisfy $\norm{z}_{\reals/\integers}<\delta$ then 
$\norm{x}_{\reals/\integers}<\tfrac12\delta$.
If $\norm{\lambda x}_{\reals/\integers}<\delta$ for every $x\in P$ then it follows from this last fact that
$\norm{\lambda p}_{\reals/\integers} < q^{-1}\delta$.
Likewise if $\norm{\lambda x}_{\reals/\integers}<\delta$ for every $x\in Q$ then $\norm{\lambda q}_{\reals/\integers} < p^{-1}\delta$.
Any element $x\in[0,pq)$ can be expressed as $x = ap+bq$ for some integers $|a|<q$ and $|b|<p$.
Therefore \begin{align*} \norm{\lambda x}_{\reals/\integers}
&\le |a|\norm{\lambda p}_{\reals/\integers} + |b|\norm{\lambda q}_{\reals/\integers}
\\&< |a| q^{-1}\delta + |b|p^{-1}\delta 
\\& < 2\delta.
\end{align*}
Since this holds for every integer $x\in[0,pq)$, provided that $2\delta<\tfrac12$
it follows from the same reasoning as above that $\lambda$ must satisfy 
\[\norm{\lambda}_{\reals/\integers} \le 2\delta (pq)^{-1} \le C\delta \#(P\cup Q)^{-2}\ll\delta\#(P\cup Q)^{-1}.\]
\qed
\end{example}


Let $\scriptm_d$ be the set of all real $d\times d$ matrices. 
Identify $\scriptm_d$ with $\reals^{d^2}$ by using the individual matrix entries as coordinates,
and equip $\scriptm_d$ with the associated Lebesgue measure, which we denote by $\mu$.
Regard $\Gl(d)$ as a subset of $\scriptm_d$.

For $T\in \scriptm_d$ we continue to denote the operator norm by $|T| = \sup_{|x|=1} |T(x)|$.
Let $\scriptm_d(\Lambda)$ be the set of all $T\in \scriptm_d$ satisfying $|T|\le \Lambda$.
Then $\mu(\scriptm_d(\Lambda)) = c_d\Lambda^{d^2}$ for a certain constant $c_d\in\reals^+$.

The next lemma will be essential in the proof of Lemma~\ref{lemma:gapd} for $d>1$.
Its proof is deferred to \S\ref{section:determinant}.
\begin{lemma} \label{lemma:determinant}
For each $d\ge 1$ there exist $c,C\in\reals^+$ and $K\in\naturals$ with the following property.
For any Lebesgue measurable set $E\subset \Gl(d)$ satisfying $0<\mu(E)<\infty$ there exist 
$T_1,\dots,T_K\in E$ and coefficients $s_j\in \integers$
satisfying 
\begin{align} &|s_i|\le C, \\& \sum_j s_j=0,
\\& |\det\big(\sum_{j=1}^K s_j T_j\big)| \ge c\mu(E)^{1/d}.\end{align}
\end{lemma}

The form of the dependence of the upper bound on $\mu(E)$ is forced by scaling considerations;
the issue is the existence of a uniform lower bound for all sets $E$ satisfying $\mu(E)=1$.

\begin{proof}[Proof of Lemma~\ref{lemma:gappreparation}]
By replacing $P$ by $P-b$ for suitable $b\in\reals^d$ we may assume that 
\[ P=\big\{ \sum_{j=1}^\rankk  n_{j}v_j: n_j\in [0,N_j) \text{ for all } 1\le j\le \rankk \big\}\]  
where $v_1,\cdots,v_\rankk\in\integers$ and the mapping $(n_1,\dots,n_\rankk)\mapsto \sum_j n_j v_j$
is injective on $\prod_j [0,N_j)$.
Set $N=\#(P)=\prod_j N_j$.

Let $\set{s_j(\omega): 1\le j\le \rankk}$ be independent random variables on an auxiliary probability space $\Omega$,
each of which takes values in $[0,1]^d$, uniformly distributed with respect to Lebesgue measure.
Let 
\begin{equation}
B = \set{(\omega,T)\in\Omega\times\scriptm_d(\Lambda):
\norm{T(v_j)-s_j(\omega)}_{\reals^d/\integers^d} 
\le \tfrac12 N_j^{-1}\delta \text{ for all } 1\le j\le \rankk }.
\end{equation}

By independence, for each $T\in \scriptm_d$, the probability that $(\omega,T)\in B$ is equal to
$\delta^{d \rankk}N^{-d}$  
provided that $\delta N_j^{-1}\le\tfrac12$ for all $1\le j\le \rankk$.
Therefore 
\[\iint_{\Omega\times \scriptm_d(\Lambda)} \one_B(\omega,T) \,\,d\omega\,d\mu(T)
=   \delta^{d \rankk}N^{-d} \mu(\scriptm_d(\Lambda))
= c_d \Lambda^{d^2} \delta^{d \rankk}N^{-d}.\]  
Therefore there exists $\omega_0$ such that 
\[\mu(\set{T\in\scriptm_d(\Lambda): (\omega_0,T)\in B}) \ge c_d \Lambda^{d^2} \delta^{d \rankk}N^{-d}.\]  
By Lemma~\ref{lemma:determinant},
there exist $T_i\in\scriptm_d(\Lambda)$ and $\alpha_i\in\integers$, defined for $1\le i\le K_d$, 
such that $(\omega_0,T_i)\in B$ for each index $i$, and $T = \sum_{i}\alpha_i T_i$ satisfies
\[ |\det(T)| 
\ge c_d \Lambda^{d^2}  \delta^{d \rankk}N^{-d}.\]
Moreover, 
$|T|\le C\Lambda$ where $C$ depends only on $d$ and on $\rankk$, $\sum_i\alpha_i=0$, and $|\alpha_i|\le C_d$.

Consider any index $j\in[1,\rankk]$.
Express \[T_i(v_j)- s_j(\omega_0) = n_{i,j} + \eta_{i,j}\] where $n_{i,j}\in\integers^d$
and $|\eta_{i,j}|= \norm{T_i(v_j)-s_j(\omega_0)}_{\reals^d/\integers^d}$.
Therefore since $\sum_i\alpha_i=0$, for each index $j$ we have
\[ T(v_j) = \sum_i \alpha_i T_i(v_j) = \sum_i \alpha_i (T_i(v_j)-s_j(\omega_0)).\]
Therefore since each $\alpha_i\in\integers$, 
\begin{align*}
\norm{T(v_j)}_{\reals^d/\integers^d}
& = \norm{\sum_i \alpha_i (T_i(v_j)-s_j(\omega_0))}_{\reals^d/\integers^d}
\\& = \norm{\sum_i \alpha_i n_{i,j} + \sum_i \alpha_i\eta_{i,j}}_{\reals^d/\integers^d}
\\& = \norm{\sum_i \alpha_i\eta_{i,j}}_{\reals^d/\integers^d}
\\& \le \sum_i |\alpha_i\eta_{i,j}|
\\& \le KC\max_i|\eta_{i,j}|
\\& = KC\max_i\norm{T_i(v_j)-s_j(\omega_0)}_{\reals^d/\integers^d}
\\& \le C N_j^{-1}\delta
\end{align*}
where $C<\infty$ depends only on $d,\rankk$.

Therefore for any $x =  \sum_{j=1}^\rankk n_j v_j\in P$, because the coefficients $n_j$ are integers,
\begin{align*}
\norm{T(x)}_{\reals^d/\integers^d} 
\le  \sum_j \norm{n_j T(v_j)}_{\reals^d/\integers^d}
\le \sum_j N_j \norm{T(v_j)}_{\reals^d/\integers^d}
\le C\delta.
\end{align*}
\end{proof}

\begin{proof}[Proof of Lemma~~\ref{lemma:gapd}] 
We may assume without loss of generality that $a=0$, and will then choose $v=0$.
The assumption that $|P|=1$ means that $\eta^d N=1$, where $N=\prod_{j=1}^\rankk N_j$.
Set $\Lambda = (2d)^{-1}\delta\eta^{-1}$. This quantity belongs to $(0,\infty)$,
and could be large or small, depending on the relative sizes of $\delta,\eta$.

According to Lemma~\ref{lemma:gappreparation}, there exists $T\in \Gl(d)$ satisfying 
\begin{gather*} |T|\le\Lambda 
\\ |\det(T)| \ge c_{d,\rankk} \delta^{d\rankk} N^{-d}\Lambda^{d^2} \end{gather*}
such that \[\norm{T(\sum_{j=1}^\rankk n_j v_j)}_{\reals^d/\integers^d} \le\tfrac12\delta \]
whenever $0\le n_j<N_j$ for all $j$.
Since $|T|\le (2d)^{-1} \delta\eta^{-1}$,  $|T(x)|\le \tfrac12\delta$ for any $x\in[0,\eta]^d$. 
Therefore $\norm{T(x)}_{\reals^d/\integers^d} \le\delta$ for all $x\in P$.

Since $\eta^d N=|P|=1$, we have
\begin{align*} |\det(T)| &\ge c_{d,\rankk} \delta^{d\rankk} N^{-d}\Lambda^{d^2}
\\&= c \delta^{d\rankk} N^{-d} (\delta\eta^{-1})^{d^2} 
\\&= c \delta^{d^2} \delta^{d\rankk} (\eta^d N)^{-d}
\\&= c \delta^{d^2+d\rankk} |P|^{-d}
\\&= c \delta^{d^2+d\rankk} \end{align*}
where $\in\reals^+$ depends only on $d,\rankk$.
\end{proof}

\begin{proof}[Proof of Lemma~\ref{lemma:largesubsets}]
For any $A\subset G$ define $\sym(A) = \set{g\in G: A+g=A}$, which is a subgroup of $G$.
Kneser's theorem (\cite{TaoVu}, Theorem~5.5)  states that
\[\#(A+B)\, \ge\, \#(A)\,+\,\#(B)\, -\, \#(\sym(A+B))\] for all nonempty sets $A,B\subset G$.

Consider $E_0=E$ and $E_1=E_0-E_0$. By Kneser's theorem,
either $\#(E_1)\ge \tfrac32 \#(E_0)$, or $\#\sym(E_1)\ge \tfrac12 \#(E_1)\ge \tfrac12 \alpha\#(G)$.
In the latter case, $H=\sym(E_1)$ is a subgroup of $G$ with at least $\tfrac12\alpha\#(G)$ elements. 
Moreover, $H\subset E_1=E-E$ since $0\in E-E$ and therefore $0+h\in E_1$ for all $h\in H$.
The group $G$ is a union of cosets of $H$, and $H$ has at most $\#(G)/\#(H)\le 2\alpha^{-1}$ cosets.

In the former case, $\#(E_1)\ge \tfrac32\alpha$.
Form $E_2=E_1-E_1$, and consider the same dichotomy as above. If $\# \sym(E_2)\ge\tfrac12 \#(E_2)
\ge \tfrac34 \#(E_1)$ and the same reasoning as in the preceding paragraph concludes the proof. 
Otherwise $\#(E_2) \ge \tfrac32\#(E_1)\ge (3/2)^2\#(E)$.

Iterating this process at most $C\log(\alpha^{-1})$ times gives the desired conclusion.
Indeed, define $N$ to be the smallest integer such that $(3/2)^N\alpha>1$. 
If the process were to fail to halt with $\#\sym(E_n)\ge \tfrac12 (3/2)^{n-1}\alpha$ for some $n<N$,
then $\#(E_N)\ge (3/2)^N\alpha\#(G) > \#(G)$, a contradiction. 
\end{proof}

\begin{proof}[Proof of Lemma~\ref{lemma:compatible}]
The first conclusion holds for arbitrary measurable sets $P,Q$, because
\begin{align*} \norm{\one_P*\one_Q}_s \le |P|^{1/\rho}|Q|^{1/\rho} \min\Big(\frac{|P|}{|Q|},\frac{|Q|}{|P|}\Big)^\gamma \end{align*}
for certain $\gamma>0$.

It suffices to prove the analogue of the second conclusion for discrete multiprogressions contained in $\integers^d$.
Reduction of the continuum case to this discrete case is straightforward.
By Lemma~\ref{lemma:properversusgeneral}, we may assume without loss of generality that $P,Q$ are proper multiprogressions.

We have already shown that $P,Q$ have comparable cardinalities.
If $\norm{\one_P*\one_Q}_s\ge \tau|P|^{1/\rho}|Q|^{1/\rho}$
then 
\begin{equation} \label{eq:manyreps} \norm{\one_P*\one_Q}_\infty \ge \tau^\gamma |Q|,\end{equation} 
where $\gamma\in\reals^+$ depends only on $\rho$;
otherwise a simple interpolation argument gives an upper bound
$<\tau|P|^{1/\rho}|Q|^{1/\rho}$ for $\norm{\one_P*\one_Q}_s$.
In the reasoning below, we will exploit only this fact \eqref{eq:manyreps},
not the stronger inequality
$\norm{\one_P*\one_Q}_s\ge \tau|P|^{1/\rho}|Q|^{1/\rho}$. This will make possible an induction
on the rank of $Q$.

Let $N=\#(Q)$.
There exist $z$ and $M\ge \tau^\gamma N$ for which there are $M$ pairs $(p_i,q_i)\in P\times Q$ 
satisfying $p_i+q_i=z$. These elements $q_i$ of $Q$ are distinct, since if $p_i+q=p_j+q$ then $p_i=p_j$.  
The set of all these differences $q_i-q_j$ contains at least $M$
distinct elements of $Q-Q$, since $M$ distinct values are obtained by fixing $i$ and varying $j$.
Since $q_i-q_j=p_j-p_i$ for all pairs of indices $i,j$, 
\begin{equation} \#(P-P\,\, \cap\,\, Q-Q)\ge \tau^\gamma N.\end{equation}

Consider the case in which $Q$ has rank equal to $1$.
Without loss of generality represent $Q$ as $\set{nw: 0\le n<N=|Q|}$, where $w\in\integers^d$ and $n\in\integers$.
We have shown that $\#(\set{n\in(-N,N): nw\in P-P}) \ge \tau^{\gamma} N$.

Define $\kappa\in(0,N)$ to be the largest positive integer such that $\kappa w\in P-P$. Then
since $P-P = -(P-P)$ and $Q-Q = -(Q-Q)$, necessarily $\kappa \ge \tfrac12\tau^\gamma N$.
Consider the quotient group $\integers/\kappa\integers$. Denote by $[x]$ the coset
of an element $x\in\integers^d$, and by $[S]$ the image of a subset $S\subset\integers$
under the quotient map $\integers\to \integers/\kappa\integers$. 
In $\integers/\kappa\integers$ consider the subset \[E=\set{[n]\in \integers/\kappa\integers: [nw]\in [(P-P)\cap(Q-Q)]},\]
which contains the image under the quotient map of $\{0,w,2w,\dots,(\kappa-1)w\}\cap(P-P)$
and hence has cardinality at least $\tfrac12\tau^\gamma N$.
Therefore by Lemma~\ref{lemma:largesubsets} 
there exists $\set{y_i: 1\le i\le m}\subset[0,\kappa)$ such that $\integers/\kappa\integers\subset \cup_{i=1}^m 
\big((mE-mE) + [y_i]\big)$, where $m$ depends only on $\tau,\rho$.

Every element of $[0,N)$ can be written as $n+l\kappa$ for some $0\le n<\kappa$ and $0\le l\le C\tau^{-\gamma}$.
Therefore every element $nw$ of $\{0,w,2w,\dots,(N-1)w\}$
can be expressed as an element of $mP-mP + k\kappa w+y_iw$ 
for some integers $k,i$ satisfying $1\le i\le m$ and $|k| \le 2\tau^{-\gamma}+2m$.
Since $\kappa w\in P-P$, any such expression belongs to $m'P-m'P+y_i w$ for some $1\le i\le m$,
where $m' = m+2C\tau^{-\gamma}$.
Thus $Q$ is contained in a union of $m$ translates of $mP-mP$, where $m$ depends only on $\tau,\rho$.

Represent $P$ as the set of all sums $a+\sum_{j=1}^\rankk \nu_j v_j$ with $0\le \nu_j<N_j$, where $\rankk$
is the rank of $P$, $N_j\ge 1$, and these sums are pairwise distinct.
Let $\scriptp$ be the set of elements of $\integers^d$ of the form 
\[\sum_{j=1}^\rankk \nu_j v_j + \sum_{i=1}^m c_i y_i\]
with $-mN_j<\nu_j< mN_j$ and $c_i\in\set{0,1}$.
Then $Q\subset\scriptp$, $\scriptp$ has rank $\le \rankk + m=\rankk+m(p,\tau)$, 
and
\[\#(\scriptp) \le \prod_j 2mN_j\cdot 2^m = 2^{m+\rankk} m^\rankk \#(P)\]
where $m,M$ depend only on $\rho,\tau$.
Since $P+Q\subset\scriptp$, 
this completes the treatment of progressions $Q$ of rank $1$.

For the general case, we proceed by induction on the rank of $Q$.
We know that \eqref{eq:manyreps} holds; $\norm{\one_Q*\one_P}_\infty \ge\tau^\gamma \#(Q)$.
Set 
$Q''=\{nw: 0\le n<N\}'$.
Represent $Q$ as $Q'+Q''$ where $\#(Q)=N\#(Q')$
and $Q'$ is a proper multiprogression of lower rank. Since $Q$ is a union of $N$
translates of $Q'$, 
\[\norm{\one_{Q'}*\one_P}_\infty \ge N^{-1}\norm{\one_Q*\one_P}_\infty\ge N^{-1}\tau^\gamma\#(Q)
=\tau^\gamma \#(Q').\]
In the same way,
\[ \norm{\one_{Q''}*\one_P}_\infty \ge \tau^\gamma \#(Q'').\]
By induction on the rank of $Q$, the first of these last two inequalities implies that
there exists a proper multiprogression
$P'$ satisfying $\#(P')\le C_{\tau,\rho}\#(P)$, of controlled rank, that contains $Q'+P$.

Now
\[ \norm{\one_{Q''}*\one_{P'}}_\infty \ge \norm{\one_{Q''}*\one_P}_\infty\ge \tau^\gamma \#(Q'').\]
Therefore the proof is concluded by invoking the rank one case treated above
to find a suitable proper multiprogression containing $Q''+P'$.
\end{proof}

\begin{proof}[Proof of Lemma~\ref{lemma:interactionimpliesstructure}]
If $p'\ge 4$ then by the Hausdorff-Young inequality, $\norm{\widehat{\varphi}\,\widehat{\psi}}_{p'/2} \le 
\norm{\varphi*\psi}_{s}$ where $s^{-1} = 2p^{-1}-1\in[0,\tfrac12]$. Therefore 
\[\norm{\one_P*\one_Q}_s \ge \lambda|P|^{1/p}|Q|^{1/p}.\]

If $2<p'<4$ then define $\theta\in(0,1)$ by the relation
$1/p' = \tfrac12 \theta + \tfrac14(1-\theta)$. Then
\begin{align*}
\norm{\widehat{\varphi}\,\widehat{\psi}}_{p'/2} 
&\le
\norm{\widehat{\varphi}\,\widehat{\psi}}_{1}^\theta 
\norm{\widehat{\varphi}\,\widehat{\psi}}_{2}^{1-\theta}
\\&\le 
\norm{\widehat{\varphi}}_2^\theta
\norm{\widehat{\psi}}_2^\theta
\norm{\widehat{\varphi}\,\widehat{\psi}}_{2}^{1-\theta}
\\& =  
\norm{{\varphi}}_2^\theta
\norm{{\psi}}_2^\theta
\norm{\varphi*\psi}_{2}^{1-\theta}
\\& \le |P|^{-1/p}|Q|^{-1/p} |P|^{\theta/2} |Q|^{\theta/2} \norm{\one_P*\one_Q}_{2}^{1-\theta}.
\end{align*}
Using the hypothesis $ \norm{\widehat{\varphi}\,\widehat{\psi}}_{p'/2} \ge\lambda$
and the relation $\theta = 3-4p^{-1}$, this leads to
\begin{equation*} \norm{\one_P*\one_Q}_2 \ge \lambda^{1/(1-\theta)} |P|^{3/4}|Q|^{3/4}.  \end{equation*}

To complete the proof, it now suffices to invoke Lemma~\ref{lemma:compatible},
with $\rho=p$ if $p\in(1,\tfrac43]$, and $\rho=\tfrac43$ if $p\in(\tfrac43,2)$.
\end{proof}

\begin{proof}[Proof of Proposition~\ref{prop:FreimanR}]
We will use this fact, which is Corollary~2.24 of \cite{TaoVu}:
Let $A,B\subset\integers^d$.
Suppose that $\#(A)\le K\#(B)\,\le\, K^2\#(A)$ and that $\#(A+B)\le K\#(A)$.
Then \[\#(n_1A-n_2A+n_3B-n_4B) \le K^{C|n|}\,\,\#(A)\] where $n=(n_1,n_2,n_3,n_4)$. 

Let $0<\eps\le\tfrac12 \min(|A|,|B|)$ be given. 
For small $s>0$ consider the set $A(s)\subset\integers^d$ consisting of all $n\in\integers^d$
such that $|A\cap (sn+s\unitQ^d)| \ge (1-\sigma_d)s^d$, where $\sigma_d$ is a sufficiently small constant.
Define $A^\dagger =A \setminus \cup_{n\in A(s)} (sn+s\unitQ^d)$.
If $s$ is chosen to be sufficiently small then $|A^\dagger|<\eps$.
In the same way define $B(s)\subset\integers^d$ and $B^\dagger\subset B$,
choosing $s$ also to be sufficiently small to guarantee that $|B^\dagger|<\eps$.

Then $\#A(s)\approx s^{-d}|A|$, and $\#B(s)\approx s^{-d}|B|$.
Since $|A+B|\ge |(A\setminus A^\dagger)+(B\setminus B^\dagger)| \ge s^d \#(A(s)+B(s))$, 
$\#(A(s)+B(s))\le Cs^{-d}K(|A|+|B|) \le CK(\#A(s)+\#B(s))$.
By Fre{\u\i}man's theorem,  formulated as Theorem~5.32 of \cite{TaoVu}, 
there exists a proper multiprogression ${\mathbf P}\subset 4A(s) \subset \integers^d$ of rank $O_K(1)$ 
such that $A(s)$ is contained in a union of $O_K(1)$ translates of ${\mathbf P}$. Set $P = s{\mathbf P}+s\unitQ^d$. 
Then $A\setminus A^\dagger \subset P$, and $|P|\le O_K(1)|A|$; the factor $O_K(1)$ does not depend on $\eps$.

To remove the small exceptional set $A^\dagger$, let  $Q\subset\reals^d$ be a proper 
continuum multiprogression, associated to $B$ in the same way that $P$ is associated to $A$.  
Set $B^* = B\setminus B^\dagger$. 

Consider a maximal collection $\scripta$ of elements $a_j\in A$ for which the sets $a_j+B^*$ are pairwise disjoint.
These sets $a_j+B^*$ are contained in $A+B$, and each has measure $|B^*|\ge |B|-\eps \ge \tfrac12 |B|$.
Therefore \[|\cup_j (a_j+B^*)| \ge M|B^*|\ge \tfrac12 M |B|\ge \tfrac12 MK^{-1}|A|\]
where $M=|\scripta|\in\naturals\cup\set{\infty}$ is the number of indices $j$. 
Since \[|\cup_j (a_j+B^*)| \le |A+B|\le K|A|,\] $M\le 2K^2|A|$.

If $a\in A$ then by the maximality of $\scripta$, there exists $a_j\in\scripta$ for which $a+B^*$ intersects $a_j+B^*$.
Then $a\in a_j+B^*-B^*$. Therefore 
\[A \subset \cup_{j=1}^M \big(a_j+ B^*-B^*\big).\] 
Since $B^*$ is contained in a union of $O_K(1)$ translates of $Q$,
this shows that $A$ is contained in a union of $O_K(1)$ translates of $Q-Q$.
\end{proof}

\begin{proof}[Proof of Proposition~\ref{prop:BSGR}] 
Begin as in the proof of 
Proposition~\ref{prop:FreimanR}, constructing $s>0$ and sets $\tilde A = A(s),\tilde B=B(s)\subset\integers^d$
with the properties indicated there.
Then there exists a set $\Lambda \subset\tilde A+\tilde B$ satisfying $\#(\Lambda)\ge c_K \#(\tilde A)\#(\tilde B)$
such that \[\#(\set{a+b: (a,b)\in \Lambda}) \le C_K \#(A)+C_K\#(B).\]
The Balog-Szemer\'edi-Gowers theorem, Theorem~2.19 of \cite{TaoVu}, 
guarantees the existence of sets $\tilde A'\subset \tilde A$, $\tilde B'\subset\tilde B$
satisfying \[\#(\tilde A') \ge c_K \#(\tilde A), \ \  \#(\tilde B') \ge c_K \#(\tilde B),
\ \  \#(\tilde A'+\tilde B') \le C_K\#(\tilde A) + C_K\#(\tilde B).\]
The sets $A'=A\cap\big(s\unitQ^d + s\tilde A'\big)$
and $B'=B\cap\big(s\unitQ^d + s\tilde B'\big)$
satisfy \[|A'|=s^d\#(\tilde A')\ge c_K |A| \ \text{ and } \ 
|B'|=s^d\#(\tilde B')\ge c_K |B|,\]
and \[|A'+B'| \le 2^d s^d(\#(\tilde A'+\tilde B')) \le C_K|A|\,+\,C_K|B|.\]
\end{proof}

\section{Proof of Lemma on large determinants} \label{section:determinant}

\begin{proof}[Proof of Lemma~\ref{lemma:determinant}]
It suffices to prove that there exist  $K$, $s=(s_1,\dots,s_K)$, and $\vec{T}=(T_1,\dots,T_K)$
satisfying the first and third conclusions.  Indeed, given $E$,
choose any element $S\in E$, and consider the set $\tilde E = \set{T-S: T\in E}$,
which satisfies $\mu(\tilde E)=\mu(E)$.
If $s$ and $\vec{T}$ satisfy the first and last conclusions for $\tilde E$
then $s'= (s_1,\dots,s_K,-\sum_{1}^K s_i)$ and $(T_1,\dots,T_K,S)$ satisfy all three conclusions for $E$.  

Consider any $A_1,\dots,A_d\in E$.
For $t=(t_1,\dots,t_d)\in\integers^d$ consider
\[P(t) = \det(\sum_{j=1}^d t_j A_j),\] which can be expanded as a polynomial
\begin{equation} \label{eq:Psub0rep} P(t) = \sum_\alpha t^\alpha Q_\alpha(\vec{A})\end{equation}
where the sum is taken over all multi-indices $\alpha=(\alpha_1,\dots,\alpha_d)$ of degree exactly $d$,
and each $Q_\alpha(\vec{A})=Q_\alpha(A_1,\dots,A_d)$ is a homogeneous  polynomial of degree $d$ in the entries
of the matrices $A_j$.

Consider the polynomial in $t$ defined by
\[ P_1(t) =P_1(t_1,t_2,\dots,t_d)= P_0(t)-P_0(0,t_2,t_3,\dots,t_d),\]
which satisfies
\begin{equation} \label{eq:Psub1rep} P_1(t) = \sum_{\alpha_1\ne 0}  t^\alpha Q_\alpha(\vec{A}),  \end{equation}
where the sum is taken over all multi-indices $\alpha= (\alpha_1,\dots,\alpha_d)$ of degree exactly $d$
such that $\alpha_1\ne 0$, and the expressions $Q_\alpha(\vec{A})$ are identical to those in \eqref{eq:Psub0rep}.

Consider next \[ P_2(t) = P_1(t)-P_1(t_1,0,t_3,t_4,\dots,t_d),\] which satisfies
\begin{equation} \label{eq:Psub2rep} P_1(t) = \sum_{\alpha_1,\alpha_2\ne 0}  t^\alpha Q_\alpha(\vec{A}),  \end{equation}
the sum now being taken over all $\alpha$ of degree $d$ for which neither of $\alpha_1,\alpha_2$ vanishes.

Iterating this construction $d$ times produces the polynomial function 
\begin{equation} \label{eq:Psubdrep} P_d(t) = \sum_{\alpha_j\ne 0\,\,\forall j} t^\alpha Q_\alpha(\vec{A})
=\prod_{j=1}^d t_j \cdot Q_{(1,1,\dots,1)}(\vec{A}).  \end{equation}
Moreover, $P_d(t)$ is a linear combination, with integer coefficients, 
of polynomials of the form $\det(\sum_{j=1}^d \tau_j A_j)$ where $\tau$ is related to $t$ by $\tau_j\in\set{t_j,0}$
for each index $j$.

Specialize to $t=(1,1,\dots,1)$ and define $Q = Q_{(1,1,\dots,1)}$.
We have expressed $Q(\vec{A}) $ as a universal $\integers$--linear combination
of finitely many expressions of the form $\det(\sum_{j=1}^d \tau_j A_j)$, where each $\tau_j$ is an element of $\set{0,1}$. 

We argue by contradiction. Let $\eps>0$ be small and suppose there exists $E\subset\scriptm_d$
such that $0<\mu(E)<\infty$ and $|\det(\sum_{j=1}^d \tau_j A_j)|<\eps\mu(E)^{1/d}$ for all $\tau_j\in\set{0,1}$.
We have shown that this implies that $|Q(A_1,\dots,A_d)|<C_d\eps\mu(E)^{1/d}$ for all $A_1,\dots,A_d\in E$.

$Q$ is a multilinear function of $(A_1,\dots,A_d)$:
\begin{multline*} Q(A_1,\dots, A_{k-1},A'_k+A''_k,A_{k+1},\dots)
\\ = Q(A_1,\dots, A_{k-1},A'_k,A_{k+1},\dots)
+Q(A_1,\dots, A_{k-1},A''_k,A_{k+1},\dots) \end{multline*}
for all matrices $A_1,\dots,A_d$ and all $k\in\set{1,2,\dots,d}$.
Indeed, $Q(A_1,\dots,A_d)$ is  a homogeneous polynomial of degree $d$ in the entries of these matrices,
and satisfies $Q(t_1A_1,\dots,t_dA_d) \equiv \prod_j t_j \cdot Q(A_1,\dots,A_d)$.
Therefore $Q$ is a linear combination of monomials, each of which is a product of factors, each of which
is a linear function of the entries of a single matrix, with one factor for each index $j$.

$Q$ satisfies
\begin{equation} Q(A,A,\dots,A) = d!\, \det(A)  \end{equation}
for all $A\in\scriptm_d$. Indeed,
\begin{align*} Q(A,\dots,A)  = \frac{\partial^d}{\partial t_1\cdots\partial t_d} \det(\sum_{j=1}^d t_j A)\Big|_{t=0}
= \frac{\partial^d}{\partial t_1\cdots\partial t_d} (\sum _j t_j)^d \det(A)\Big|_{t=0}  = d!\,\det(A).  \end{align*}

Denote by $\scripte$ the closed convex hull of $E$, that is, the closure of the set of
all finite linear combinations $\sum_{k} s_k A_k$ with $A_k\in E$, $s_k\ge 0$, and $\sum_k s_k=1$.
By the multilinearity of $Q$,
the condition $|Q(A_1,\dots,A_d)|< C_d \eps\mu(E)^{1/d}$ for all $(A_1,\dots,A_d)\in E^d$
implies the same inequality for all $(A_1,\dots,A_d)\in \scripte\times E^{d-1}$.
Repeating this reasoning for  all indices $k\in\set{1,2,\dots,d}$ in succession shows that
\begin{equation} |Q(A_1,\dots,A_d)|< C_d \eps\mu(E)^{1/d} \ \text{ for all $(A_1,\dots,A_d)\in \scripte^d$}.\end{equation}
Specializing to $A_1=\dots=A_d$ gives
\begin{equation} |\det(A)|< C \eps\mu(E)^{1/d} \ \text{ for all $A\in \scripte$}\end{equation}
where $C$ depends only on the dimension $d$.

If $|\det(A)|\le\eps$ for all $A\in E$ then the same holds for all $A\in E\cup -E$. Passing to its convex hull 
as above, we have reduced the case of a general measurable set $E\subset\scriptm_d$ to that of a compact convex set
$E\subset\scriptm_d$ that satisfies $E = -E$.
\end{proof}

\begin{sublemma}\label{sublemma:SL}
For any $d\ge 1$ there exists $c\in\reals^+$ with the following property.
Let $E\subset\scriptm_d$ be a compact convex set satisfying $0<\mu(E)<\infty$ and $E=-E$.
Then there exists $A\in E$ satisfying $|\det(A)|\ge c\mu(E)^{1/d}$.
\end{sublemma}

\begin{proof}
Denote by $e_1,\dots,e_d$ the standard basis vectors for $\reals^d$.
Denote by $v_j(A)$ the $j$--th column of $A\in\scriptm_d$, for $1\le j\le d$.
Define $\pi(A)=v_1(A)\in\reals^d$, the first column of $A$.
Regard $\scriptm_d = \reals^{d^2}$ as $\reals^d\times\reals^{d(d-1)}$,
by assigning to $A$ the coordinates $x = v_1(A)\in\reals^d$ and $y = (v_2(A),\dots,v_d(A))\in\reals^{d(d-1)}$.
For $x\in\reals^d$ let \[E_x=\{y\in\reals^{d(d-1)}: (x,y) \in E\}.\]

We will use the notation $|S|$ to indicate the Lebesgue measures of sets $S$
in $\reals^d$, in $\reals^{d^2} = \scriptm_d$, and in $\reals^{d(d-1)}$.
Regard $\scriptm_d$ as a Hilbert space, using the inner product structure
defined by its identification with $\reals^{d^2}$ using matrix entries,
and using the standard inner product structure for $\reals^{d^2}$.

For $\Phi\in\Gl(d)$ consider $\Phi(E)=\set{\Phi\circ A: A\in E}$.
If $|\det(\Phi)|=1$ then both the range of the function $|\det|:E\to[0,\infty)$ 
and the Lebesgue measure of $E$ are unchanged under replacement of $E$ by $\Phi(E)$.

$E$ contains some closed ellipsoid centered at $0$, whose measure is comparable to that of $E$; replace
$E$ by such an ellipsoid.  
The set $\pi(E)\subset\reals^d$ is itself a closed ellipsoid. Specify $r\in \reals^+$ 
by requiring that a ball of radius $r$ in $\reals^d$ have measure equal to $|\pi(E)|$.
By replacing $E$ by $\Phi(E)$ for an appropriately chosen element $\Phi\in\Gl(d)$
satisfying $\det(\Phi)=1$, we may reduce to the case in which $\pi(E)$ is
a ball of radius $r$ centered at the origin.

For any $j\in\set{1,2,\dots,d}$ and any
$y\in\reals^d\times\reals^{d(d-1)}$, $\det(e_j,y)$ depends only on the orthogonal
projection of $y$ onto a certain subspace $V_j$ of $\reals^{d(d-1)}$ having dimension $(d-1)^2$.
Indeed, $|\det(e_j,y)|$ is equal to the absolute value of an associated cofactor, which depends only
on the entries of a certain $(d-1)\times(d-1)$ minor. These entries provide coordinates for $V_j$.
Moreover, $|\det(e_j,y)| = |\Det(\Pi_{j}(y))|$ where 
$\Pi_j$ denotes this orthogonal projection,
and $\Det(z)$ denotes the $d-1$--dimensional determinant of $z\in\scriptm_{d-1} \leftrightarrow \reals^{(d-1)^2}$. 

We claim that there exists $j\in\set{1,2,\dots,d}$ such that $x=\tfrac12 r e_j$ satisfies
\begin{equation}|\Pi_j(E_x)| \ge c(|E|r^{-d})^{(d-1)/d}.\end{equation}
Granting this, we conclude by induction on the dimension $d$ that there exists
$z=\Pi_j(y)\in \Pi_j(E_x)$ satisfying 
\[|\Det(z)| \ge c |\Pi_j(E_x)|^{1/{d-1}} \ge c(|E|r^{-d})^{1/d}=cr^{-1}|E|^{1/d}.\]
Consequently 
\[ |\det(x,y)| = |\det(\tfrac12 r e_j,y)| = \tfrac12 r \, |\Det(z)| \ge cr\cdot r^{-1}|E|^{1/d} =  c|E|^{1/d},\]
completing the proof of Sublemma~\ref{sublemma:SL}.
\end{proof}

\begin{proof}[Proof of claim]
Denote by $\scriptb,\scriptb'$ the closed balls centered at $0$, with radii equal to $r,r/2$ respectively.
Since $\int_{\pi(E)}|E_x|\,dx = |E|$,
there must exist $\bar x\in \pi(E)$ for which \[|E_{\bar x}| \ge |E|/|\pi(E)| = c|E|r^{-d}.\]
For any $x,x'\in \pi(E)$, \[tE_x+(1-t)E_{x'}\in E_{tx+(1-t)x'}.\] 
Applying this with $x'=\bar x$, we conclude that $E_x$ contains a scaled translate of $E_{\bar x}$
for many points $x\in\scriptb$, with scaling factor $\ge \tfrac18$. Repeating this argument, we conclude that 
$E_x$ contains a scaled translate of $E_{\bar x}$ for each $x\in\scriptb'$,  with scaling factor
bounded below by a positive constant. 

Thus there exists an ellipsoid $E'\subset\reals^{d(d-1)}$ satisfying
\[|E'| \ge c |E|r^{-d}\] such that 
for each $x\in \scriptb'$, $E_x$ contains $E'+\alpha(x)$ for some $\alpha(x)\in\reals^{d(d-1)}$.
Since \[\Pi_x(E_x)\supset\Pi_x(E'+\alpha(x))=\Pi_x(E')+\Pi_x(\alpha(x)),\]
$|\Pi_x(E_x)| \ge |\Pi_x(E')|$ and consequently
it suffices to show that there exists $j\in\sset{1,2,\dots,d}$ satisfying \[|\Pi_{e_j}(E')| \ge c|E'|^{(d-1)/d}.\]

A variant of the Loomis-Whitney inequality states that
for any Lebesgue measurable set $S\subset\reals^{d(d-1)}$,
\begin{equation}\label{eq:loomiswhitney} |S| \le \prod_{j=1}^d |\Pi_{e_j}(S)|^{\frac1{d-1}}.  \end{equation}
Therefore
\begin{equation*} |E'| \le \prod_{j=1}^d |\Pi_{e_j}(E')|^{\frac1{d-1}}\le \max_{j} |\Pi_{e_j}(E')|^{d/(d-1)}.  \end{equation*}
\end{proof}

A generalization of the Loomis-Whitney inequality, which includes the variant used above, is as follows.
Let $(X,\scripta,\lambda)$ be a measure space.
Denote by $\lambda^k$ the product measure $\lambda\times\cdots\times\lambda$ with $k$
factors, on $X^k$.
For each $j\in\{1,2,\dots,d\}$ define $\phi_j:X^d\to X^{d-1}$ by
$\Phi_j(x_1,\dots,x_d) = (x_1,\dots,x_{j-1},x_{j+1},\dots,x_d)$
with the obvious modifications for $j=1$ and $j=d$.
Then for any measurable set $E\subset X^d$, \[\lambda^d(E) \le \prod_{j=1}^d \lambda^{d-1}(\phi_j(E)).\]
Standard proofs of the Loomis-Whitney inequality establish this generalization.
\qed

\section{On second variations for general inequalities}

Let $T$ be a bounded linear operator from $L^p$ to $L^q$, for an arbitrary pair of measure spaces; 
we will specialize later to the case in which $T$ is the Fourier transform and $q=p'$.
Denote by $\norm{T}$ the operator norm, which we assume throughout the discussion to be finite and strictly positive.
For $0\ne h\in L^p$ consider 
\begin{equation*} \Phi(h) = \norm{Th}_q/\norm{h}_p. \end{equation*}
Our present goal is to develop a substitute for the second order Taylor expansion 
of $\Phi$ about an extremizer $F$.
The functional $\Phi$ fails to actually be twice continuously differentiable; 
a quantity such as $|F(x)+f(x)|^p$ cannot be expanded in Taylor series about $|F(x)|$
unless $|f(x)|$ is small relative to $|F(x)|$.

\begin{definition}
Let $p,q,T$ be as above and let $0\ne F\in L^p$. The real quadratic form $\scriptq_F$
is defined formally to be
\begin{multline}
\scriptq_F(h) 
= \tfrac{q-1}2 \norm{TF}_q^{-q} \int (\Re(Th/TF))^2 |TF|^{q}
+ \tfrac{1}{2} \norm{TF}_q^{-q} \int (\Im(Th/TF))^2 |TF|^{q}
\\ - \tfrac{p-1}2 \norm{F}_p^{-p} \int (\Re(h/F))^2 |F|^{p} 
- \tfrac12 \norm{F}_p^{-p} \int (\Im(h/F))^2 |F|^{p}.
\end{multline}
\end{definition}

The quadratic form $\scriptq_F$ is arrived at by consideration of
the formal second order Taylor expansion of $\Phi$ about $F$.
In this definition,
$(\Re(Th/TF))^2|TF|^q$ and $(\Im(Th/TF))^2|TF|^q$ are interpreted as zero
at any point at which $TF$ vanishes, as is reasonable since the net power of $TF$ is $q-2>0$.
Likewise, 
$(\Re(h/F))^2 |F|^p$ and $(\Im(h/F))^2|F|^p$ are interpreted
as zero at any point at which $F$ vanishes. This is not reasonable
for general functions $h$, but is reasonable for functions that
are pointwise $O(|F|)$;
we will utilize $\scriptf_Q(h)$ only for such functions.

\begin{proposition} \label{prop:taylorsubstitute}
For any exponents $p,q$ satisfying $p<2\le q$
there exist constants $c,C,\eta_0,\gamma\in\reals^+$ with the following property.
Let $0<\eta\le\eta_0$ and suppose that $\delta\le \eta^\gamma$.
Let $T:L^p\to L^q$ be a bounded linear operator with operator norm $\norm{T}\in(0,\infty)$.
Let $0\ne F\in L^p$ satisfy $\norm{TF}_q = \norm{T}\cdot\norm{F}_p$.
Suppose that $f\in L^p$,
that $\norm{f}_p\le \delta\norm{F}_p$,
and that $\Re\big(\int f\, \overline{F}|F|^{p-2}\big)=0$. 
Decompose \begin{equation} f=f_\sharp+f_\flat \end{equation}
where
\begin{equation}
f_\sharp(x) = \begin{cases}
f(x)\ &\text{if}\  |f(x)|\le \eta |F(x)|
\\ 0 &\text{otherwise}.
\end{cases}
\end{equation}
Then 
\begin{equation}
\frac{\norm{\Phi(F+f)}_q}{\norm{T}\cdot\norm{F+f}_p}
\le 1 + \scriptq_F(f_\sharp)
+C\eta \norm{f_\sharp}_p^2\norm{F}_p^{-2} 
- c\eta^{2-p} \norm{f_\flat}_p^p \norm{F}_p^{-p}.
\end{equation}
\end{proposition}

Throughout the proof we will assume that $q$ is strictly greater than $2$, but it will be clear
that a simplified analysis applies for $q=2$.
Both of the relations $p<2$ and $2<q$ are important in the proof,
which relies on certain properties of the expression $|1+z|^r$ 
for $r\in\set{p,q}$ and arbitrary $z\in\complex$.  We begin with the case of real $z$.
\begin{lemma} \label{lemma:fka}
There exist $c,C\in\reals^+$ such that for $t\in\reals$ and $0<\eta\le 1$,
\begin{equation} |1+t|^p \ge 
\begin{cases}
1+pt+\tfrac12 p(p-1)t^2 - C\eta t^2\qquad&\text{for $t\in[-\eta,\eta]$}
\\ 1+pt+c \eta^{2-p} |t|^p &\text{otherwise}. \end{cases}
\end{equation}
\end{lemma}

\begin{proof}
Consider the function $\phi(t)=(1+t)^p -1-pt-\tfrac12 p(p-1)t^2$. 
Since $\phi,\phi',\phi''$ vanish at $t=0$, $\phi(t)\ge -Ct^3$ for $t\in[0,1]$,
so $\phi(t)\ge -C\eta t^2$ on $[0,\eta]$ for all $\eta\le 1$.
On the other hand, $\phi$
has third derivative $p(p-1)(p-2)(1+t)^{p-3}$, which is negative on the interval $(-1,0]$.
It satisfies $\phi(0)=\phi'(0)=\phi''(0)=0$.
Therefore $\phi$ is strictly decreasing on $[-1,0]$, hence strictly positive on $[-1,0)$,
giving the required inequality on $[-\eta,0]$ for any $\eta\in(0,1]$.

By continuity, there exists $\eta_0>0$ such that $\phi>0$ on $[-1-\eta_0,0]$. 
Moreover, for $-1-\eta_0\le t\le -\eta$, $\tfrac12 p(p-1)t^2 \ge c \eta^{2-p} |t|^p$. 
Therefore for 
$t\in[-1-\eta_0,-\eta]$,
\[ |1+t|^p \ge 1 + pt + \tfrac12 p(p-1)t^2 \ge 1+pt + c\eta^{2-p}|t|^p.\]

For $t\ge\eta$, the second inequality holds because $(1+t)^p-1-pt$ has positive
second derivative $p(p-1)(1+t)^{p-2} \ge p(p-1)$. 

The function $|1+t|^p$ is decreasing on $(-\infty,-1]$ while $1+pt$ is increasing.
Since the second inequality holds at $-1-\eta_0$, it consequently holds on any compact
subinterval of $(-\infty,-1-\eta_0]$. On the other hand,
for $t<0$ and $|t|$ very large, since $p>1$
\[|1+t|^p \ge \tfrac12 |t|^p > 1+p|t| + \tfrac14 |t|^p \ge 1+pt + \tfrac14 |t|^p.\]
\end{proof}

\begin{lemma} \label{lemma:complexquasitaylor}
Let $p\in(1,2)$ and let $\eta>0$ be sufficienty small.
There exist $c,C\in\reals^+$ such that for any $z\in\complex$, 
\begin{equation}
|1+z|^p \ge \begin{cases} 1 + p\Re(z) + \tfrac12 p(p-1) (\Re(z))^2 + \tfrac{p}2(\Im(z))^2 
- C\eta |z|^2 &\text{ if } |z|\le\eta
\\ 1+p\Re(z)+c \eta^{2-p} |z|^p &\text{ if } |z|>\eta. \end{cases}
\end{equation}
\end{lemma}

\begin{proof} 
Write $z=u+iv\in\reals+i\reals$. Then $|1+z|^p = ((1+u)^2+v^2)^{p/2}$.
\newline (i)\ 
If $|z|\le \tfrac12 $ then Taylor expansion gives
\[ |1+z|^p = 1+pu+\tfrac12 p(p-1)u^2 + \tfrac{p}2 v^2 + O(|z|^3)\]
and the remainder term $O(|z|^3)$ is $\ge -C\eta |z|^2$
when $|z|\le\eta$.
\newline (ii)\ 
There exists $C_0$ such that if $|z|\ge C_0$ then
\[|1+z|^p \ge \tfrac12 |z|^p\ge 1 + p|z| +\tfrac 14|z|^p
\ge 1+ pu + \tfrac14 |z|^p.\]  
\newline (iii)\ 
If $\eta< |z|\le C_0$, $u\ge \eta$, and $|v|\le |u|$ then by Lemma~\ref{lemma:fka}, 
\[|1+z|^p \ge |1+u|^p \ge 1+pu+c\eta^{2-p} |u|^p \ge 1+pu + c'\eta^{2-p} |z|^p.\] 
\newline (iv)\ 
If $\eta< |z| \le C_0$, $-C_0\le u\le\eta$, and $|v|\ge \eta$ then since $p\le 2$,
\begin{multline*}
|1+z|^p = ((1+u)^2+v^2)^{p/2} = |1+u|^p (1+|1+u|^{-2}v^2)^{p/2} 
\\ \ge |1+u|^p (1+c\eta^{2-p} |1+u|^{-p} |v|^p) = |1+u|^p + c\eta^{2-p} |v|^p. \end{multline*}
Indeed, $|1+u|\le 1+C_0$, so $|1+u|^{-2}$ is bounded away from $0$,
so $|1+u|^{-2}v^2$ is $\ge c\eta^2$. 
Therefore $(1+|1+u|^{-2}v^2)^{p/2}\ge 1+c\eta^{2-p}|1+u|^{-p}|v|^p$,
justifying the inequality step in the display.

If $|u|\le\eta$ then Lemma~\ref{lemma:fka} allows us to continue the chain of inequalities
\begin{align*} |1+z|^p &\ge 1+pu + \tfrac12p(p-1) u^2-C\eta u^2 + c\eta^{2-p} |v|^p  
\\ &\ge 1+pu + c\eta^{2-p} |u|^p + c\eta^{2-p} |v|^p 
\\ &\ge 1+pu + c' \eta^{2-p}|z|^p.\end{align*}
If $-C_0\le u\le -\eta$ then again Lemma~\ref{lemma:fka} gives
\[ |1+z|^p 
\ge 1+pu + c\eta^{2-p} |u|^p + c\eta^{2-p} |v|^p 
\ge 1+pu + c' \eta^{2-p}|z|^p.\]
\newline (v)\ 
If  $|v|\le\eta$ and $-C_0\le u\le -1-\eta_0$
then \[|1+z|^p \ge |1+u|^p > 1+pu+c|u|^p \ge 1+pu+c'|z|^p.\]
\newline (vi)\ 
If $\eta<|z|\le C_0$, $|v|\le\eta$, and $-1-\eta_0\le u\le -1+\eta_0 $
then
\begin{align*} |1+z|^p \ge |1+u|^p \ge 1+pu +\tfrac12 p(p-1)u^2 
\ge 1+pu+c|z|^2 
\ge 1+pu +c\eta^{2-p} |z|^p.
\end{align*}
\newline (vii)\ 
If $\eta<|z|\le C_0$, $|v|\le\eta$, and $-1+\eta_0\le u\le 0$ then 
\begin{align*} |1+z|^p &\ge |1+u|^p (1+|1+u|^{-2}v^2)^{p/2}
\\& \ge |1+u|^p (1+\tfrac{p}2 |1+u|^{-2}v^2 -C|1+u|^{-4}v^4)
\\& = |1+u|^p +\tfrac{p}2 |1+u|^{p-2}v^2 -C|1+u|^{p-4} v^4.
\end{align*}
Since $0\le 1+u\le 1$ and $p<2$,
\[ \tfrac{p}2|1+u|^{p-2}v^2 \ge \tfrac{p}2 v^2 \ge \tfrac12 p(p-1)v^2.\]
Provided that $\eta$ is sufficiently small,
this quantity is large relative to $|1+u|^{p-4}v^4$ since $|1+u|$
is bounded away from zero, $p-4<0$, and $|v|\le \eta$.
Consequently
\begin{multline*} |1+z|^p \ge  |1+u|^p +\tfrac12 p(p-1)v^2
\ge 1+pu+\tfrac12 p(p-1)(u^2+v^2) \ge 1+pu +c\eta^{2-p} |z|^p; \end{multline*}
the inequality $|1+u|^p \ge 1+pu+\tfrac12 p(p-1)u^2$ for $u$ in this range was
shown in the  proof of Lemma~\ref{lemma:fka}.
\end{proof}

\begin{lemma} \label{lemma:expansion}
Let $1<p<2 \le q$.
Let $T:L^p\to L^q$ be a bounded linear operator.
Let $0\ne F\in L^p$ be an extremizer for the inequality $\norm{Tg}_q\le \norm{T}\cdot\norm{g}_p$.
There exist $\eta_0,C,c\in\reals^+$ with the following property.
Let $0<\eta,\tau\le\eta_0$.
Let $f\in L^p$ be arbitrary.
Decompose $f=f_\sharp+f_\flat$ where 
$f_\sharp(x)=f(x)$ if $|f(x)| < \eta F(x)$, and $f_\sharp(x)=0$ otherwise.  Then
\begin{multline*}
\norm{F+f}_p^p
\ge  \norm{F}_p^p  + 
p\Re\big(\textstyle\int f\overline{F} |F|^{p-2} \big)
\\ 
+  \tfrac{p(p-1)}{2} \int |F|^{p}(\Re(f_\sharp/F))^2 
+  \tfrac{p}{2} \int |F|^{p}(\Im (f_\sharp/F))^2 
\\
-C\eta \norm{F}_p^{p-2}
\norm{f_\sharp}_p^2 
+ c\eta^{2-p}  
\norm{f_\flat}_p^p 
\end{multline*}
and
\begin{multline*}
\norm{T(F+f)}_q^q \le   \norm{TF}_q^q 
+ q\Re\big(\textstyle\int Tf\overline{TF} |TF|^{q-2} \big)
\\ +  \tfrac{q(q-1)}{2} \int (\Re(Tf_\sharp/TF))^2 |TF|^{q} 
+  \tfrac{q}{2} \int (\Im(Tf_\sharp/TF))^2 |TF|^{q} 
\\ +(C\tau+C\eta)\norm{T}^2\norm{TF}_p^{q-2}\norm{f_\sharp}_p^2 
\\ + C\tau^{2-q}\norm{T}^q \norm{f}_p^q + C\eta^{-1} \norm{T}^2\norm{TF}_q^{q-2}\norm{f_\flat}_p^2.
\end{multline*}
\end{lemma}

Here $|F|^p (\Re(f_\sharp/F))^2$ and $|F|^p (\Im(f_\sharp/F))^2$ 
are again interpreted as zero at any point at which $F$ vanishes. 
Likewise, $|TF|^{q})(\Re(Tf_\sharp/TF))^2$
and $|TF|^{q})(\Im(Tf_\sharp/TF))^2$
are interpreted as zero at any point at which $TF$ vanishes.

\begin{proof}
By viewing $|F+f|^p$ as $|F|^p|1+f/F|^p$ and invoking Lemma~\ref{lemma:expansion} one obtains
\begin{multline} \label{eq:Fplusf}
|F+f|^p \ge |F|^p + p|F|^{p-2}\Re(\overline{F} f) 
+ \tfrac12 p(p-1)|F|^{p}\cdot(\Re(f_\sharp/F))^2 
\\ + \tfrac{p}2 |F|^{p}\cdot(\Im(f_\sharp/F))^2 
+ c\eta^{2-p}  |f_\flat|^p - C\eta |F|^{p-2} |f_\sharp|^2.
\end{multline}
Integration gives the first conclusion.

In establishing the second conclusion, we set $G=TF$ and $g=Tf$ to simplify notation,
retaining the convention that the real and imaginary parts of $g/G$ are interpreted as zero
at points at which $G=0$, when accompanied by a factor of $|G|^{q-2}$.
Let $\tau\in(0,\tfrac12]$.
At points $y$ at which $|g|\le \tau |G|$ and $G\ne 0$,
\begin{align*} |G+g|^q 
&\le |G|^q + q\Re(g\overline{G}|G|^{q-2}) + \tfrac12 q(q-1)(\Re(g/G))^2|G|^{q} 
+ \tfrac{q}2 (\Im(g/G))^2|G|^{q} +C|g|^3|G|^{q-3}.
\end{align*}
At points at which $|g|>\tau|G|$, one has the bounds
$|g|\cdot|G|^{q-1} \le \tau^{-(q-1)}|g|^q$
and $|g|^2\cdot|G|^{q-2} \le \tau^{-(q-2)}|g|^q$,
and consequently if $C$ is chosen to be sufficiently large then
\begin{multline*} |G+g|^q \le |G|^q 
+ q\Re(g\,\cdot\,\overline{G}|G|^{q-2}) 
\\ + \tfrac12 q(q-1)(\Re(g/G))^2|G|^{q} 
+ \tfrac{q}2 (\Im(g/G))^2|G|^{q} + C\tau^{2-q} |g|^q \end{multline*}
simply because both the left-hand side of the inequality, and all terms on the right, are $O(\tau^{2-q}|g|^q)$.
Thus uniformly at almost all points of the measure space on which $G,g$ are defined,
\begin{multline*} |G+g|^q \le |G|^q 
+ q\Re(g\,\cdot\,\overline{G}|G|^{q-2}) 
\\ + \tfrac12 q(q-1)(\Re(g/G))^2 |G|^q + \tfrac{q}2 (\Im(g/G))^2|G|^q
\\ + C\tau |g|^2|G|^{q-2} + C\tau^{2-q} |g|^q.  \end{multline*}

Integrating the above upper bound for $|G+g|^q$ and invoking H\"older's inequality therefore gives 
\begin{multline*}
\norm{G+g}_q^q \le   \norm{G}_q^q 
+ q\Re(\int g\,\cdot\,\overline{G}|G|^{q-2}) 
 +  \tfrac{q(q-1)}{2} \int (\Re(g/G))^2 |G|^{q} 
\\
 +  \tfrac{q}{2} \int (\Im(g/G))^2 |G|^{q} 
+C\tau\norm{G}_q^{q-2}\norm{g}_q^2 +C\tau^{2-q} \norm{g}_q^q.
\end{multline*}
Now
\begin{align*} (\Re(Tf/TF))^2 &\le (1+\eta) (\Re(Tf_\sharp/TF))^2 + (1+\eta^{-1})(\Re(Tf_\flat/TF))^2
\\& \le (1+\eta) (\Re(Tf_\sharp/TF))^2 + 2\eta^{-1}|Tf_\flat|^2|TF|^{-2} \end{align*}
and a corresponding inequality holds for the imaginary parts. Substituting this into the preceding
inequality gives
\begin{align*}
\norm{T(F+f)}_q^q &\le   \norm{TF}_q^q 
+ q\Re(\int Tf\,\cdot\,\overline{TF}|TF|^{q-2}) 
\\&\qquad
 +  \tfrac{q(q-1)}{2} \int (\Re(Tf_\sharp/TF))^2 |TF|^{q} 
 +  \tfrac{q}{2} \int (\Im(Tf_\sharp/TF))^2 |TF|^{q} 
\\ & \qquad \qquad +C\tau\norm{TF}_q^{q-2}\norm{Tf}_q^2 +C\tau^{2-q}\norm{Tf}_q^q
+ C\eta^{-1} \norm{TF}_q^{q-2}\norm{T}_p^{2}\norm{f_\flat}_p^2.
\end{align*}
\end{proof}

\begin{proof}[Proof of Proposition~\ref{prop:taylorsubstitute}]
Because $F$ extremizes the ratio $\Phi$, the first variation of $\Phi$ about $F$ must vanish. 
For any $h\in L^p$, both $\norm{F+th}_p$ and $\norm{TF+tTh}_q$ are differentiable
functions of $t\in\reals$.
Therefore $\frac{d}{dt}\Phi(F+th)\big|_{t=0}=0$. Calculation of this derivative gives 
\begin{equation} \norm{F}_p^p\, \Re\big(\textstyle\int Th\,\cdot\, \overline{TF} |TF|^{q-2} \big)
- \norm{TF}_q^q\, \Re\big(\textstyle\int h\cdot \overline{F} |F|^{p-2} \big)=0. \end{equation}
Therefore \begin{equation} \Re\big(\textstyle\int h\overline{F} |F|^{p-2} \big)=0\ \  \Longrightarrow\  \ 
\Re\big(\textstyle\int Th\,\cdot\, \overline{TF} |TF|^{q-2} \big)=0.\end{equation}

Consider the ratio $\Phi(F+f) = \norm{T(F+f)}_q/\norm{F+f}_p$.
Given $\eta>0$, decompose $f$ as $f_\sharp+f_\flat$ in the manner specified 
in the statement of Proposition~\ref{prop:taylorsubstitute}.
Invoke Lemma~\ref{lemma:expansion}, with $\tau$ chosen to equal $\eta$, to obtain an upper bound for the numerator
of this ratio, and a lower bound for its denominator, taking  a $q$-th root in the numerator 
and a $p$-th root in the denominator. 
The terms proportional to $\Re(\int Tf\cdot \overline{TF}|TF|^{q-2})$
in the numerator and $\Re(\int f\cdot \overline{F}|F|^{p-2})$ in the denominator vanish.
Factor $\norm{TF}_q$ from the numerator, and $\norm{F}_p$ from the denominator;
their ratio gives a factor of $\norm{T}$.
Apply binomial expansions to the remaining numerator and denominator and simplify to obtain 
\begin{multline} \label{eq:junkterms}
\norm{T}^{-1} \Phi(F+f) \le  1 + \scriptq_F(f_\sharp)
+C\eta  \norm{F}_p^{-2} \norm{f}_p^2 
\\ 
+ C\eta^{2-q} \norm{f}_p^q\norm{F}_p^{-q}
- c \eta^{2-p}   \norm{F}_p^{-p} \norm{f_\flat}_p^p 
+ C\eta^{-1}  \norm{F}_p^{-2}\norm{f_\flat}_p^2.
\end{multline}
Provided that $\norm{f_\flat}_p/\norm{F}_p<\delta$,
the sum of the second and third terms in the second line of \eqref{eq:junkterms} can be rewritten as
\begin{align*} \norm{F}_p^{-p}\norm{f_\flat}_p^p &\big(-c\eta^{2-p} + C\eta^{-1} \big(\norm{f_\flat}_p/\norm{F}_p\big)^{2-p} \big) 
\\&\le \norm{F}_p^{-p}\norm{f_\flat}_p^p \big(-c\eta^{2-p} + C\eta^{-1} \big(\norm{f}_p/\norm{F}_p\big)^{2-p} \big) 
\\& \le -\norm{F}_p^{-p}\norm{f_\flat}_p^p \big(c\eta^{2-p} - C\delta^{2-p}\eta^{-1} \big)  \end{align*}
since $\norm{f_\flat}_p\le\norm{f}_p\le\delta\norm{F}_p$.
The first term in the second line of \eqref{eq:junkterms} is
\begin{align*} C\eta^{2-q} \norm{f}_p^q\norm{F}_p^{-q}
\le C\delta^{q-2}\eta^{2-q} \norm{f}_p^2\norm{F}_p^{-2}.  \end{align*}
Thus
\begin{multline} 
\norm{T}^{-1} \Phi(F+f) \le  1 + \scriptq_F(f_\sharp)
\\
+C\big(\eta + (\delta/\eta)^{q-2}\big)  \norm{F}_p^{-2} \norm{f}_p^2 
- \big(c\eta^{2-p} - C\delta^{2-p}\eta^{-1} \big)  \norm{F}_p^{-p}\norm{f_\flat}_p^p.
\end{multline}
Since $p<2$ and $q>2$, the final line is
\[ \le 2C\eta   \norm{F}_p^{-2} \norm{f}_p^2 - \tfrac12 c\eta^{2-p} \norm{F}_p^{-p}\norm{f_\flat}_p^p\]
provided that $\delta\le c_0 \eta^\gamma$ where
\[ \gamma = \max\big(\frac{q-1}{q-2},\,\frac{3-p}{2-p}\big)\]
and $c_0$ is a sufficiently small constant.
Finally,
\begin{align*} \eta   \norm{F}_p^{-2} \norm{f}_p^2 
&\le C\eta   \norm{F}_p^{-2} \norm{f_\sharp}_p^2 + C\eta \norm{F}_p^{-2}\norm{f_\flat}_p^2
\\ &\le C\eta   \norm{F}_p^{-2} \norm{f_\sharp}_p^2 + C\delta^{2-p}\eta \norm{F}_p^{-p}\norm{f_\flat}_p^p.\end{align*}
The second term can be absorbed into the term $- c\eta^{2-p} \norm{F}_p^{-p}\norm{f_\flat}_p^p$;
$\delta^{2-p}\eta\ll \eta^{2-p}$ since $p\in(1,2)$ and $\delta,\eta$ are small.
\end{proof}

When $q=p'$, this exponent $\gamma$ is equal to $\max(\frac1{2-p},\frac{3-p}{2-p}) = \frac{3-p}{2-p}$.

\section{Second variation for the Hausdorff-Young inequality}
We now return to the special case $Tf=\widehat{f}$ with $f\in L^p(\reals^d)$,
and $q=p'$. Thus $\Phi(f) = \norm{\widehat{f}}_q/\norm{f}_p$.  Set \[G(x) = e^{-\pi|x|^2}.\]
Then $\int_{\reals^d}G=1$, $\widehat{G}\equiv G$, and \[\norm{G}_p^p = p^{-d/2}.\]

The quadratic form $\scriptq=\scriptq_G$ introduced above becomes
\begin{multline} \scriptq(h) 
= \tfrac12 (q-1) q^{d/2} \int G^{q-2}|\Re(\widehat{h})|^2 
+ \tfrac12  q^{d/2} \int G^{q-2}|\Im(\widehat{h})|^2 
\\
- \tfrac12 (p-1) p^{d/2} \int G^{p-2} |\Re(h)|^2  
- \tfrac12  p^{d/2} \int G^{p-2} |\Im(h)|^2  
\end{multline}
since $G$ is real-valued. Since $(q-1)\ge 1$,
\begin{multline} \scriptq(h) 
\le \tfrac12 (q-1) q^{d/2} \int G^{q-2}|\widehat{h}|^2 
- \tfrac12 (p-1) p^{d/2} \int G^{p-2} |\Re(h)|^2  
- \tfrac12  p^{d/2} \int G^{p-2} |\Im(h)|^2.  
\end{multline}
Note that $G^{p-2} = e^{(2-p)\pi|x|^2}$ and $2-p>0$; finiteness of $\int |h|^2 G^{p-2}$
implies finiteness of $\norm{h}_p^p$, but not vice versa. 

We have shown that if $f\in L^p(\reals^d)$ has small norm and
satisfies $\Re(\int f G^{p-1})=0$, then for any small $\eta>0$,
there exists a decomposition $f=f_\sharp+f_\flat$ such that
\[\norm{f}_p^p = \norm{f_\sharp}_p^p + \norm{f_\flat}_p^p\]
and
\begin{equation}\label{eq:finalexpansion} 
\frac{\norm{\widehat{G+f}_q}} { {\norm{G+f}_p} } \le \bestA_p^d
+ \bestA_p^d \scriptq(f_\sharp) -c\eta^{2-p} \norm{f_\flat}_p^p + C\eta \norm{f}_p^2.
\end{equation}
Since $|f_\sharp|\le\eta G$, $\int G^{p-2}|f_\sharp|^2$ is necessarily finite.
Since $G^{q-2}$ is a Schwartz function, $\int |\widehat{f_\sharp}|^2 G^{q-2}$ is also finite
by H\"older's inequality.  Thus $\scriptq(f_\sharp)$ is well-defined.

Denote the inner product for $L^2(\reals^d)$ by $\langle f,\,g\rangle = \int f\,\overline{g}$.
Denote convolution by $*$; $\varphi*\psi(x) = \int \varphi(x-y)\psi(y)\,dy$.
Define $G^t(x) = (G(x))^t = e^{-\pi t|x|^2}$. 
For any $s,t>0$,
\begin{equation} 
\widehat{G^t}(\xi)  = t^{-d/2} e^{-\pi |\xi|^2/t} = t^{-d/2} G^{1/t}(\xi)
\end{equation}
By Plancherel's identity, for any $h\in L^p$,
\begin{equation}\label{eq:removeFT} \textstyle\int |\widehat{h}|^2 G^{q-2} 
= \langle \widehat{h},\,G^{q-2}\widehat{h}\rangle
= \langle h, (G^{q-2})^\vee *h\rangle
= (q-2)^{-d/2} \langle h, G^\sigma * h\rangle\end{equation}
where 
\begin{equation} \sigma = (q-2)^{-1} = \frac{p-1}{2-p}.  \end{equation}

\begin{definition}
For $\varphi\in L^2(\reals^d)$, 
\begin{equation} \T(\varphi) = G^t\cdot \big(G^\sigma*(G^t \varphi)\big)\end{equation}
where  $\sigma = (p-1)(2-p)^{-1}$  and $t = (2-p)/2$.
\end{definition}
This operator
$\T$ is a complex-linear, compact, self-adjoint, nonnegative operator on complex-valued $L^2(\reals^d)$. 
Then
\begin{equation}
\int |\widehat{f}|^2 G^{q-2} = (q-2)^{-d/2}\langle h,\scriptt h\rangle
\end{equation}
where $h = fG^{-t} = fG^{(p-2)/2}$ for all $f$ in a suitable dense subspace of $L^p$.

\begin{lemma} \label{lemma:eigenvalues}
There is a complete set of orthogonal eigenfunctions for $\T$ indexed by $\set{0,1,2,\dots}^d$
of the form 
\begin{equation} \psi_\alpha(x)=(x^\alpha+r_\alpha(x))e^{-p\pi|x|^2/2},\end{equation}
where $r_\alpha$ is a polynomial of degree $\le |\alpha|-2$,
with corresponding eigenvalues 
\begin{equation} \lambda_\alpha = (p-1)^{|\alpha|}(2-p)^{d/2}.\end{equation}
\end{lemma}

This can be read off, via a change of variables, from properties of the Mehler kernel and Hermite semigroup; 
see Beckner \cite{beckner}, p.~163.
Since Theorem~\ref{thm:refinement1} and Theorem~\ref{thm:refinement2}
depend on precise values of eigenvalues and an identification
of the span of a certain family of eigenfunctions, and since a self-contained derivation 
is not significantly longer than the derivation from standard formulas via change of variables, we include a self-contained 
proof of Lemma~\ref{lemma:eigenvalues} in \S\ref{section:Hermite}.

\begin{definition}
The subspace $\scriptv_1\subset L^2(\reals^d)$ consists of all complex-valued functions $F\in L^2$ that satisfy
\begin{equation} \label{scriptvprimedefn}
\int F(x)\,x^\alpha G^{p/2}(x)\,dx=0\ \text{ for all $0\le|\alpha|\le 1$.} \end{equation} \end{definition}

\begin{definition}
The subspace $\scriptv_2\subset\scriptv_1$ is the closed 
subspace of $L^2(\reals^d)$ consisting of all complex-valued functions $F$ that satisfy 
\begin{equation} \label{scriptvdefn}
\int F(x)\,x^\alpha G^{p/2}(x)\,dx=0\ \text{ for all $0\le|\alpha|\le 2$.}
\end{equation}
\end{definition}

According to Lemma~\ref{lemma:eigenvalues}, both $\scriptv_1,\scriptv_2$ are invariant subspaces for $\T$.
These two subspaces are of special significance because
if $h\in\scriptv_1$ and $\Re(h)\in\scriptv_2$  then
$hG^{1-(p/2)}$ belongs to the real tangent space $\scriptn_G$ to $\frakG$ at $G$.
The converse relation holds for a dense subspace of $\scriptn_G$.

The following inequalities are direct consequences of Lemma~\ref{lemma:eigenvalues}.
\begin{corollary} \label{cor:spectralcorollary}
\begin{align} \norm{\T(F)}_{L^2(\reals^d)} &\le  (p-1)^{3} (2-p)^{d/2}\norm{F}_{L^2(\reals^d)}
\ \text{ for all $F\in\scriptv_2$.} 
\\ \norm{\T(F)}_{L^2(\reals^d)} &\le  (p-1)^{2} (2-p)^{d/2}\norm{F}_{L^2(\reals^d)}
\ \text{ for all $F\in\scriptv_1$.} \end{align}
\end{corollary}

These will be used to establish the following upper bound for the quadratic form $\scriptq$.
\begin{proposition} \label{prop:spectral} Let $d\ge 1$ and $p\in(1,2)$.
For all complex-valued functions $h\in L^p(\reals^d)$ satisfying 
\begin{equation} \label{h;scriptv} 
\left\{\begin{aligned}
& \int_{\reals^d} |h|^2 G^{p-2}<\infty
\\&\int h(x)\,x^\alpha G^{p-1}(x)\,dx=0\ \text{ for all $0\le|\alpha|\le 1$}
\\ &\int \Re(h)(x)\,x^\alpha G^{p-1}(x)\,dx=0\ \text{ for all $|\alpha| = 2$}
\end{aligned} \right. \end{equation}
one has
\begin{equation} \scriptq(h) \le -\tfrac12 (2-p)(p-1)p^{d/2} \int_{\reals^d} |h|^2 G^{p-2}.  \end{equation}
\end{proposition}

\begin{proof}[Proof of Proposition~\ref{prop:spectral}]
Let $h$ satisfy \eqref{h;scriptv}. Define
\[ F = G^{(p-2)/2}h \ \text{ and }\ t = \tfrac12 (2-p). \]
Then $F\in L^2(\reals^d)$ with $\int |h|^2 G^{p-2} = \int |F|^2$, $F\in\scriptv_1$,
and $\Re(F)\in\scriptv_2$. By \eqref{eq:removeFT},
\begin{equation*} \int |\widehat{h}|^2 G^{q-2} 
= (q-2)^{-d/2} \langle F, \T(F)\rangle 
= (p-1)^{d/2}(2-p)^{-d/2}\langle F, \T(F)\rangle \end{equation*}
where $\langle\cdot,\cdot\rangle$ denotes the Hermitian inner product.
Thus
\begin{align*}
\scriptq(h) 
&\le \tfrac12(q-1)q^{d/2} (p-1)^{d/2}(2-p)^{-d/2}\langle F,\T F\rangle
-\tfrac12(p-1)p^{d/2} \norm{\Re(F)}_2^2
-\tfrac12 p^{d/2} \norm{\Im(F)}_2^2
\\&
= \tfrac12(p-1)^{-1}(2-p)^{-d/2}p^{d/2}\langle F,\T F\rangle
-\tfrac12(p-1)p^{d/2} \norm{\Re(F)}_2^2
-\tfrac12 p^{d/2} \norm{\Im(F)}_2^2.
\end{align*}

Let $F_{\scriptv_2}$ be the orthogonal projection of $F$ onto $\scriptv_2$;
here complex-valued $L^2(\reals^d)$ is regarded as a {\em real} Hilbert space and the 
orthogonal projection is taken in this sense.
Since $F\in\scriptv_1$, since $\scriptv_2$ and $\scriptv_1\ominus\scriptv_2$ are 
invariant subspaces for $\T$, and since  $F-F_{\scriptv_2}\perp F_{\scriptv_2}$,
\[ \langle F,\T F\rangle \le (p-1)^2(2-p)^{d/2} \norm{F-F_{\scriptv_2}}_2^2
+ (p-1)^3(2-p)^{d/2} \norm{F_{\scriptv_2}}_2^2.  \]
The assumptions on $h$ imply that $\Re(F)\in\scriptv_2$, and $i\Im(F)\in\scriptv_1$, 
which contains $\scriptv_2$.  Therefore $F-F_{\scriptv_2}$ equals the 
orthogonal projection of $i\Im(F)$ onto $\scriptv_1\ominus\scriptv_2$.
Therefore  since $p-1\le 1$,
\begin{multline*} (p-1)^2(2-p)^{d/2} \norm{F-F_{\scriptv_2}}_2^2
+ (p-1)^3(2-p)^{d/2} \norm{F_{\scriptv_2}}_2^2
\\ \le (p-1)^2(2-p)^{d/2} \norm{\Im(F)}_2^2
+ (p-1)^3(2-p)^{d/2} \norm{\Re(F)}_2^2.  \end{multline*}
In all,
\begin{align*} \scriptq(h) &\le
\Big(\tfrac12(p-1)^{-1}(2-p)^{-d/2}p^{d/2} (p-1)^3(2-p)^{d/2} -\tfrac12(p-1)p^{d/2}\Big) \norm{\Re(F)}_2^2
\\& \qquad + \Big(\tfrac12(p-1)^{-1}(2-p)^{-d/2}p^{d/2} (p-1)^2(2-p)^{d/2} -\tfrac12 p^{d/2}\Big) \norm{\Im(F)}_2^2
\\& = \tfrac12 p^{d/2} (p-2) \Big( (p-1)\norm{\Re(F)}_2^2 + \norm{\Im(F)}_2^2\Big).  \end{align*}
The factor $p-2$ is negative. Since $p-1<1$, and since $\norm{F}_2^2 = \int |h|^2G^{p-2}$, we conclude that
\begin{equation} \scriptq(h) \le -\tfrac12 p^{d/2} (2-p)  (p-1)\int |F|^2
= -\tfrac12 p^{d/2} (2-p)  (p-1)\int |h|^2 G^{p-2}.  \end{equation}
\end{proof}

The condition $f\in\scriptn_G$ is not inherited by $f_\sharp$
under the decomposition $f=f_\sharp+f_\flat$. 
Consequently $G^{(p-2)/2}f_\sharp$ need not satisfy \eqref{h;scriptv}.
The next result ensures that this discrepancy is harmless.
\begin{corollary} \label{cor:ohbother}
For each $d\ge 1$ and each compact set $\Lambda\subset(1,2)$
there exists  $C<\infty$  with the following property for every exponent $p\in\Lambda$.
Let $f\in L^p(\reals^d)$ and suppose that
$f=f_\sharp+f_\flat$ with $\int |f_\sharp|^2 G^{p-2}<\infty$.
Suppose that $f$ satisfies \eqref{h;scriptv}.
Then
\begin{equation} \scriptq(f_\sharp) 
\le -\tfrac12 (2-p)(p-1)p^{d/2} \int |f_\sharp|^2 G^{p-2} +C\norm{f_\flat}_p^2.  \end{equation}
\end{corollary} 

\begin{proof}
Split $\Re(f_\sharp) G^{(2-p)/2}$ 
as a function in $\scriptv_2$ plus a function orthogonal to $\scriptv_2$;
likewise split $\Im(f_\sharp)G^{(2-p)/2}$
as a function in $\scriptv_1$ plus a function orthogonal to $\scriptv_1$.
Apply Proposition~\ref{prop:spectral} to conclude that
\begin{multline*} \scriptq(f_\sharp) \le -\tfrac12 (2-p)(p-1)p^{d/2} \int |f_\sharp|^2 G^{p-2} 
\\
+ C\sum_{|\alpha|\le 2} \big|\int \Re(f_\sharp) x^\alpha G^{p-1} \big|^2
+ C\sum_{|\alpha|\le 1} \big|\int \Im(f_\sharp) x^\alpha G^{p-1} \big|^2
\end{multline*}
where $C<\infty$ depends on $p,d$.  Now 
\begin{equation*} 
\sum_{|\alpha|\le 2} \big|\int f_\sharp x^\alpha G^{p-1} \big|^2
= \sum_{|\alpha|\le 2} \big|\int f_\flat x^\alpha G^{p-1} \big|^2
\le C \norm{f_\flat}_p^2
\end{equation*}
by H\"older's inequality since $p\le 2$; $G^{p-1}$ decays rapidly since $p-1>0$.
\end{proof}

We have shown that if $f$ satisfies \eqref{h;scriptv}
then for all sufficiently small $\eta>0$,
\begin{align*}
\Phi(G+f)
& \le \bestA_p^d -\big(\tfrac12(2-p)(p-1)p^{d/2}\bestA_p^d-O(\eta)\big) \int |f_{\sharp,\eta}|^2 G^{p-2}
-c_{d,p}\eta^{2-p}\norm{f_{\flat,\eta}}_p^p
\\&
= \bestA_p^d -\big(\bestB_{p,d}-O(\eta)\big)\norm{G}_p^{-p} \int |f_{\sharp,\eta}|^2 G^{p-2}
-c'_{d,p} \eta^{2-p}\norm{G}_p^{-p} \norm{f_{\flat,\eta}}_p^p
\end{align*}
where $c_{d,p}>0$. 

This can be rephrased in simpler terms at the expense of some loss of information:
\begin{equation} 
\Phi(G+f)
\le \bestA_p^d -\big(\lambda \bestB_{p,d}-O(\eta)\big)\norm{G}_p^{-2} \norm{f_{\sharp,\eta}}_p^2
-c'_{d,p} \eta^{2-p}\norm{G}_p^{-p} \norm{f_{\flat,\eta}}_p^p,
\end{equation}
where 
\begin{align*} \lambda
= \norm{G}_p^{2-p} \inf_{0\ne g\in L^p} \norm{g}_p^{-2}\int_{\reals^d} |g|^2 e^{-(p-2)\pi |x|^2}\,dx.
\end{align*}
Now $\lambda\ge 1$. Indeed, letting
$r= 2/(2-p)$ be the exponent conjugate to $2/p$  and invoking H\"older's inequality gives
\begin{multline*}
\norm{g}_p^p = \int |g|^p G^{p(p-2)/2} G^{-p(p-2)/2}
 \le \big(\int |g|^2 G^{p-2} \big)^{p/2} \big(\int G^{-rp(p-2)/2} \big)^{1/r}
\\ =  \big(\int |g|^2 G^{p-2} \big)^{p/2} \big(\int G^{p} \big)^{(2-p)/2}
 =  \big(\int |g|^2 G^{p-2} \big)^{p/2} \norm{G}_p^{(2-p)p/2}.
\end{multline*}
Therefore
$\norm{G}_p^{p-2} \norm{g}_p^{-2} \int |g|^2 G^{p-2} \ge 1$.

In actuality, $\lambda=1$ since $g$ can be chosen so that H\"older's inequality becomes an equality.
Information has been sacrificed in two ways; the condition $g\in\scriptn_G$ was not used in this step,
and the factor $\bestB_{p,d}$ arises only for linear combinations of a few specific eigenfunctions.

We conclude that
\begin{equation} \label{eq:longroad}
\Phi(G+f)
\le \bestA_p^d -\big(\bestB_{p,d}-O(\eta)\big)\norm{G}_p^{-2} \norm{f_{\sharp,\eta}}_p^2
-c'_{d,p} \eta^{2-p}\norm{G}_p^{-p} \norm{f_{\flat,\eta}}_p^p
\end{equation}
for all $0\ne f\in\scriptn_G$ with $\dist_p(f,\frakG)/\norm{f}_p$ sufficiently small
and all $\eta\in(0,\eta_0]$.

\section{Conclusion of proofs of main theorems}
Fix a dimension $d\ge 1$ and an exponent $p\in(1,2)$.
Let $\scriptg$ be the group of linear bijections
of $L^p(\reals^d)$ generated by  translations, modulations, multiplication
by nonzero constants, and $f\mapsto |\det(T)|^{1/p}\,f\circ T$ where $T\in\Gl(d)$. 
Each of these generators $\phi$, except multiplication by constants,
preserves $L^p$ norms, and moreover, satisfies
$\norm{\widehat{\phi f}}_{p'} = \norm{\widehat{f}}_{p'}$ for all $p\in L^p$.

By the implicit function theorem, any function $u\ne 0$ 
with $\dist_p(u,\frakG)/\norm{u}_p$ sufficiently small can be expressed
in a unique way as $u=\pi(u) + u^{\perp}$ where
$\pi(u)\in\frakG\setminus\{0\}$, and $u^\perp\in\scriptn_{\pi(u)}$.
Moreover
\begin{equation} \dist_p(u,\frakG)\le \norm{u^\perp}_p=\diststar(u,\frakG).\end{equation}
The condition $v\in\scriptn_g$ is invariant under $\scriptg$ in the sense that 
for any $\phi\in\scriptg$, 
\[ v\in\scriptn_g\Longleftrightarrow \phi(v)\in\scriptn_{\phi(g)}.\]

Let $G(x)=e^{-\pi |x|^2}$.
The group $\scriptg$ acts transitively on $\frakG$. 
Therefore it suffices to analyze functions $u\in L^p$ satisfying $\pi(u)=G$,
and $u-\pi(u)=u^\perp\in\scriptn_G$.
The condition that $u^\perp\in\scriptn_G$ is
$\Re\big(\int u^\perp \,PG^{p-1}\big)=0$ for all $P$
belonging to the set $\scriptp$ of quadratic polynomials
introduced in Definition~\ref{defn:scriptp}.
This is a restatement of \eqref{h;scriptv} with $h=u^\perp$.

\begin{proof}[Proof of Theorem~\ref{thm:refinement2}]
The statement is invariant under multiplication
by nonzero constants, so it suffices to consider functions $f$ that satisfy
$\norm{\pi(f)}_p = \norm{G}_p$ such that $\dist_p(f,\frakG)$ is sufficiently small.
The group $\scriptg$ acts transitively on $\frakG$;
there exists $\phi\in\scriptg$ such that $\phi(\pi(f)) = G$, and $\phi$ preserves $L^p$ norms.
Then $\phi(f) = G + \phi(f^\perp)$ and $\phi(f^\perp)= (\phi f )^\perp$ belongs to  $\scriptn_G$. 

We have shown that there exists $\eta_0>0$ depending only on $d,p$ such that 
if $\pi(f)=G$ and $\norm{(\phi f)^\perp}_p = \norm{f^\perp}_p$ is sufficiently small,
then for any $\eta\in(0,\eta_0]$,
\begin{equation}\begin{aligned}
\frac{\norm{\widehat{\phi f}}_q}{\norm{\phi f}_p}
&\le \bestA_p^d
-\big(\bestB_{p,d}-O(\eta)\big) \norm{\phi f}_p^{-p} \int |(\phi f)^\perp_\eta|^2 |G|^{p-2}
\\
& \qquad\qquad\qquad\qquad - c_{d,p}\eta^{2-p}\norm{\phi f}_p^{-p}\cdot\norm{(\phi f)^\perp-(\phi f)^\perp_\eta}_p^p. 
\end{aligned}\end{equation}
This is equivalent to
\begin{equation}\begin{aligned}
\frac{\norm{\widehat{f}}_q}{\norm{f}_p}
&\le \bestA_p^d
-\big(\bestB_{p,d}-O(\eta)\big) \norm{f}_p^{-p} \int | f^\perp_\eta|^2 |\pi(f)|^{p-2}
- c_{d,p}\eta^{2-p}\norm{f}_p^{-p}\norm{f^\perp- f^\perp_\eta}_p^p. 
\\&\le \bestA_p^d
-\big(\bestB_{p,d}-O(\eta)\big) \norm{f}_p^{-p} \int | f^\perp_\eta|^2 |\pi(f)|^{p-2}
\\ & \qquad\qquad\qquad\qquad - c_{d,p}\eta^{2-p}\big(\diststar(f,\frakG)/\norm{f}_p \big)^{-(2-p)}
\norm{f}_p^{-2}\norm{f^\perp- f^\perp_\eta}_p^2 
\end{aligned}\end{equation}
since $\norm{f^\perp-f^\perp_\eta}_p \le \norm{f^\perp}_p = \diststar(f,\frakG)$ and $p<2$.

Thus for any $A<\infty$ and $\eta>0$ there exists $\delta>0$
such that whenever $\norm{f}_p\ne 0$ and $\dist_p(f,\frakG)\le\delta\norm{f}_p$,
\begin{equation}
\frac{\norm{\widehat{f}}_q}{\norm{f}_p}
\le \bestA_p^d
-\big(\bestB_{p,d}-O(\eta)\big) \norm{f}_p^{-p} \int | f^\perp_\eta|^2 |\pi(f)|^{p-2}
-A\norm{f}_p^{-2}\norm{f^\perp-f^\perp_\eta}_p^2.
\end{equation}
\end{proof}

\begin{proof}[Proof of Theorem~\ref{thm:refinement1}]
By \eqref{eq:longroad} and the same reasoning as used above to
reduce the analysis of $f$ to that of $\phi(f)$,
for any $A<\infty$ and $\eta>0$ there exists $\delta>0$
such that whenever $\norm{f}_p\ne 0$ and $\dist_p(f,\frakG)\le\delta\norm{f}_p$,
\begin{equation}
\frac{\norm{\widehat{f}}_q}{\norm{f}_p}
 \le \bestA_p^d -\big(\bestB_{p,d}-O(\eta)\big)\norm{f}_p^{-2} \norm{f^\perp_\eta}_p^2
-A \norm{f}_p^{-2} \norm{f^\perp - f^\perp_{\eta}}_p^2.
\end{equation}
Choosing $A$ sufficiently large gives
\begin{equation}
\frac{\norm{\widehat{f}}_q}{\norm{f}_p}
 \le \bestA_p^d -\bestB_{p,d}\norm{f}_p^{-2}{\rm dist}^\star(f,\frakG)^2
+ \eta \norm{f}_p^{-2}{\rm dist}^\star(f,\frakG)^2.
\end{equation}
\end{proof}

\medskip
Theorem~\ref{thm:truemain} follows directly upon combining Theorem~\ref{thm:refinement1} and
Proposition~\ref{prop:submain}. \qed

\section{Sharpness of the constant $\bestB_{p,d}$}
We conclude by showing that the constant $\bestB_{p,d}$ in Theorem~\ref{thm:refinement2}
is optimal for $d=1$; the same reasoning applies in all dimensions, but the notation is more involved.
As above, consider $G(x)=e^{-\pi x^2}$.
Let $\{H_k: k=0,1,2,\dots\}$ be the family of Hermite polynomials with respect to $G^{p/2}$;
by this we mean that $H_k$ has degree $k$
and $\{H_k(x)G^{p/2}(x): k\ge 0\}$ is an orthonormal basis for $L^2(\reals^d)$.

Fix a real-valued auxiliary function $\eta\in C^\infty(\reals)$ that is supported in $[-1,1]$,
satisfies $\eta\equiv 1$ in some interval containing $0$,
and $\eta(-x)\equiv\eta(x)$ for all $x\in\reals$.
Let  $\rho>0$ satisfy $\pi\rho^2<1$.

For small $\eps>0$ consider $G+f$ where $f=f_\eps$ is defined to be
\begin{equation} f(x) = \eta(x/\rho\sqrt{\ln(1/\eps)}) \big(\eps H_3(x)+ c_\eps H_1(x)\big) G(x) \end{equation}
and the coefficient $c_\eps$ is chosen to ensure that 
$\int f(x) x^kG^{p-1}(x)\,dx=0$ for $k=0,1,2$. 
This orthogonality condition for $k=1$ uniquely specifies $c_\eps$ for all sufficiently small $\eps$.
The orthogonality then holds for $k=0,2$ since $f$ is odd while $x^k G^{p-1}$ is even. 

One finds that $c_\eps = O(\eps^{1+\delta})$ as $\eps\to 0$, for some $\delta>0$.
$f$ is supported where $|x|\le \rho\sqrt{\ln(1/\eps)}$,
so $\eps H_3(x) \le C\rho^3 \eps \ln(1/\eps)^3$
while $G(x) \ge e^{-\pi \rho^2\ln(1/\eps)} = \eps^{\pi\rho^2}$.
Thus $|\eps H_3(x)| \le C'\eps^{\gamma} G(x)$ for all $x$ in the support of $f=f_\eps$,
for all $\gamma < 1-\pi\rho^2>0$.

Set $g=g_\eps = G^{(p-2)/2} f$.
Then \[g = \eps H_3 G^{p/2} + O(\eps^{1+\delta})\] in $L^2(\reals)$.
Apply the analysis of $\Phi$ developed above, with $\eta=\eps^\gamma$ for some fixed $\gamma\in(0,1-\pi\rho^2)$.
Then $f_\sharp\equiv f$, and $f_\flat=0$.  The analysis therefore gives
\begin{align*}
\frac{\norm{\widehat{G+f}}_q}{\norm{G+f}_p}
= \bestA_p^d + \bestA_p^d \scriptq(f) + o_\eps(1)\norm{f}_p^2
= \bestA_p^d + \bestA_p^d \scriptq(f) + o(\eps^2)
\end{align*}
and
\begin{align*}
\scriptq(f) 
&= -\tfrac12(2-p)(p-1)p^{1/2}\norm{\eps H_3G^{p/2}}_{L^2(\reals)}^2 + O(\eps^{2+2\delta})  
\\&= -\tfrac12(2-p)(p-1)p^{1/2} \eps^2 + O(\eps^{2+2\delta})  
\\&= -\tfrac12(2-p)(p-1)p^{1/2} \int f^2G^{p-2} + O(\norm{f}_p^{2+2\delta})
\\&= -\tfrac12(2-p)(p-1)p^{1/2} (\norm{G}_p+O(\eps))^p \norm{f}_p^{-p} \int f^2G^{p-2} + O(\norm{f}_p^{2+2\delta})
\\&= -\tfrac12(2-p)(p-1) \norm{G+f}_p^{-p} \int f^2G^{p-2} + O(\norm{f}_p^{2+2\delta'})
\end{align*}
where $\delta'>0$.
Thus $F=G+f$ satisfies
\begin{align*} \frac{\norm{\widehat{F}}_q}{\norm{F}_p} 
&= \bestA_p^d - \bestB_{p,d} \int |f|^2 G^{p-2} + O(\norm{f}_p^{2+2\delta'})  
\\&= \bestA_p^d - \bestB_{p,d} \int |f_\sharp|^2 G^{p-2} + O(\norm{f}_p^{2+2\delta'}).  
\end{align*}
Thus the constant $\bestB_{p,d}$ is indeed optimal. \qed


\section{Spectrum of $\T$}\label{section:Hermite}
Recall that $\T$ is the compact self-adjoint linear operator on $L^2(\reals^d)$
defined by $\varphi\mapsto G^t\cdot (G^\sigma*(G^t \varphi))$
where $\sigma = (p-1)(2-p)^{-1}$.
In this section we sketch a calculation of its eigenvalues and an orthogonal basis of eigenfunctions,
establishing Lemma~\ref{lemma:eigenvalues}.

The powers $G^r$ of $G$ satisfy
\begin{equation}\label{gaussianconvolution}
G^s*G^t  = (rs^{-1}t^{-1})^{d/2} G^r\ \text{ where } r^{-1} = s^{-1}+t^{-1}\end{equation} 
and of course, $G^sG^t = G^{s+t}$.

\begin{lemma} $\T(G^{p/2}) = (2-p)^{d/2} G^{p/2}$.\end{lemma}


\begin{proof}
Let $s>0$. By the above formulas, $\T(G^s)$ is a scalar multiple of $G^{s'}$ where
$s'= ((s+t)^{-1}+\sigma^{-1})^{-1}+t$.
Since $t=\tfrac12(2-p)$ and $\sigma = (p-1)(2-p)^{-1}$, 
if $s=p/2$ then 
\[s' = \big(\,(\tfrac12 p + \tfrac12(2-p))^{-1}+(2-p)(p-1)^{-1}\big)^{-1} + \tfrac12(2-p)
= \tfrac12 p=s. \]
By \eqref{gaussianconvolution}, $\T(G^{p/2}) = \lambda_0 G^{p/2}$ with
\[\lambda_0^{2/d} = \frac{s-t}{s+t}\,\sigma^{-1}
= \frac{ \tfrac12 p-\tfrac12(2-p) } {\tfrac12 p+\tfrac12(2-p) } \cdot \frac{2-p } {p-1 } = 2-p.  \]
\end{proof}

\begin{lemma}
For any multi-index $\alpha$,
\begin{equation} \label{eq:hermiteclaim} \T(x^\alpha G^{p/2}) = Q_\alpha (x)G^{p/2} 
\ \text{ where }\ 
Q_\alpha (x)= (p-1)^{|\alpha|} x^\alpha + R_\alpha(x) \end{equation}
for some polynomial $R_\alpha$ of degree $\le |\alpha|-2$.
\end{lemma}

\begin{proof}
We will first prove, by induction on the degree of $P$, 
that for any polynomial $P$ there exists a polynomial $Q_P$ such that $\T(PG^{p/2})=Q_PG^{p/2}$. 
Since $t + \tfrac12 p = \tfrac12(2-p)+\tfrac12 p=1$,
$G^t G^{p/2}=G$ 
and consequently
$\T(PG^{p/2}) \equiv G^t G^\sigma*(PG)$.

By induction on the degree,
\begin{align*} G^\sigma*(x_k P G) 
&= G^\sigma* \big(-(2\pi)^{-1} P \partial_{x_k} G\big)
\\&= -(2\pi)^{-1} G^\sigma* \partial_{x_k} (PG)
+(2\pi)^{-1}G^\sigma * (\partial_{x_k}P)G
\\& = -(2\pi)^{-1}\partial_{x_k}(G^\sigma* (P G))
+(2\pi)^{-1}G^\sigma * (\partial_{x_k}P )G
\\& = -(2\pi)^{-1}\partial_{x_k}(Q_P \cdot (G^\sigma* G))
+(2\pi)^{-1}Q_{\partial_{x_k}P} \cdot (G^\sigma * G)
\\& = -(2\pi)^{-1} \partial_{x_k} Q_P\, \cdot\, (G^\sigma* G)
-(2\pi)^{-1} Q_P \partial_{x_k}(G^\sigma*G)
+(2\pi)^{-1}Q_{\partial_{x_k}P} \cdot(G^\sigma * G).
\end{align*}
Now
\begin{multline*} \partial_{x_k}(G^\sigma*G) 
= \partial_{x_k}(\sigma+1)^{-d/2} G^{\sigma/(\sigma+1)}
\\
= -2\pi x_k \sigma(\sigma+1)^{-1}(\sigma+1)^{-d/2} G^{\sigma/(\sigma+1)} 
 = -2\pi \sigma (\sigma+1)^{-1} x_k G^\sigma*G. \end{multline*}
Since
$\sigma+1 = \frac{p-1}{2-p}+1 = (2-p)^{-1}$ and $\frac{\sigma}{\sigma+1} = p-1$,
we obtain
\begin{equation} Q_{x_k P} = 
(p-1) x_k Q_P -(2\pi)^{-1} \partial_{x_k}Q_P + (2\pi)^{-1} Q_{\partial_{x_k}P}.  \end{equation}
The lemma \eqref{eq:hermiteclaim} follows from this recursion formula. 
\end{proof}

\begin{proof}[Proof of Lemma~\ref{lemma:eigenvalues}]
By the preceding lemma together with the Gram-Schmidt procedure, 
$\T$ has an orthonormal set of eigenfunctions of the indicated form with eigenvalues 
\begin{equation} \lambda_\alpha = 
\left(\frac{\sigma}{\sigma+1}\right)^{|\alpha|}\lambda_0 = (p-1)^{|\alpha|} (2-p)^{d/2}.\end{equation}

This set of eigenfunctions is complete in $L^2(\reals^d)$. Indeed,
these are obtained by applying the Gram-Schmidt procedure to the set of 
all monomials times the Gaussian $G^{p/2}$, with the usual ordering by degree.
Therefore these are identical to the Hermite functions, up to a change of variables and normalization.
\end{proof}

\section{Application to Young's convolution inequality}

Let $p_j\in(1,2]$ for $j=1,2,3$ satisfy $\sum_{j=1}^3 p_j^{-1}=2$.
Let $q_j$ be the exponent conjugate to $p_j$.
Then $\sum_j q_j^{-1} = 3-\sum_j p_j^{-1} = 3-2=1$.
By Plancherel's theorem and the Hausdorff-Young and H\"older inequalities,
\begin{equation*}
\big|\langle f_1*f_2,f_3\rangle\big| 
= \big|\langle \widehat{f_1}\widehat{f_2},\,\widehat{f_3}\rangle\big|
\le \int |\widehat{f_1}\widehat{f_2}\widehat{f_3}|
\le \prod_{j=1}^3 \norm{\widehat{f_j}}_{q_j}
\le \prod_{j=1}^3 \bestA_{p_j}^d \norm{{f_j}}_{p_j}
= \bestC_{\vec{p}}^d \prod_{j=1}^3 \norm{{f_j}}_{p_j}.
\end{equation*}
Suppose that each function $f_j$ has positive norm. If
equality holds in Young's inequality,
then it holds at each step of this chain of inequalities. Therefore
each $f_j$ extremizes the $L^{p_j}$ Hausdorff-Young inequality,
$|\widehat{f_j}|^{q_j}$ is a constant multiple of $|\widehat{f_i}|^{q_i}$
almost everywhere for all $i,j$, 
and $\widehat{f_1}\widehat{f_2}\widehat{f_3}$ is a constant multiple of its absolute value.
This leads to the description of the set $\vec{\frakG}$ of extremizing ordered triples
stated in \S\ref{subsect:convolution}.

In proving Corollary~\ref{cor:convolution}, we may assume that $\vec{f}=(f_1,f_2,f_3)$
satisfies $\norm{f_j}_{L^{p_j}}=1$ for each index $j=1,2,3$.
It follows that if 
$\big|\langle f_1*f_2,f_3\rangle\big| \ge(1-\delta) \bestC_{\vec{p}}^d \prod_{j=1}^3 \norm{{f_j}}_{p_j}$
then for each $j$,
\begin{equation} \norm{\widehat{f_j}}_{q_j} \ge (1-c\delta) \norm{f_j}_{p_j} \end{equation}
and consequently by Theorem~\ref{thm:truemain},
there exist Gaussians $g_j$ satisfying
\begin{equation}\norm{f_j-g_j}_{p_j} \le C\delta^{1/2}.\end{equation} 

Setting $G_j = \widehat{g_j}$, we have $\norm{G_j}_{q_j}\asymp 1$ and
\begin{equation} \big|\langle G_1G_2,G_3\rangle \big|
\ge (1-C\delta^{1/2}) \prod_{j=1}^3 \norm{G_j}_{q_j}.\end{equation}
Express $G_j(x) = c_j e^{i\xi_j\cdot x} e^{-Q_j(x-a_j)}$
where  $c_j\in\complex$, $\xi_j\in\reals^d$, 
and $Q_j$ is a positive definite real quadratic form
$Q_j(x) = \sum_{m,n=1}^d b_{m,n}x_mx_n$.
Therefore
\[ \iint_{\reals^{d}} e^{-\sum_{j=1}^3 Q_j(x-a_j)}
\ge (1-C\delta^{1/2}) \prod_{j=1}^3 \norm{e^{-Q_j(\cdot)}}_{q_j}.\]

The conclusion \eqref{eq:bad4bound} of Corollary~\ref{cor:convolution}
follows from an elementary analysis based on completion of the square in
the exponent. Details are left to the reader.

\section{A remark}
It is natural to ask what prevents adaptation of the analysis \cite{liebgaussian}
of exact extremizers to an analysis of near-extremizers. 
One unstable step relies on analytic continuation, and
appears in Step 4 of the proof of Theorem~4.5 of \cite{liebgaussian}.
If $\phi(w,z)$ is an entire function of $(w,z)\in\complex\times\complex$
whose restriction to $\reals\times\reals$ factors as
$h(\Re(w))\tilde h(\Re(z))$, then $\phi$ likewise factors on $\complex\times\complex$. 
In an adaptation of the method to near extremizers, the best one could hope to know
after Step 3 would be that
the restriction to $\reals\times\reals$ agreed with a product 
function up to an additive error with small $L^qL^p$ norm.
In such a circumstance, but in the absence of further information, 
$\phi$ need not be close to a product function on $\complex\times\complex$.
It is not clear to this author whether additional information could be gleaned 
to make this step go through.


\begin{thebibliography}{20}

\bibitem{babenko}
K.~I.~Babenko, {\em An inequality in the theory of Fourier integrals}, (Russian) 
Izv. Akad. Nauk SSSR Ser. Mat. 25 (1961), 531--542

\bibitem{beckner} 
W.~Beckner, {\em Inequalities in Fourier analysis}, Ann. of Math. (2) 102 (1975), no. 1, 159--182 

\bibitem{brascamplieb}
H.~J.~Brascamp and E.~H.~Lieb,
{\em Best constants in Young's inequality, its converse, and its generalization 
to more than three functions}, Advances in Math. 20 (1976), no. 2, 151–173



\bibitem{christmarcos}
M.~Charalambides and M.~Christ,
{\em Near--extremizers for Young's inequality for discrete groups},
preprint, math.CA arXiv:1112.3716

\bibitem{chenfrankwerth}
S.~Chen, R.~L.~Frank, and T.~Werth,
{\em Remainder terms in the fractional Sobolev inequality}, 
preprint, math.AP arXiv:1205.5666 

\bibitem{christonextremals}
M.~Christ,
{\em On extremals for a Radon-like transform}, preprint, math.CA arXiv:1106.0728

\bibitem{christyoungest} 
\bysame,
{\em Near extremizers of Young's inequality for $\reals^d$}, preprint, math.CA arXiv:1112.4875

\bibitem{christbmtwo}
\bysame, {\em Near equality in the two-dimensional Brunn-Minkowski inequality},
preprint, math.CA arXiv:1206.1965

\bibitem{christbmhigh}
\bysame, {\em Near equality in the Brunn-Minkowski inequality},
preprint, math.CA arXiv:1206.1965

\bibitem{christRS3}
{\em Near equality in the Riesz-Sobolev inequality}, arXiv:1309.5856 math.CA, submitted


\bibitem{eisnertao}
T.~Eisner and T.~Tao,
{\em Large values of the Gowers-Host-Kra seminorms}, 
J. Anal. Math. 117 (2012), 133--186

\bibitem{fournier}
J.~J.~F.~Fournier,
{\em Sharpness in Young's inequality for convolution}, Pacific J. Math. 72 (1977), no. 2, 383--397


\bibitem{liebgaussian} 
E.~Lieb,
{\em Gaussian kernels have only Gaussian maximizers}, 
Invent. Math. {\bf 102} (1990), 179--208.

\bibitem{TaoVu}
T.~Tao and V.~Vu, 
{\em Additive Combinatorics},
Cambridge Studies in Advanced Mathematics, 105. Cambridge University Press, Cambridge, 2006

\end{thebibliography}
\end{document}